\documentclass[11pt]{amsart}
\usepackage[dvipsnames]{xcolor}
\usepackage{amsfonts,amssymb,amsmath,amscd,amstext}
\usepackage[colorlinks=true,linkcolor=teal,citecolor=purple]{hyperref}
\usepackage[utf8]{inputenc}
\usepackage{graphicx}
\usepackage{changes}
\usepackage{comment}
\usepackage[noabbrev,capitalize,nameinlink]{cleveref}
\usepackage{mathdots}
\usepackage[a4paper,top=2.3cm,bottom=2.3cm,left=2cm,right=2cm]{geometry}
\renewcommand{\leq}{\leqslant}
\renewcommand{\geq}{\geqslant}
\renewcommand{\le}{\leqslant}
\renewcommand{\ge}{\geqslant}

\newcommand{\hhh}{{\mathcal{H}}}
\newcommand{\hn}{\mathbb{H}^n}

\newcommand{\normale}{N}

\newcommand{\hnn}{\ell}

\newcommand{\vh}{\nu}

\newcommand{\fisd}{g_{S,k}}
\newcommand{\fisdapp}{g_{S,k,\omega}}

\newcommand{\rr}{{\mathbb{R}}}

\newcommand{\hh}{{\mathbb{H}}}

\newcommand{\Om}{\Omega}
\newcommand{\eps}{\varepsilon}

\newcommand{\ns}{\nabla^S}

\newcommand{\Es}{\mathrm{G}}

\newcommand{\mh}{\mathcal{H}}

\newcommand{\s}{\mathcal S}
\newcommand{\scu}{\longrightarrow}

\newcommand{\x}{\mathfrak{X}}

\renewcommand{\th}{\tilde h}

\newcommand{\E}{\mathrm{F}}
\newcommand{\En}{\mathrm{E}}
\newcommand{\adel}{A(\delta)}

\definechangesauthor[color=red]{J}
\definecolor{champagne}{rgb}{0.97, 0.91, 0.81}

\definecolor{asparagus}{rgb}{0.53, 0.66, 0.42}
\DeclareMathOperator{\ric}{Ric}

\DeclareMathOperator{\divv}{div}

\DeclareMathOperator{\trace}{trace}

\DeclareMathOperator{\hess}{Hess}

\DeclareMathOperator{\supp}{supp}
\DeclareMathOperator{\spann}{span}

\DeclareMathOperator{\tor}{Tor}
\DeclareMathOperator{\torn}{Tor_{\nabla}}
\DeclareMathOperator{\tors}{Tor_{\nabla^S}}
\newtheorem{theorem}{Theorem}[section]
\newtheorem{proposition}[theorem]{Proposition}

\newtheorem{corollary}[theorem]{Corollary}

\theoremstyle{definition}
\newtheorem*{subriembp}{Sub-Riemannian stable Bernstein problem}

\newtheorem{remark}[theorem]{Remark}
\newtheorem{example}[theorem]{Example}

\makeatletter

\theoremstyle{remark}

\numberwithin{equation}{section}

\definechangesauthor[color=purple]{S}
\definechangesauthor[color=red]{P}
\definechangesauthor[color=blue]{G}

\author[G. Giovannardi]{Gianmarco Giovannardi}
\address[Gianmarco Giovannardi]{Dipartimento di Matematica e Informatica "U. Dini"\\ Università degli Studi di Firenze\\ Viale Morgani 67/A, 50134 Firenze (FI), Italy\\
\textbf{\emph{gianmarco.giovannardi@unifi.it}}}

\author[A.~Pinamonti]{Andrea Pinamonti}
\address[Andrea Pinamonti]{
Dipartimento di Matematica\\ Università di Trento\\
Via Sommarive, 14, 38123 Povo (TN), Italy\\
\textbf{\emph{andrea.pinamonti@unitn.it}} }

\author[S.~Verzellesi]{Simone Verzellesi}
\address[Simone Verzellesi]{Dipartimento di Matematica "Tullio Levi-Civita", Università degli Studi di Padova, via Trieste 63, 35131 Padova (PD), Italy\\
\textbf{\emph{simone.verzellesi@unipd.it}}}

\title[Curvature estimates for minimal hypersurfaces in the Heisenberg group]{Curvature estimates for minimal hypersurfaces in the Heisenberg group}

\date{\today}

\thanks{\textit{Memberships and funding information.} The authors are members of the Istituto Nazionale di Alta Matematica (INdAM), Gruppo Nazionale per l'Analisi Matematica, la Probabilità e le loro Applicazioni (GNAMPA).
 G. Giovannardi is supported by 
 MIUR-PRIN 2022 Project \emph{Geometric-Analytic Methods
for PDEs and Applications}. 
A. Pinamonti and G. Giovannardi are supported by MIUR-PRIN 2017 Project \emph{Gradient flows, Optimal Transport and Metric Measure Structures}.
S. Verzellesi is supported by INdAM–GNAMPA 2023 Project \emph{Equazioni differenziali alle derivate parziali di tipo misto o dipendenti da campi di vettori}. G.Giovannardi is supported by INdAM-GNAMPA 2024 Project   \emph{Problemi ellitici e sub-ellitici: non linearità, singolarità e crescita critica}.  A. Pinamonti and S. Verzellesi are supported by MIUR-PRIN 2022 Project \emph{Regularity problems in sub-Riemannian structures}  Project code: 2022F4F2LH\\
\indent \textit{Statements and Declarations.} The authors disclose any financial interest and any conflict of interest.}

\begin{document}
\begin{abstract}
This paper examines minimal hypersurfaces in sub-Riemannian Heisenberg groups. We extend the celebrated Simons formula and Kato inequality to the sub-Riemannian setting, and we apply them to obtain integral curvature estimates for stable hypersurfaces. These results lead to structural conditions that imply a Bernstein-type rigidity theorem for smooth, non-characteristic hypersurfaces in the second Heisenberg group.
\end{abstract}
\maketitle

\section{Introduction}

\begin{comment}
In the 1915 Bernstein \cite{MR1544873} showed that an entire graph $u:\rr^2 \to \rr$ in $\rr^3$ that solves 
\[
\divv \left(\dfrac{\nabla u}{\sqrt{1+|\nabla u|^2}} \right)=0
\]
must be a flat plane    
\end{comment}
In this paper, we investigate minimal hypersurfaces within sub-Riemannian Heisenberg groups. Our main contribution is the extension to the sub-Riemannian setting of the celebrated \emph{Simons formula} for the Laplacian of the second fundamental form of a minimal hypersurface, discovered by Simons in \cite{MR233295}. Morover, we provide an improved \emph{Kato inequality} for the gradient of the second fundamental form of a minimal hypersurface, extending its Riemannian version first established by Schoen, Simon and Yau in \cite{MR0423263}. Combining the previous results, we derive integral curvature bounds for stable hypersurfaces. Finally, we apply them to uncover structural conditions under which complete, stable hypersurfaces in the second Heisenberg group are flat from the sub-Riemannian viewpoint, providing a first Bernstein-type rigidity result in Heisenberg groups of higher dimension. The solution to the Euclidean Bernstein problem played a pivotal role in advancing the understanding of the geometric structure of minimal hypersurfaces. Therefore, it seems appropriate to briefly recall the historical origins of this subject.
\subsection*{The Euclidean Bernstein problem}

The \emph{Bernstein problem}, originally solved by Bernstein in $\rr^2$ (cf. \cite{MR1544873}),  consists in characterizing entire solutions $u:\rr^n \to \rr$ to the minimal surface equation
\begin{equation}
\label{eq:MCE0}
\divv \left( \dfrac{D u}{\sqrt{1+|D u|^2}}\right)=0.
\end{equation}
Thanks to the combined contributions of Fleming (cf. \cite{MR157263}), De Giorgi (cf. \cite{MR178385}), Almgren (cf. \cite{MR200816}), Simons (cf. \cite{MR233295}) and Bombieri, De Giorgi 
and Giusti (cf. \cite{MR250205}), we know that, for $n\le7$, entire solutions to \eqref{eq:MCE0} are affine functions, or, equivalently, their graphs are hyperplanes. In contrast, for $n\ge8$, there exist entire analytic solutions to \eqref{eq:MCE0} that are not affine. 
Indeed, the well-known monotonicity formula for the perimeter density allows to reduce the solution to the Bernstein problem in $\rr^{n+1}$ to the existence of singular minimal cones in $\rr^n$ (cf. \cite{MR157263,MR178385}), which occurs if and only if $n\geq 8$ (cf. \cite{MR200816,MR233295,MR250205}).
The very same approach is suitable for the solution to a way more general formulation of the Bernstein problem, i.e. the characterization of global perimeter minimizers in $\rr^n$. In this setting, when $n\leq 7$, the unique non-empty global perimeter minimizers in $\rr^n$ are half-spaces, while there are counterexamples when $n\geq 8$.
We refer to \cite{MR0775682} for a detailed account on the Bernstein problem in the Euclidean space.
\subsection*{An alternative approach} A new approach to the Bernstein problem
was proposed by Schoen, Simon and Yau in their seminal paper \cite{MR0423263}, where the authors solved the latter in the class of complete, stable hypersurfaces satisfying suitable volume growth assumptions and under the additional constraint $n\leq 6$. This second approach can be summarized in the following steps.
\begin{enumerate}
    \item  Combining the celebrated \emph{Simons identity} for minimal hypersurfaces $S\subseteq\rr^n$ (cf. \cite{MR233295}), namely
\begin{equation}\label{simonseuclideaintro}
    \Delta^Sh=-|h|^2h,
\end{equation}
with the \emph{Kato-type inequality}
\begin{equation}\label{katoeuclideaintro}
    \left(1+\frac{2}{n-1}\right)\left|\nabla^S|h|^2\right|^2\leq 4|h|^2|\nabla^S h|^2
\end{equation}
(cf. \cite{MR0423263}), one provides a lower bound for $\Delta^S|h|^2$ of the form
\begin{equation}\label{simonskatoeuclideaintro}
    2|h|^2\Delta^S|h|^2\geq  \left(1+\frac{2}{n-1}\right)\left|\nabla^S|h|^2\right|^2-4|h|^6.
\end{equation}
Here, $h$ is the second fundamental form associated to $S$, $\Delta^S$ is the tangential Laplacian and $\nabla^S$ is the tangential gradient.
\item Owing to \eqref{simonskatoeuclideaintro}, one establishes $L^p$-estimates for stable hypersurfaces such as
\begin{equation}\label{lpestimateseuclideaintro}
    \int_S|h|^p\varphi^p\,d\sigma\leq C\int_S|\nabla^S\varphi|^p \,d\sigma,
\end{equation}
where $C$ is a geometric constant, $p$ lies in a range of exponents which depends on the dimension $n$, $\sigma$ is the surface measure and $\varphi$ is a smooth test function. 
\item Assuming that $S$ satisfies volume growth conditions of the form
\begin{equation}\label{updeeuclideaintro}
    \sigma(S\cap B_r(p))=O\left(r^{n-1}\right)
\end{equation}
as $r\to\infty$, one exploits \eqref{lpestimateseuclideaintro} to show that $S$ is totally geodesic, meaning that $h\equiv 0$.
\item Complete, totally geodesic Euclidean hypersurfaces are hyperplanes.
\end{enumerate}
In particular, since for $n\leq 7$ boundaries of perimeter minimizers are smooth, complete, stable hypersurfaces and satisfy \eqref{updeeuclideaintro} (cf. \cite{MR2976521}), this new approach yields a new solution to the Bernstein problem when $n\leq 6$.
 Although this second approach fails to solve the Bernstein problem in its full generality, as the case $n=7$ is not covered, it is originally stated in the more general setting of Riemannian manifolds satisfying suitable curvature constraints. We point out that, in the Euclidean setting,  Bellettini (cf. \cite{bellettini2023extensionsschoensimonyauschoensimontheorems}) has recently extended Schoen-Simon-Yau's result up to $n=7$. This second approach has the advantage of being applicable to the solution of the so-called \emph{stable Bernstein problem}, i.e. the characterization of complete, stable hypersurfaces, thus without a priori requiring that they are boundaries of global perimeter minimizers. While the thee-dimensional version of the latter has been solved by do Carmo and Peng (cf. \cite{MR546314}), Fischer-Colbrie and Schoen (cf. \cite{MR562550}) and Pogorelov (cf. \cite{MR630142}) via \emph{ad hoc} techniques, Schoen-Simon-Yau's approach reduces the solution to the higher dimensional case to the establishment of volume growth estimates as in \eqref{updeeuclideaintro}. Following this approach, Chodosh and Li (cf. \cite{chodosh4,Chodosh2023due}) and Catino, Mastrolia and Roncoroni (cf. \cite{MR4706440}), Chodosh, Li, Minter and Stryker (cf. \cite{chodosh5}) and Mazet (cf. \cite{mazet2024stableminimalhypersurfacesmathbb}) recently solved the stable Bernstein problem in $\rr^4$, $\rr^5$ and $\rr^6$ respectively. We refer to \cite{MR4308670,Colombo20242275} for some recent developments in the Riemannian setting.
\subsection*{The sub-Riemannian Bernstein problem} Like its Euclidean and Riemannian precursors, the sub-Riemannian Bernstein problem is an intriguing topic within the broader framework of sub-Riemannian geometry and it fits into the more general context of studying minimal hypersurfaces in sub-Riemannian structures (cf. \cite{MR2165405,MR2983199,MR2262784,MR2472175,MR4346009,pmc1,pozuelo2024existence,MR4314055,SGR,MR2043961,MR3445204,MR2448649,MR3276118} and references therein). This research area is particularly relevant in the sub-Riemannian Heisenberg group $\hn$, 
which constitutes a prototypical model in the setting of Carnot groups (cf. \cite{MR2363343}), sub-Riemannian manifolds (cf. \cite{MR3971262}), CR manifolds (cf. \cite{MR2312336}) and Carnot-Carathéodory spaces (cf. \cite{MR1421823}). 
We briefly recall that the $n$-th Heisenberg group $(\hn,\cdot)$ is $\mathbb R^{2n+1}$ endowed with the group law
\begin{equation*}
\label{eq:Hproduct}
    p\cdot p'=(z,t)\cdot(z',t')=(\bar x,\bar y,t)\cdot (\bar x',\bar y',t')=\left(\bar x+\bar x',\bar y+\bar y', t+t'+\sum_{j=1}^n\left(x_j'y_j-x_jy_j'\right)\right),
\end{equation*}
where we denoted points $p\in\mathbb R^{2n+1}$ by $p=(z,t)=(\bar x,\bar y,t)=(x_1,\ldots,x_n,y_1,\ldots,y_n,t)$.
With this operation, $\hn$ is a Carnot group, whose associated \emph{horizontal distribution}, which we denote by $\hhh$, is generated by the left-invariant vector fields
\begin{equation*}
    Z_j=X_j=\frac{\partial}{\partial x_j}+y_j\frac{\partial}{\partial t}\qquad \text{and}\qquad Z_{n+j}=Y_j=\frac{\partial}{\partial x_j}-x_j\frac{\partial}{\partial t}
\end{equation*}
for $j=1,\ldots,n$. A vector field which is tangent to $\hhh$ at every point is called \emph{horizontal}. If we denote by $T$ the left-invariant vector field $\frac{\partial}{\partial t}$, then $
X_1,\ldots,X_n,Y_1,\ldots,Y_n,T
$ constitutes a global frame of left-invariant vector fields.  The only nontrivial commutation relations are
\begin{equation*}
    [X_j,Y_j]=-[Y_j,X_j]=-2T
\end{equation*}
for any $j=1,\ldots,n$. 
$\hn$ inherits a sub-Riemannian structure by fixing be the unique Riemannian metric $\langle\cdot,\cdot\rangle$ which makes $X_1,\ldots,X_n,Y_1,\ldots,Y_n,T$ orthonormal. Moreover, $\hn$ can be endowed with an appropriate affine connection $\nabla$, the so-called \emph{pseudohermitian connection}, which is metric but not torsion-free, and in a sense realizes it as a flat sub-Riemannian model. These tools both provide an intrinsic definition of perimeter, the so-called \emph{horizontal perimeter}, and enhance the study of the extrinsic geometry of submanifolds in $\hh^n$. Consequently, the sub-Riemannian formulation of the Bernstein problem appears just as natural as its Riemannian counterpart.  %
Nevertheless, its essence is substantially different from the Euclidean one.
First of all, a suitable generalization of the Euclidean monotonicity formula is known to hold only for particular classes of hypersurfaces (cf. \cite{MR2644313}), whereas its validity for general hypersurfaces remains a major open problem in the field. This obstruction prevents the possibility to follow either the classical approach (cf, \cite{MR233295}) or the more recent one developed in \cite{bellettini2023extensionsschoensimonyauschoensimontheorems}.   %
Moreover, there are several examples of minimal cones, both smooth (cf. \cite{MR2435652,MR2983199,SGR}) and with low regularity (cf. \cite{MR2448649,MR2455341}) which are not flat from the intrinsic viewpoint of $\hn$.
This new phenomenon is caused  by the fact that an hypersurface $S\subseteq \hh^n$, although smooth from a classical differential standpoint, may be intrinsically irregular due to the presence of the so-called \emph{characteristic points}, i.e. those points $p\in S$ for which the tangent space $T_pS$ coincides with the horizontal distribution $\hhh_p$.  For the above-mentioned reasons, the first approach we have described in the Euclidean setting does not appear to be suitable for this context, neither in $\hh^1$ nor in higher dimension. Nevertheless, by means of \emph{ad hoc} techniques, the Bernstein problem in $\hh^1$ is now largely solved.
Danielli, Garofalo, Nhieu and Pauls (cf. \cite{MR2648078}) and Hurtado, Ritoré and Rosales (cf. \cite{MR2609016}, and cf. also \cite{MR2165405,MR2983199,MR2435652,MR2472175}) proved that a complete, two-sided, connected, stable $C^2$-surface in $\hh^1$ must be a \emph{vertical plane} (without characteristic points), a \emph{horizontal plane} (with one characteristic point) or the \emph{hyperbolic paraboloid} $t=x_1y_1$ (with a line of characteristic points). In general, we recall that a \emph{vertical hyperplane} in $\hh^n$ is an Euclidean hyperplane which is tangent to $T$ at every point.
The hyperbolic paraboloid, which is a cone for the intrinsic geometry of $\hn$, can be easily lifted to provide smooth, non-flat, minimal cones for any $n\geq 2$ (cf. \cite{ruled}). A special class of surfaces without characteristic points, which we call \emph{non-characteristic}, is that of \emph{intrinsic graphs} as introduced by Franchi, Serapioni and Serra Cassano in \cite{MR1871966}.
In this setting, Barone Adesi, Serra Cassano and Vittone (cf. \cite{MR2333095}) had previously shown that entire, stable intrinsic graphs associated to a $C^2$-function are vertical planes. We point out that, differently from the Euclidean setting, stability plays a role even when dealing with intrinsic graphs. Indeed, Danielli, Garofalo and Nhieu exhibited striking examples of entire, minimal intrinsic graphs which are unstable (cf. \cite{MR2405158}). In the non-characteristic setting, the same conclusion was achieved by Galli and Ritoré (cf. \cite{MR3406514}) in the class of non-characteristic, complete, two-sided, connected and stable $C^1$-surfaces. The $C^1$-regularity assumption was later improved to Euclidean Lipschitz regularity by Nicolussi Golo and Serra Cassano (cf. \cite{MR3984100}) and by the first author and Ritoré (cf. \cite{GR24}).
We point out that the best possible regularity to guarantee the above rigidity in $\hh^1$ is still an open problem (cf. \cite{MR4433085} for some developments in this direction).
On the other hand, although in $\hh^1$ there are counterexamples to the regularity of perimeter minimizers, some evidences (cf. e.g. \cite{MR2774306}) suggest that it is reasonable to study the higher dimensional Bernstein problem in the smooth category.
Therefore, in light of the above results and considerations, the appropriate intrinsic formulation of the Bernstein conjecture in arbitrary dimension reads as follows.
\begin{subriembp}
   Is it true that smooth, complete, two-sided, connected, stable non-characteristic hypersurfaces $S\subseteq \hh^n$ are vertical hyperplanes? 
\end{subriembp}

While, as we have just noticed, the problem in $\hh^1$ is fairly well understood, very little can be said in the higher dimensional case.
In \cite{MR2333095}, the authors provide a negative answer to this question for $n\ge5$, 
essentially by lifting the Euclidean analytic counterexamples available in $\rr^{n+1}$ when $n\geq 8$. However, the purely Euclidean character of these counterexamples suggests that the dimensional bound $n\geq 5$ might not be optimal.
In any case, the validity of this long-standing conjecture in the remaining cases $\hh^2$, $\hh^3$ and $\hh^4$ remains a completely open problem.
Motivated by the above considerations, in this paper we develop the geometric tools underlying Schoen-Simon-Yau's approach. As a consequence, we discover some reasonable assumptions, which we now briefly describe, to solve the sub-Riemannian Bernstein's conjecture.

\subsection*{Structural assumptions}
According to the result concerned, we will rely on some of the following assumptions that we will call \eqref{taiwansplitting}, \eqref{graddialfasololungojeinuconk} and \eqref{hpintermediatraqejvraffinata}. 
On the non-characteristic part of a hypersurface $S\subseteq\hh^n$ we can define its \emph{horizontal unit normal} $\vh$ as the normalization of the projection of the Riemannian unit normal $N$ onto the horizontal distribution $\hhh$. Then the \emph{horizontal shape operator} $A$ is given by the covariant derivative of $\vh$ with respect to the pseudohermitian connection $\nabla$. Since, differently from the Riemannian setting, $A$ is not necessarily self-adjoint, its symmetrized counterpart $\tilde{A}$ can be considered, and the associated \emph{horizontal second fundamental forms} $h$ and $\tilde{h}$ can be defined (cf. \cite{MR2401420,MR2354992,MR3385193,MR2898770,MR4193432}). We say that $S$ satisfies \eqref{taiwansplitting} when $J(\vh)$, the ninety-degree rotation of $\nu$ (cf. \Cref{heissectionnnnnnnnn}), is an eigenvector for $\tilde{h}$ on the non-characteristic part of $S$. This mild assumption, which is automatically satisfied in $\hh^1$, emerges naturally in the sub-Riemannian setting, for instance in the study of \emph{umbilic hypersurfaces} as introduced in \cite{MR3794892} (cf. \Cref{eigenumbisectionnnnnn}) and of (not necessarily umbilic) minimal hypersurfaces (cf. \Cref{newnewexamplecatenoid} and \Cref{newnewhyperparauno}.).
In the non-characteristic part of $S$, the intersection between the horizontal distribution $\mh$ and the tangent bundle $TS$ generates a $(2n-1)$-dimensional sub-bundle $\hhh TS$, the \emph{horizontal tangent bundle}. In turn, the latter admits the orthogonal decomposition $\hhh TS=\spann J(\nu)\oplus \hhh'TS$, where the $(2n-2)$-dimensional sub-bundle $\hhh'TS$ is invariant under the complex structure induced by the rotation $J$. The remaining tangent direction of $S$, say $\s$, is non-horizontal and orthogonal to $\hhh TS $, whence it is given by a linear combination between $T$ and $\nu$. However, as we are in the non-characteristic part of $S$, $\s$ cannot coincide with $\vh$, so that there exists a smooth function $\alpha$, the \emph{fundamental function} of $S$, such that $\mathcal{S}=T-\alpha \vh$ belongs to $TS$. The latter appears frequently in the sub-Riemannian theory of hypersurfaces in the Heisenberg group (cf. e.g. \cite{MR2165405,MR3794892, MR3385193,MR2898770}) and can be equivalently defined by $\alpha=\tfrac{\langle N,T\rangle}{|\normale^\hh|}$, where $\normale^\hh$ is the projection of the Riemannian unit normal $N$ onto the horizontal distribution $\hhh$. For instance, it is the curvature of a length-minimizing geodesic realizing the distance between a hypersurface and a given point \cite{MR4193432}. Moreover, when $S$ is embedded in $\hh^n$, the fundamental function can be characterized by the identity $\alpha= T d^S$, where $d^S$ is the signed \emph{Carnot-Carathéodory distance} from $S$ (cf. \Cref{noncharhypersectionnnnnnnnnn}). We say that $S$ satisfies \eqref{graddialfasololungojeinuconk} if the fundamental function $\alpha$ in constant along the sub-bundle $\hhh' TS$ on the non-characteristic part of $S$. Again, since $\hhh'TS=\{0\}$ at non-characteristic points when $n=1$, \eqref{graddialfasololungojeinuconk} is satisfied by every surface in $\hh^1$. Property \eqref{graddialfasololungojeinuconk}, which is automatically verified by \emph{vertical hypersurfaces} (cf. \Cref{newnewverticalhypuno}), appears frequently in the context of minimal hypersurfaces, both non-characteristic (cf. \Cref{newnewexamplecatenoiddue}) and characteristic (cf. \Cref{newnewhyperparadue}). We stress that {\eqref{graddialfasololungojeinuconk} does not prescribe any kind of behavior of $\alpha$ along the non-horizontal direction $\s$.
In order to describe our last assumption, we recall that a smooth  hypersurface $S$ is \emph{minimal} whether its \emph{horizontal mean curvature} $H$ vanishes on the non-characteristic part of $S$, and that it is \emph{stable} if it is minimal and
\begin{equation}\label{stabilitaintro}
 \int_S q\,\xi^2  \,d\sigma_\hhh\leq\int_S|\nabla^{\hhh,S}\xi|^2 \,d\sigma_\hhh
    \end{equation}
for any smooth function $\xi$ compactly supported on the non-characteristic part of $S$. Here $\nabla^{\hhh,S}$ is the \emph{horizontal tangent gradient}, $\sigma_\hhh$ is the sub-Riemannian surface measure and $q$, the \emph{stability function}, is defined by $q=|\tilde h|^2+4\langle \nabla\alpha,J(\vh)\rangle+2(n+1)\alpha^2$ (cf. \Cref{varformsectionnnnnnn}). If compared to the Riemannian stability inequality for minimal hypersurfaces immersed in a Riemannian manifold, the stability function $q$ plays the role of the Riemannian term $|h_R|^2+ \ric(N,N)$, where $h_R$ is the Riemannian second fundamental form, $\text{Ric}$ is the Ricci curvature of the ambient manifold and $N$ is the unit normal. In the Riemannian framework, it is customary to rely on suitable lower bounds for both the Ricci curvature and the sectional curvatures in order to achieve rigidity results (cf. e.g \cite{MR562550,MR0423263,MR4519145} and references therein). Accordingly, we say that $S$ satisfies \eqref{hpintermediatraqejvraffinata} if a lower bound for the stability function $q$ holds in the form
\begin{equation}\label{riccisubintro}
    q\geq |\tilde h|^2+(2n-2)\alpha^2
\end{equation}
on the non-characteristic part of $S$. When $n=1$,\eqref{riccisubintro} is verified by any complete, non-characteristic minimal surface in $\hh^1$ (cf. \cite{MR3406514,GR24}). 
Even in higher dimension, \eqref{hpintermediatraqejvraffinata}, which is again verified by any vertical hypersurface (cf. \Cref{newnewverticalhypuno}) seems to be a natural lower bound for complete, non-characteristic minimal hypersurfaces (cf. \Cref{newnewexamplecatenoidtre} and \Cref{newnewexampleelicoid}). However, differently from \eqref{taiwansplitting} and \eqref{graddialfasololungojeinuconk}, the presence of characteristic points prevents the validity of \eqref{hpintermediatraqejvraffinata}, both in $\hh^1$ and in higher dimension (cf. \Cref{newnewhyperparatre} and \Cref{newnewhorizontalhyperplane}). 
A more detailed description of these assumptions, including the theoretical reasons to introduce them and a list of motivating examples, is given in \Cref{secsimkatnew}.
\subsection*{Simons formulas} A compelling reason to be interested in the extrinsic geometry of hypersurfaces moves from recent results proved by the last two authors of this paper (cf. \cite{ruled}). Namely, complete, non-characteristic, embedded hypersurfaces $S\subseteq\hh^n$ with vanishing symmetric horizontal second fundamental form $\tilde h$ are vertical hyperplanes. 
One of the main contributions of this paper consists in the establishment of a full sub-Riemannian counterpart of the Simons identity \eqref{simonseuclideaintro} for $\Delta^{\hhh,S} h$, where $\Delta^{\hhh,S}$ is the \emph{horizontal tangential Laplacian} of $S$, which relates the latter to the stability function $q$ appearing in \eqref{stabilitaintro} with the aid of appropriate sub-Riemannian Gauss-Codazzi equations (cf. \Cref{gcsub}). 
\begin{theorem}[Simons formula]\label{newnewfullsimonsintro}
Let $S$ be a smooth, immersed, minimal hypersurface in $\hh^n$. Then
             \begin{equation}\label{newnewfullsimonsintroformula}
         \begin{split}
              \Delta^{\hhh,S}h(X,Y)&=-qh(X,Y)+8\alpha^2h(X,Y)\\
              &\quad+4\hess^{\hhh,S}\alpha(\pi(J(X)),Y)+4\hess^{\hhh,S}\alpha(X,\pi(J(Y)))\\
              &\quad+\Big(16\alpha\pi(J(X))\alpha-8\alpha^2h(X,J(\vh))+4\left(\nabla_X\vh\right)\alpha\Big)\langle Y,J(\vh)\rangle\\
&\quad-2X\alpha h(Y,J(\vh))-2Y\alpha h( X,J(\vh))\\
              &\quad+2\alpha h(Y,\nabla_{\pi(J(X))}\vh)-2\alpha  \langle\nabla_X\nabla_{J(\vh)}\vh,Y\rangle-4\alpha^2h(\pi(J(X)),\pi(J(Y)))\\
              &\quad+2\alpha\left\langle J\left(\nabla _X\vh\right),\nabla_Y\vh\right\rangle
         \end{split}
     \end{equation}
   on the non characteristic part of $S$, for any $X,Y \in\Gamma(\hhh TS)$.
\end{theorem}
Here $\hess^{\hhh,S}\alpha$ is the \emph{horizontal tangent Hessian} of $\alpha$, while $\pi$ is the projection onto the horizontal tangent distribution to $S$ (cf. \Cref{geompropsection}).}
The sub-Riemannian Simons formula is significantly more complex than \eqref{simonseuclideaintro}, clearly highlighting the influence of the non-commutative structure in which we are operating,
due to the involvement of the fundamental function $\alpha$ and of its first and second-order horizontal derivatives.
In this regard, the control provided by \eqref{graddialfasololungojeinuconk} allows us to establish a contracted version of \eqref{newnewfullsimonsintroformula}, which guarantees an explicit formula for $\hat\Delta^{\hhh,S}|\tilde h|^2$. Here $\hat\Delta^{\hhh,S}$ is the self-adjoint counterpart of $\Delta^{\hhh,S}$ introduced by Danielli, Garofalo and Nhieu in \cite{MR2354992}. 
\begin{theorem}[Contracted Simons formula]\label{newnewcontractedsimonsstatementintro}
   Let $S$ be a smooth, immersed, minimal hypersurface in $\hh^n$. Assume \eqref{graddialfasololungojeinuconk}. 
    Then 
    \begin{equation}\label{newnewcontractedsimonsintroformula}
    \begin{split}
         \frac{1}{2}\hat\Delta^{\hhh,S}|\tilde h|^2
        & =|\nabla^{\hhh,S}\tilde h|^2-q|\tilde h|^2+6\alpha^2|\tilde h|^2-6\alpha^2\left|\nabla_{J(\vh)}\vh\right|^2+4J(\vh)\alpha\,|\tilde h|^2-4J(\vh)\alpha\,\ell^2\\
        &\quad-\left(4J(\vh)\alpha+6\alpha^2\right) \left\langle\tilde h,\tilde h_J\right\rangle,
    \end{split}
\end{equation}
on the non-characteristic part of $S$. 
\end{theorem}
Here $h_J(X,Y)=\tilde h(\pi(J(X)),\pi(J(Y))$ for any $X,Y\in\Gamma(\hhh TS)$, and $\ell=\tilde h(J(\vh),J(\vh))$.
We stress that the \emph{a priori} lack of sub-Riemannian \emph{geodesic frames} requires some delicate \emph{ad hoc} computations (cf. \Cref{adhoccomputsectionnnnnnn}). Although a result of this kind holds in arbitrary dimension, %
the specific structure of $\hh^2$ 
allows to provide a precise description of the last term appearing in \eqref{newnewcontractedsimonsintroformula} (cf. \Cref{newnewpropsuformadiscalprod}), namely
       \begin{equation}\label{newnewcontractedsimonsinh2intro}
    \begin{split}
         \frac{1}{2}\hat\Delta^{\hhh,S}|\tilde h|^2
        & =|\nabla^{\hhh,S}\tilde h|^2-q|\tilde h|^2+4\alpha^2|\tilde h|^2-8\alpha^2\left|\nabla_{J(\vh)}\vh\right|^2+2\alpha^2\ell^2\\
        &\quad+4\left(J(\vh)\alpha+\alpha^2\right)\left( 2|\tilde h|^2-4\left|\nabla_{J(\vh)}\vh\right|^2+\ell^2\right)\\
        &\quad+\left(8J(\vh)\alpha+6\alpha^2\right) \left(\left|\nabla_{J(\vh)}\vh\right|^2-\ell^2\right)
    \end{split}
\end{equation}
on the non-characteristic part of $S$. In turn, the structure of \eqref{newnewcontractedsimonsinh2intro} highlights in a natural way one of the roles played by \eqref{taiwansplitting} and \eqref{hpintermediatraqejvraffinata}, guaranteeing a lower bound for $\hat\Delta^{\hhh,S}|\tilde h|^2$ of the form       \begin{equation}\label{newnewcontractedsimonsinh2introintro}
    \begin{split}
         \frac{1}{2}\hat\Delta^{\hhh,S}|\tilde h|^2
        \geq|\nabla^{\hhh,S}\tilde h|^2-q|\tilde h|^2+\alpha^2\left(4|\tilde h|^2-6\ell^2\right)   
    \end{split}
\end{equation}
on the non-characteristic part of $S$.
We point out that the lower bound provided by \eqref{newnewcontractedsimonsinh2introintro} is sharp in the class of hypersurfaces satisfying \eqref{taiwansplitting}, \eqref{graddialfasololungojeinuconk} and \eqref{hpintermediatraqejvraffinata}, being saturated by suitable catenoidal-type hypersurfaces (cf. \Cref{newnewexamplecatenoidalquattro}).

\subsection*{Kato inequalities}
All the aforementioned contracted Simons inequality are characterized by the presence of the gradient term $|\nabla^{\hhh,S}\tilde h|$. While the latter always admits the trivial lower bound 
\begin{equation}\label{newnewkatostatement1conkintro}
             |\nabla^{\hhh,S}|\tilde h|^2|^2\leq 4|\tilde h|^2|\nabla^{\hhh,S}\tilde h|^2,
         \end{equation}
the minimality of $S$, coupled with \eqref{taiwansplitting} and \eqref{graddialfasololungojeinuconk}, allows to improve \eqref{newnewkatostatement1conkintro} to a parametric sub-Riemannian Kato inequality in the spirit of \eqref{katoeuclideaintro}.
\begin{theorem}[Improved Kato inequalities]\label{newnewkatoconkintro}
      Let $S$ be a smooth, immersed hypersurface in $\hh^n$. Then \eqref{newnewkatostatement1conkintro} holds. Assume that $S$ is minimal. Assume \eqref{taiwansplitting} and \eqref{graddialfasololungojeinuconk}. 
    Then \begin{equation}\label{newnewkatostatement2conkintro}
            \left(1+\frac{k}{2n-1}\right)|\nabla^{\hhh,S}|\tilde h|^2|^2\leq 4|\tilde h|^2|\nabla^{\hhh,S}\tilde h|^2 +4\alpha^2|\tilde h|^2\left((4k-2)|\tilde h|^2+(2+2kn-2k-4n)\hnn^2\right)
         \end{equation}
        on the non-characteristic part of $S$ for any $k\in [0,2]$.
         \end{theorem}
Even just focusing on the second Heisenberg group, the last terms appearing respectively in \eqref{newnewcontractedsimonsinh2introintro} and \eqref{newnewkatostatement2conkintro} constitute a crucial novelty compared to the Euclidean setting. Indeed, the presence of these two \emph{$\alpha^2$-remainders} is essentially due to the fact that, differently from the Euclidean setting, the lack of torsion-freeness of $\nabla$ prevents $\tilde h$ from being a \emph{Codazzi tensor} (cf. \Cref{codazzipertildehpropositionnnnnnnnnnnnnn}). Consequently, the parametric nature of the improved Kato inequality is crucial in managing to control these additional terms. While the $\alpha^2$-reminder in \eqref{newnewcontractedsimonsinh2introintro} is constrained by \eqref{hpintermediatraqejvraffinata}, the freedom to choose $k\in[0,2]$ in \eqref{newnewkatostatement2conkintro} allows to balance between the contribution of the gradient term and the $\alpha^2$-reminder. Roughly speaking, for a better control on the latter we need to pay a worse contribution from the former. A careful combination of \eqref{newnewcontractedsimonsinh2introintro} with \eqref{newnewkatostatement2conkintro} (cf. \Cref{secsimonskato}), allows to establish the sub-Riemannian counterpart of \eqref{simonskatoeuclideaintro}, whose version in $\hh^2$ reads as
\begin{equation}\label{simonskatosubintro}
       2|\tilde h|^2\hat\Delta^{\hhh,S}|\tilde h|^2\geq \left(1+\frac{k}{3}\right)|\nabla^S|\tilde h|^2|^2-4q|\tilde h|^4+4\alpha^2|\tilde h|^2\fisd.
\end{equation}
Here $\fisd$ is the contribution coming from the two $\alpha^2$-reminders of \eqref{newnewcontractedsimonsinh2introintro} and \eqref{newnewkatostatement2conkintro} (cf. \eqref{doveedefinitafis}). 
\subsection*{Integral curvature estimates} Once \eqref{simonskatosubintro} is achieved, and as soon as it is possible to choose $k\in [0,2]$ small enough to ensure that  $\fisd\geq 0$, the stability inequality \eqref{stabilitaintro} implies sub-Riemannian integral curvature estimates of the following form
\begin{theorem}[Integral curvature estimates in $\hh^2$]\label{newnewmainstabilityintrostatement}
Let $S\subseteq\hh^2$ be a smooth, connected, two-sided, embedded, stable, non-characteristic hypersurface. %
Assume \eqref{taiwansplitting},  \eqref{graddialfasololungojeinuconk} and \eqref{hpintermediatraqejvraffinata}. Assume that $\fisd(p)\geq 0$ 
        for a given $k\in [0,2]$ and for any $p\in S$. Let $\beta\in  \left[  1-\frac{k}{2n-1},1+\sqrt{\frac{k}{2n-1}} \right)$. There exists a constant $C=C(\beta,k)>0$, thus independent on $S$, such that
\begin{equation}\label{newnewmaineqstabintro}
   \int_S|\tilde h|^{2\beta+2}\varphi^{2\beta+2}\,d \sigma_\hhh\leq C\int_S|\nabla^{\hhh,S}\varphi|^{2\beta+2}\,d \sigma_\hhh
\end{equation}
for any $\varphi\in C^1_c(S)$.
    \end{theorem}
Relying on \eqref{newnewcontractedsimonsinh2introintro} and \eqref{newnewkatostatement2conkintro}, the above statement is formulated in $\hh^2$ at least for what concerns the $\alpha^2$-remainder $\fisd$. Clearly, similar conditions are available in arbitrary dimension (cf. \Cref{secsimkatnew}).

\subsection{An application to the Bernstein problem}
Heuristically, the application of \eqref{newnewmaineqstabintro} to the solution to the Bernstein problem requires $\beta$ to be chosen as close as possible to its upper bound, since the bigger is the latter, the higher is the dimension $n$ in which we can apply this approach. Indeed, under natural sub-Riemannian volume growth assumptions of the form  
\begin{equation}\label{updesubintro}
     \sigma_\hhh(S\cap B_r(p))=O\left(r^{2n+1}\right)
\end{equation}
as $r\to\infty$ (cf. \Cref{abttsect}), a suitable choice of a sequence of cut-off functions and \eqref{newnewmaineqstabintro} imply that
\begin{equation*}
     \int_{S\cap B_{r}(p)}|\tilde h|^{2\beta+2}\,d \sigma_\hhh =O\left(r^{2n-1-2\beta}\right)
\end{equation*}
as $r\to\infty$. All in all, then, we would like to choose $k$ large enough to ensure that $2n-1-2\beta\leq 0$, but still small enough to ensure that $\fisd\geq 0$. Regarding the first condition, it is easy to verify that even the optimal choice $k=2$ allows only the case $n=2$. On the other hand, when $n=2$, every choice of $k\in\left(\frac{3}{4},2\right]$ is an admissible candidate (cf. \Cref{abttsect}). 
Finally, when $n=2$ every choice of $k\in\left[0,\frac{9}{8}\right]$ ensures that $\fisd\geq 0$ (cf. \eqref{nuovosemplice}). In this way $\tilde h\equiv 0$, whence, by \cite{ruled}, $S$ is a vertical hyperplane. As a corollary of the previous results, we achieved the following Bernstein-type rigidity result
    
    \begin{theorem}\label{mainmainmainmainintro}
        Let $S\subseteq\hh^2$ be a smooth, complete, connected, embedded, two sided non-characteristic hypersurface. Assume that $S$ is stable. Assume that $S$ verifies \eqref{taiwansplitting}, \eqref{graddialfasololungojeinuconk} and \eqref{hpintermediatraqejvraffinata}. Assume in addition that there exists $p\in S$ such that \eqref{updesubintro} holds.
        Then $S$ is a vertical hyperplane.
    \end{theorem}
We point out that this approach would be pointless in $\hh^1$. Indeed, in $\hh^1$, every minimal surface satisfies $\tilde{h}\equiv0$, but there are examples of minimal non-characteristic surfaces which are not vertical planes (cf. \cite{MR2472175}). This difference between $\hh^1$ and the higher dimensional case may be explained by the fact that the horizontal tangent distribution $\hhh TS$ is bracket-generating if, and only if, $n\geq 2$ (cf. \cite{MR4103357}). A relevant instance of this phenomenon can be appreciated in the different approaches to regularity employed in $\hh^1$ \cite{MR2583494} and  in higher dimension \cite{MR2774306}. 

\subsection*{Plan of the paper} In \Cref{prelisection}, we recall some preliminaries concerning the Heisenberg group. In \Cref{geompropsection}, we collect several properties of hypersurfaces. In \Cref{adhoccomputsectionnnnnnn}, we deduce useful computational features of the symmetric form $\tilde h$. In \Cref{secsimkatnew}, we introduce \eqref{taiwansplitting}, \eqref{graddialfasololungojeinuconk}, and \eqref{hpintermediatraqejvraffinata}, and we prove \Cref{newnewfullsimonsintro}, \Cref{newnewcontractedsimonsstatementintro}, \Cref{newnewkatoconkintro}, and \eqref{simonskatosubintro}. In \Cref{stabilitisectionnnnnn}, we prove \Cref{newnewmainstabilityintrostatement}. Finally, in \Cref{abttsect}, we prove \Cref{mainmainmainmainintro}. 
In \Cref{appendice} we propose a weaker form of \eqref{hpintermediatraqejvraffinata}. 
\subsection*{Acknowledgements} The authors are grateful to Otis Chodosh, Manuel Ritoré and Francesco Serra Cassano for the many stimulating conversations on the topics discussed in this paper.

 \section{Preliminaries}\label{prelisection}
 \subsection{The Heisenberg group}\label{heissectionnnnnnnnn}
In the following, we denote by $\Gamma(T\hh^n)$ and $\Gamma(\hhh)$ the families of smooth vector fields and of smooth horizontal vector fields respectively. 
The \emph{complex structure} $J:\Gamma(T\hh^n)\longrightarrow\Gamma(T\hh^n)$ is the unique $C^\infty(\hh^n)$-linear map which satisfies
\begin{equation*}
    J(X_i)=Y_i,\qquad J(Y_i)=-X_i\qquad\text{and}\qquad J(T)=0
\end{equation*}
for any $i=1,\ldots,n$. In particular, observe that
\begin{equation}\label{jeskew}
    J(J(X))=-X\qquad\text{and}\qquad\langle X,J(X)\rangle =0
\end{equation}
for any $X\in\Gamma(\hhh)$. The latter, together with the distribution $\hhh$, realizes $\hh^n$ as a \emph{pseudohermitian manifold} (cf. \cite[Appendix]{MR2165405}). 
We recall that $\langle\cdot,\cdot\rangle$ is the unique Riemannian metric which makes $X_1,\ldots,X_n,Y_1,\ldots,Y_n,T$ orthonormal. 
Restricting it to $\hhh$, and still denoting this restriction by $\langle\cdot,\cdot\rangle$, $\hn$ inherits a \emph{sub-Riemannian structure} which realizes it as a \emph{sub-Riemannian manifold.} We denote by $|\cdot|$ the norm induced by $\langle\cdot,\cdot\rangle$.
Moreover, we denote by $\nabla$ the so-called \emph{pseudohermitian connection} (cf. e.g. \cite{MR4193432}), i.e. the unique \emph{metric connection} (cf. \cite{MR1138207}) whose torsion tensor is
\begin{equation}\label{pseudotorsion}
    \nabla_X Y-\nabla_Y X-[X,Y]=2\langle J(X),Y\rangle T
\end{equation}
for any $X,Y\in\Gamma(T\hh^n)$. Although $\nabla$ is not a torsion-free connection, it has the advantage of  vanishing along left-invariant vector fields, meaning that
\begin{equation}\label{phflat}
    \nabla_{Z_i}Z_j=0
\end{equation}
for any $i,j=1,\ldots,2n+1$ (cf. \cite{MR2898770}). In view of \eqref{phflat}, it is easy to check that
\begin{equation*}
    X\in\Gamma(T\hh^n),\,Y\in\Gamma(\hhh)\qquad\implies\qquad\nabla_X Y\in\Gamma(\hhh).
\end{equation*}
Moreover, denoting by $R$ the curvature tensor associated with $\nabla$, i.e.
 \begin{equation*}
     R(X,Y)Z=\nabla_X \nabla_Y Z-\nabla_Y\nabla_XZ-\nabla_{[X,Y]}Z
 \end{equation*}
for any $X,Y,Z\in\Gamma(T\hh^n)$, \eqref{phflat} implies that $R\equiv0$, whence $\hh^n$ is flat from the pseudohermitian standpoint.
The pseudohermitian connection is related to the complex structure by 
\begin{equation}\label{jcommutaeq}
        \nabla_XJ(Y)=J(\nabla_X Y).
    \end{equation}
for any $X,Y\in\Gamma(T\hn)$ (cf. e.g. \cite{MR2214654}).
Given a function $f\in C^\infty(\hh^n)$, we denote by 
\begin{equation*}
    \nabla f=\sum_{j=1}^{2n+1}\left(Z_jf\right)Z_j\qquad\text{and}\qquad \nabla^\hhh f=\sum_{j=1}^{2n}\left(Z_jf\right)Z_j\qquad
\end{equation*}
the gradient and the horizontal gradient associated with the pseudohermitian connection $\nabla$.
\subsection{Carnot-Carathéodory structure}
An absolutely continuous curve $\Gamma:[a,b]\scu \hn$ is \emph{horizontal} if
\begin{equation}\label{horiz}
 \Dot\Gamma(t)\in\hhh_{\Gamma(t)}
\end{equation}
for a.e. $t\in [a,b]$, and it is \emph{sub-unit} if it is horizontal and $|\Dot\Gamma(t)|\leq 1$ for a.e. $t\in[a,b]$. If we define 
\begin{equation*}
    d(p,q):=\inf\{b\,:\,\Gamma:[0,b]\scu \mathbb{H}^n\text{ is sub-unit, $\Gamma(0)=p$ and $\Gamma(b)=q$}\},
\end{equation*}
then Chow-Rashevskii theorem (cf. e.g. \cite{MR2363343}) implies that $d$ is a distance on $\mathbb{H}^n$, the so-called \emph{Carnot-Carathéodory distance}. The metric space $(\hn,d)$ is a prototype of Carnot-Carathéodory space.
For any $r>0$ and any $p\in\hh^n$, we let $B_r(p)$ be the open ball of radius $r$ centered at $p$ induced by $d$. 
\subsection{Perimeter and perimeter minimizers}  Let $\Om\subseteq\hn$ be open and $E\subseteq \hn$ measurable. The \emph{horizontal perimeter} (or \emph{$\hh$-perimeter}) of $E$ in $\Om$ (cf. e.g. \cite{MR1871966, MR1984849,MR1404326}) is defined by
\begin{equation*}
    P_{\mathbb H}(E,\Om):=\sup\left\{\int_E\divv_{\hhh}(\bar\varphi)\,d\mathcal{L}^{2n+1}\,:\,\bar\varphi\in C^1_c(\Om,\mathcal{H}),\,|\bar\varphi|_p\leq 1\text{ for any }p\in\Om\right\},
\end{equation*}
where by $C^1_c(\Om,\hhh)$ we denote the class of compactly supported horizontal vector fields defined on $\Om$, and where $\divv_{\hhh}$ is the so-called \emph{horizontal divergence}, defined by 
\begin{equation*}
    \divv_{\hhh}\left(\sum_{j=1}^n(\varphi_jX_j+\varphi_{n+j}Y_j)\right):=\sum_{j=1}^n(X_j\varphi_j+ Y_j\varphi_{n+j})
\end{equation*}
for any $\sum_{j=1}^n(\varphi_jX_j+\varphi_{n+j}Y_j)\in C^1(\Om,\hhh)$. We say that $E$ is an \emph{$\mathbb H$-Caccioppoli set} whenever $P_{\mathbb H}(E,\Om)<+\infty$ for any bounded open set $\Om\subseteq\hn$. Finally, we recall (cf. e.g. \cite{MR3587666}) that an $\mathbb H$-Caccioppoli set $E$ is an \emph{$\mathbb H$-perimeter minimizer} in $\Om$ whenever
\begin{equation*}
    P_{\mathbb H}(E,\Om)\leq P_{\mathbb H}(F,\Om)
\end{equation*}
for any $\Om \Subset\hn$ and for any $\mathbb H$-Caccioppoli set $F$ such that $E\Delta F\Subset\Om$. When $E$ is an $\mathbb H$-perimeter minimizer in $\Om=\hh^n$, we refer to it as \emph{global $\mathbb H$-perimeter minimizer}.
\section{Geometric properties of non-characteristic hypersurfaces}\label{geompropsection}
\subsection{Non-characteristic hypersurfaces}\label{noncharhypersectionnnnnnnnnn}
Let $S\subseteq\hh^n$ be a smooth, immersed hypersurface without boundary. We recall (cf. e.g. \cite{MR3587666}) that a point $p\in S$ is called \emph{characteristic} when
\begin{equation*}
    \hhh_p=T_p S,
\end{equation*}
and is called \emph{non-characteristic} otherwise. In the latter case, the \emph{horizontal tangent space}
\begin{equation*}
    \hhh T_p S=\hhh_p\cap T_p S
\end{equation*}
is a $(2n-1)$-dimensional vector space. The set of characteristic points of $S$ is denoted by $S_0$ and is called the \emph{characteristic set} of $S$. When $S_0=\emptyset$, $S$ is called \emph{non-characteristic}, and the \emph{horizontal tangent distribution} $\hhh TS$ is actually a constant-rank sub-bundle of $TS$. According to the previous notation, we denote by $\Gamma(TS)$ and by $\Gamma(\hhh TS)$ the families of smooth vector fields which are tangent to $S$ and which are horizontal and tangent to $S$ respectively. In the following, unless otherwise specified, we assume that $S$ is a smooth, immersed, non-characteristic hypersurface without boundary. When our statements are of local nature, we assume without loss of generality that $S$ is two-sided and embedded. We denote by $\normale$ its Riemannian unit normal, and by $\normale^\hh$ its projection onto $\hhh$, that is
\begin{equation*}
    \normale^\hh=\normale-\langle\normale,T\rangle T.
\end{equation*}
Being $S$ non-characteristic, then $\normale^\hh(p)\neq 0$ for any $p\in S$, so that the \emph{horizontal unit normal}
\begin{equation*}
    \vh=\frac{\normale^\hh}{|\normale^\hh|}
\end{equation*}
is well-defined on the whole $S$. Notice that the horizontal unit normal can be characterized to be the unique unitary horizontal vector field which is orthogonal to any horizontal tangent vector field. When $E$ is an $\hh$-Caccioppoli set in $\hh^n$ with boundary of class $C^1$, it is known (cf. e.g. \cite{MR2354992}) that $P_\hh(E,\cdot)=|\normale^\hh|\mathcal{H}^{2n}\llcorner \partial E$, being $\mathcal{H}^{2n}$ the standard $(2n)$-dimensional Hausdorff measure. Therefore, if $S$ is a two-sided hypersurface as above, in the following we adopt the notation
\begin{equation*}
    \sigma_\hhh=|\normale^\hh|\mathcal{H}^{2n}\llcorner S
\end{equation*}
to denote the relevant sub-Riemannian hypersurface measure as introduced e.g. in \cite{MR2354992, MR2262196}.
Let $d^S$ be the signed Carnot-Carathéodory distance from $S$. When $S$ is embedded, $d^S$ is smooth and satisfies the eikonal equation $|\nabla^\hhh d^S|=1$ in a neighborhood of $S$ (cf. \cite{MR4193432}). In this case, we assume that $\vh$ is defined in a neighborhood of $S$ by
\begin{equation}\label{normcondist}
    \vh=\nabla^\hhh d^S.
\end{equation}
When $\vh$ is locally extended as in \eqref{normcondist}, it follows that
\begin{equation}\label{propvh2}
    Z_k(\vh_h)=Z_h(\vh_k)
\end{equation}
for any $h,k=1,\ldots,2n$ such that $|h-k|\neq n$, and
\begin{equation}\label{propvh3}
    X_k(\vh_{n+k})=Y_k(\vh_k)-2\alpha\qquad\text{and}\qquad Y_k(\vh_k)=X_k(\vh_{n+k})+2\alpha
\end{equation}
for any $k=1,\ldots,n$, where here and in the following, according to \cite{MR3385193,MR3794892}, we adopt the notation $\alpha=Td^S$.
In particular, an easy computation (cf. \cite{ruled}) reveals that
\begin{equation}\label{propvh5}
    \nabla_{\vh}\vh=-2\alpha J(\vh),
\end{equation}
 Moreover, the Riemannian normal $\normale$ can be locally extended by letting 
\begin{equation}\label{riemanniannormal}
    N=\frac{1}{\sqrt{1+\alpha^2}}\vh+\frac{\alpha}{\sqrt{1+\alpha^2}}T.
\end{equation}
Let us provide a more precise description of the tangent space to $S$. First, \eqref{jeskew} implies that $J(\vh)\in\Gamma(\hhh T S)$.
Moreover, denoting by $\hhh' TS$ the distribution defined by
\begin{equation*}
    \hhh'T_p S=\hhh T_pS\cap J\left(\hhh T_p S\right)
\end{equation*}
for any $p\in S$, it is easy to check that it is a $(2n-2)$-dimensional sub-bundle of $\hhh TS$, and that the latter can be orthogonally decomposed as 
\begin{equation*}
    \hhh TS=\hhh' TS\oplus\spann J(\vh).
\end{equation*}
 Finally, \eqref{riemanniannormal} implies that the vector field $\s$ defined by 
\begin{equation*}
            \s=T-\alpha \vh
        \end{equation*}
belongs to $\Gamma(TS)$ and satisfies $\langle\s,X\rangle=0$ for any $X\in\Gamma(\hhh TS)$. Therefore, the tangent space to $S$ admits the orthogonal decomposition
\begin{equation*}
    TS=\hhh'TS\oplus\spann J(\vh)\oplus\spann \s.
\end{equation*}
 In the following, we denote by $\pi:\Gamma(\hhh)\longrightarrow\Gamma(\hhh TS)$ the projection map
 \begin{equation}\label{proiezionepersimons}
     \pi(X)=X-\langle X,\vh\rangle \vh=\sum_{i=1}^{2n-1}\langle X,E_i\rangle E_i
 \end{equation}
for any $X\in\Gamma(\hhh)$ and any local orthonormal frame $\En_1,\ldots,\En_{2n-1}$ of $\hhh TS$. In particular,
  \begin{equation}\label{pijeixinhprmo}
         \pi(J(X))\in\Gamma(\hhh' TS)  \qquad\text{and}\qquad J(\pi(J(X)))=-X+\langle X,J(\vh)\rangle J(\vh).
     \end{equation}
    Let $X\in\Gamma(\hhh TS)$. By \eqref{pseudotorsion}, 
\begin{equation}\label{inhsiiiiiiiii}
        [J(\vh),X]=\nabla_{J(\vh)}X-\nabla_XJ(\vh)\in\Gamma(\hhh TS).
    \end{equation}
  In addition,    \begin{equation}\label{commutatoreesecondaforma}
        \langle [J(\vh),X],X\rangle= \langle\nabla _X\vh,\pi(J(X))\rangle.
    \end{equation}
 The horizontal unit normal $\vh$ evolves along $\s$ as follows.
    \begin{proposition}
        Let $S$ be a smooth, immersed, non-characteristic hypersurface without boundary. Then 
        \begin{equation}\label{nablaessenu}
            \nabla_\s\vh=\nabla^\hhh\alpha+2\alpha^2J(\vh).
        \end{equation}
        In particular, if $X\in\Gamma(\hhh TS)$, then
        \begin{equation}\label{accaesse}
            \langle \nabla_{\s}\vh,X\rangle=X\alpha+2\alpha^2\langle J(\vh),X\rangle.
        \end{equation}
    \end{proposition}
    \begin{proof}
      We assume without loss of generality that $S$ is embedded. Let $\vh$ be extended as in \eqref{normcondist}. Then, recalling \eqref{propvh5},
      \begin{equation*}
              \nabla_\s\vh=\nabla_T\vh-\alpha\nabla_{\vh}\vh=\sum_{j=1}^{2n}T(Z_j d)Z_j+2\alpha^2 J(\vh)=\sum_{j=1}^{2n}Z_j\alpha Z_j+2\alpha^2 J(\vh)=\nabla^\hhh\alpha+2\alpha^2J(\vh).
      \end{equation*}
       Finally, \eqref{accaesse} easily follows.
    \end{proof}
\subsection{Tangent pseudohermitian connection}\label{sectanpseucon}
    The \emph{tangent pseudohermitian connection}, denoted by $\nabla^S:\Gamma(TS)\times\Gamma(\hhh TS)\longrightarrow\Gamma(\hhh TS)$, is defined by
\begin{equation*}
    \nabla^S _X Y=\nabla_XY-\langle\nabla_XY,\vh\rangle\vh
\end{equation*}
for any $X\in\Gamma(TS)$ and any $Y\in\Gamma(\hhh TS)$. An easy computation reveals that $\nabla^S$ is a well-defined affine connection, and that it is metric in the sense that   \begin{equation}\label{metricequationnablas}
         X\langle Y,Z\rangle=\langle\nabla^S_XY,Z\rangle+\langle Y,\nabla^S_XZ\rangle
     \end{equation}
     for any $X\in\Gamma(TS)$ and any $Y,Z\in\Gamma(\hhh TS)$.
 Accordingly, the torsion tensor $\tor_{\nabla^S}(X,Y):\Gamma(\hhh TS)\times\Gamma(\hhh TS)\longrightarrow \Gamma(TS)$ is defined by
\begin{equation*}
     \tors (X,Y)=\nabla^S_XY-\nabla^S_YX-[X,Y].
\end{equation*}
We stress that we are not requiring $ \tors (X,Y)$ to be horizontal, so that, by Frobenious theorem, $\tors$ is well-defined. The latter admits the following explicit expression.
    \begin{proposition}\label{torsions}
    Let $X,Y\in\Gamma(\hhh TS)$. Then
    \begin{equation}\label{torsionofnablas}
    \tors (X,Y)=2\langle J(X),Y\rangle \s.
\end{equation}
 \end{proposition}
\begin{proof}
    Let $X,Y\in\Gamma(\hhh TS)$. 
      If $X=\sum_{j=1}^{2n}X^jZ_j$ and $Y=\sum_{j=1}^{2n}Y^jZ_j$, then
\begin{equation*}
    \begin{split}
       -\langle[X,Y],\vh\rangle
       &\overset{\eqref{pseudotorsion}}{=}2\langle J(X),Y\rangle\langle\vh, T\rangle+\langle\nabla _X\vh,Y\rangle-\langle\nabla _Y\vh,X\rangle\\
       &\overset{\eqref{phflat}}{=}\sum_{i,j=1}^{2n}X^iY^j(Z_i\vh_j-Z_j\vh_i)\\
       &\overset{\eqref{propvh2},\eqref{propvh3}}{=}-2\alpha \sum_{i=1}^nX^iY^{n+i}+2\alpha \sum_{i=1}^nX^{n+i}Y^i\\
       &=-2\alpha \langle J(X),Y)\rangle.
    \end{split}
\end{equation*}    
    In particular, $
            \tors (X,Y)=\torn (X,Y)-\langle\torn(X,Y),\vh\rangle\vh-\langle[X,Y],\vh\rangle\vh=2\langle J(X),Y\rangle \s$.
\end{proof}

If $f\in C^\infty(S)$, we denote by 
\begin{equation*}
    \nabla ^{\hhh,S}f=\sum_{j=1}^{2n-1}\left(\mathrm{E}_jf\right)\mathrm{E}_j
\end{equation*}
the horizontal tangential gradient associated with the connection $\nabla^S$, where $\mathrm{E}_1,\ldots,\mathrm{E}_{2n-1}$ is any local orthonormal frame of $\hhh TS$. More generally, if $p\in\mathbb N$ and $T$ is a horizontal $(p,0)$-tensor field (cf. \cite{MR1138207}), meaning that $T:\Gamma(\hhh TS)^p\longrightarrow C^\infty(S)$ is a $C^\infty(S)$-multilinear map, the $(p+1,0)$ tensor field $\nabla^S T:\Gamma(TS)\times \Gamma(\hhh TS)^p\longrightarrow C^\infty(S)$ is defined by
\begin{equation*}
  \nabla^S_X T(X_1,\ldots,X_p)=X\left( T(X_1,\ldots,X_p)\right)-T\left(\nabla^{S}_XX_1,\ldots,X_p\right)-\ldots-T\left(X_1,\ldots,\nabla^{S}_X X_p\right)
\end{equation*}
for any $X\in\Gamma(TS)$ and any $X_1,\ldots,X_p\in\Gamma(\hhh TS)$. According to the above notation, we denote by $\nabla^{\hhh,S}T$ the restriction of $\nabla^ST$ to $\Gamma(\hhh TS)^{p+1}$. As a general fact, $\nabla^S$ verifies the Leibniz-type rule
      \begin{equation}\label{leibnizpertensori}
          \nabla^S_X(TU)=U\nabla^S_XT+T\nabla^S_XU
      \end{equation}
for any $X\in\Gamma(TS)$ and any couple of tensor fields $T,U$. If  $T,U$ are as above, we set 
\begin{equation*}
    \langle T,U\rangle =\sum_{i_1,\ldots,i_p=1}^{2n-1}T\left(\En_{i_1},\ldots,\En_{i_p}\right)U\left(\En_{i_1},\ldots,\En_{i_p}\right)\qquad\text{and}\qquad|T|^2=\sum_{i_1,\ldots,i_p=1}^{2n-1}T\left(\En_{i_1},\ldots,\En_{i_p}\right)^2
\end{equation*}
for any local orthonormal frame $\mathrm{E}_1,\ldots,\mathrm{E}_{2n-1}$ of $\hhh TS$.
The \emph{horizontal tangential Hessian} of $T$ is the $(p+2,0)$-tensor field $\hess^{\hhh,S}T:\Gamma(\hhh TS)^{p+2}\longrightarrow C^\infty(S)$ defined by
\begin{equation*}
    \hess^{\hhh,S} T(X,Y,X_1,\ldots,X_p)=\nabla^S_X\nabla^S_Y T (X_1,\ldots,X_p)
\end{equation*}
for any $X,Y,X_1,\ldots,X_p\in\Gamma(\hhh TS)$. Finally, the \emph{horizontal tangential Laplacian} of $T$ is the $(p,0)$-tensor field $\Delta^{\hhh,S}T:\Gamma(\hhh TS)^{p}\longrightarrow C^\infty(S)$ defined by
\begin{equation*}
    \Delta^{\hhh,S} T(X_1,\ldots,X_p)=\trace \hess^{\hhh,S}T(\cdot,\cdot,X_1,\ldots,X_p)=\sum_{j=1}^{2n-1}\hess^{\hhh,S}T(\mathrm{E}_j,\mathrm{E}_j,X_1,\ldots,X_p)
\end{equation*}
for any $X_1,\ldots,X_p\in\Gamma(\hhh TS)$ and any local orthonormal frame $\mathrm{E}_1,\ldots,\mathrm{E}_{2n-1}$ of $\hhh TS$. We denote by $R^S:\Gamma(\hhh TS)\times\Gamma(\hhh TS)\times\Gamma(\hhh TS)\longrightarrow \Gamma(\hhh TS)$ the curvature tensor associated with $\nabla^S$, that is
 \begin{equation*}
     R^{S}(X,Y)Z=\nabla^S_X \nabla_Y^S Z-\nabla^S_Y\nabla^S_XZ-\nabla^S_{[X,Y]}Z
 \end{equation*}
 for any $X,Y,Z\in\Gamma(\hhh TS)$,
 and with some abuse of notation we let
  \begin{equation*}
     R^S(X,Y,Z,W)=\langle\nabla^S_X \nabla_Y^S Z-\nabla^S_Y\nabla^S_XZ-\nabla^S_{[X,Y]}Z,W\rangle
 \end{equation*}
 for any $X,Y,Z,W\in\Gamma(\hhh TS)$.
The horizontal Hessian is affected by $R^S$ as follows.
  \begin{proposition}
      Let $T:\Gamma(\hhh TS)\times\Gamma(\hhh TS)\longrightarrow C^\infty(S)$ be a $(2,0)$-tensor field. 
     Then
     \begin{equation}\label{commutatensoreconriemannpersimons}
     \begin{split}
                 \hess^{\hhh,S} T(Y,X,Z,W)&=\hess^{\hhh,S} T(X,Y,Z,W)\\
                 &\quad+T(R^S(X,Y)Z,W)+T(Z,R^S(X,Y)W)\\
                 &\quad+2\langle J(X),Y\rangle\left(\nabla^S_\s T\right)(Z,W)
     \end{split}
     \end{equation}
     for any $X,Y,Z,W\in\Gamma(\hhh TS)$.
 \end{proposition}
 \begin{proof}
 Notice that 
     \begin{equation*}
     \begin{split}
         \hess^{\hhh,S} T(X,Y,Z,W)
         &=XY(T(Z,W))-X(T(\nabla^S_YZ,W))-X(T(Z,\nabla^S_Y W))\\
         &-\nabla^S_XY(T(Z,W))+T(\nabla^S_{\nabla^S_XY}Z,W)+T(Z,\nabla^S_{\nabla^S_XY}W)\\
         &-Y(T(\nabla^S_XZ,W))+T(\nabla^S_Y\nabla^S_XZ,W)+T(\nabla^S_XZ,\nabla^S_YW)\\
         &-Y(T(Z,\nabla^S_XW))+T(\nabla^S_YZ,\nabla^S_XW)+T(Z,\nabla^S_Y\nabla^S_XW)
     \end{split}
     \end{equation*}
     and, in the same way,
          \begin{equation*}
     \begin{split}
         \hess^{\hhh,S} T(Y,X,Z,W)       
         &=YX(T(Z,W))-Y(T(\nabla^S_XZ,W))-Y(T(Z,\nabla^S_X W))\\
         &-\nabla^S_YX(T(Z,W))+T(\nabla^S_{\nabla^S_YX}Z,W)+T(Z,\nabla^S_{\nabla^S_YX}W)\\
         &-X(T(\nabla^S_YZ,W))+T(\nabla^S_X\nabla^S_YZ,W)+T(\nabla^S_YZ,\nabla^S_XW)\\
         &
      -X(T(Z,\nabla^S_YW))+T(\nabla^S_XZ,\nabla^S_YW)+T(Z,\nabla^S_X\nabla^S_YW).
     \end{split}
     \end{equation*}
          Hence
          \begin{equation*}
              \begin{split}
                  \hess^{\hhh,S}& T(Y,X,Z,W)-\hess^{\hhh,S} T(X,Y,W,Z)=\tor_\nabla^S(X,Y)(T(Z,W))\\
                  &\quad+T(\nabla^S_X\nabla^S_YZ-\nabla^S_Y\nabla^S_XZ-\nabla^S_{\nabla^S_XY-\nabla^S_YX}Z,W)\\
                  &\quad+T(Z,\nabla^S_X\nabla^S_YW-\nabla^S_Y\nabla^S_XW-\nabla^S_{\nabla^S_XY-\nabla^S_YX}W)\\
                 &=\tor_{\nabla^S}(X,Y)(T(Z,W))-T(\nabla^S_{\tor_{\nabla^S}(X,Y)}Z,W)-T(Z,\nabla^S_{\tor_{\nabla^S}(X,Y)}W)\\
            &\quad+T(R^S(X,Y)Z,W)+T(Z,R^S(X,Y)W)\\
              &=\nabla^S_{\tors(X,Y)}T(Z,W)+T(R^S(X,Y)Z,W)+T(Z,R^S(X,Y)W)\\
              &\overset{\eqref{torsionofnablas}}{=}2\langle J(X),Y\rangle\nabla^S_{\s}T(Z,W)+T(R^S(X,Y)Z,W)+T(Z,R^S(X,Y)W).\\
              \end{split}
          \end{equation*}
 \end{proof}

\subsection{Second fundamental forms and mean curvature}
In the current literature, different types of second fundamental form are available in the sub-Riemannian Heisenberg group (cf. e.g. \cite{MR2401420,MR2354992,MR3385193,MR2898770,MR4193432}). The \emph{horizontal shape operator} $A:\Gamma(TS)\longrightarrow \Gamma(\hhh TS)$ and the \emph{symmetric horizontal shape operator} $\tilde A:\Gamma(\hhh T S)\longrightarrow\Gamma(\hhh TS)$ are defined respectively by
 \begin{equation*}
     A(X)=\nabla_X\vh\qquad\text{and}\qquad  \tilde A(X)=\nabla_X\vh+\alpha J'(X)
 \end{equation*}
 for any $X\in\Gamma(\hhh TS)$, where $J'= J\text{ on }\hhh'TS$ and $J'(J(\vh))=0$.
It is easy to check that $A$ and $\tilde A$ are well-defined. Accordingly, the \emph{horizontal second fundamental form} $h$ and the \emph{symmetric horizontal second fundamental form} $\tilde h$ are the horizontal $(2,0)$-tensor fields defined by
\begin{equation*}
    h(X,Y)=\langle A(X),Y\rangle\qquad\text{and}\qquad \tilde h(X,Y)=\langle \tilde A(X),Y\rangle
\end{equation*}
for any $X,Y\in\Gamma(\hhh TS)$.
As in the Riemannian setting, the horizontal second fundamental form $h$ relates the connections $\nabla$ and $\nabla^S$ by the identity
\begin{equation*}
    \nabla_XY
    =\nabla^S_XY-h(X,Y)\vh
\end{equation*}
for any $X,Y\in\Gamma(\hhh TS)$. It is well known that $\tilde h$ is symmetric, while, when $n\geq 2$, $h$ may not be symmetric (cf. e.g. \cite{MR2354992,MR2898770}). More precisely, $h$ and $\tilde h$ are related in the following way.
\begin{proposition}\label{commutareh}
    Let $X,Y\in \Gamma(\hhh T S)$. Then
    \begin{equation}\label{theh}
    \th (X,Y)=h(X,Y)+\alpha C(X,Y)=\frac{h(X,Y)+h(Y,X)}{2},
\end{equation}
 where $C:\Gamma(\hhh T S)\times\Gamma (\hhh TS)\longrightarrow C^\infty(S)$, which we will refer to as \emph{commutation tensor}, is the skew-symmetric horizontal $(2,0)$-tensor field defined by
    \begin{equation*}
        C(X,Y)=\langle J(X),Y\rangle.
    \end{equation*}
    In particular
    \begin{equation}\label{comecommutah}
        h(Y,X)=h(X,Y)+2\alpha C(X,Y).
    \end{equation}
    Finally,
    \begin{equation}\label{normatildeetildeh}
        |h|^2=|\tilde h|^2+2(n-1)\alpha^2.
    \end{equation}
\end{proposition}
\begin{proof}
Let $X,Y\in \Gamma(\hhh T S)$. By definition of $h$ and $\tilde h$ we have that $\th (X,Y)=h(X,Y)+\alpha C(X,Y).$
Moreover, since $\tilde h$ is symmetric and $C$ is skew-symmetric in view of \eqref{jeskew}, then
\begin{equation*}
     \th (X,Y)=\th (Y,X)=h(Y,X)+\alpha C(Y,X)=h(Y,X)-\alpha C(X,Y),
\end{equation*}
whence \eqref{theh} and \eqref{comecommutah} follow. Finally, \eqref{normatildeetildeh} follows from \cite[Proposition 5.4]{ruled}.
\end{proof}

According to its Riemannian counterpart, the \emph{horizontal mean curvature} $H$ is then defined by
\begin{equation}\label{hvarieespressioni}
    H=\trace h=\trace \tilde h=\divv_\hhh\vh,
\end{equation}
the last identity following from \cite{MR2354992}. In the following, we say that $S$ is \emph{minimal} whenever $H\equiv 0$.
\begin{remark}
   Although we have defined the horizontal shape operator $A$ only for horizontal tangent vector fields, we emphasize that $A(\s)$ is nevertheless well-defined and admits an explicit expression by virtue of \eqref{nablaessenu}. Therefore, we shall use the notation $h(\s, X)$ for any given $X \in \Gamma(\hhh TS)$.
\end{remark}
\subsection{Eigenvectors}\label{eigenumbisectionnnnnn}
Fix $p\in S$. Since $\tilde h_p$ is symmetric, then it is diagonalizable. Therefore, in the following, we denote by $\E_1,\ldots, \E_{2n-1}$ any local orthonormal frame of $\hhh T S$ around $p$ such that
\begin{equation}\label{diagontildeh}
        \tilde h_p (\E_i|_p,\E_j|_p)=\lambda_i\delta_{i,j}
    \end{equation}
    for any $i,j=1,\ldots,2n-1$, where $\lambda_1,\ldots,\lambda_{2n-1}$ are the eigenvalues of $\tilde h_p$. When $S$ satisfies milder assumptions, more can be said about $\E_1,\ldots,\E_{2n-1}$. According to \cite{MR3385193,MR3794892}, in the following we adopt the notation $\hnn=h(J(\vh),J(\vh))$, and we let $\x=\nabla_{J(\vh)}\vh-\hnn J(\vh)$.
As we know from \cite{MR3794892}, if $p\in S$, then $\x|_p=0$ if and only if $\tilde A|_p(\hhh'T_p S)\subseteq \hhh'T_p S$. In particular, when
\begin{equation}\label{taiwansplitting}\tag{$\mathrm {P} 1$}
    \x\equiv 0,
\end{equation}
 $J(\vh)|_p$ is an eigenvector of $\tilde h_p$ for any $p\in S$. Therefore, when \eqref{taiwansplitting} holds and $\E_1,\ldots,\E_{2n-1}$ is as in \eqref{diagontildeh}, we may always assume that
\begin{equation}\label{diagotildehtaiwan}
    \E_1,\ldots,\E_{n-1},\ldots,\E_{n+1},\ldots,\E_{2n-1}\in\Gamma(\hhh 'TS)\qquad\text{and}\qquad\E_n=J(\vh).
\end{equation}
A relevant class of hypersurfaces which satisfy \eqref{taiwansplitting} is that of \emph{horizontally umbilic hypersurfaces} introduced in \cite{MR3794892}. 
In the context of non-characteristic minimal hypersurfaces, \eqref{taiwansplitting} holds, e.g., for the following relevant examples.
\begin{example}[Vertical hyperplanes]\label{newnewverticalhyperex}
    Vertical hyperplanes trivially satisfy \eqref{taiwansplitting}, as $\nabla_{J(\vh)}\vh\equiv 0$.
\end{example}
\begin{example}[Catenoidal hypersurfaces, $\mathrm{I}$]\label{newnewexamplecatenoid}
    A non-trivial example (cf. \cite{MR2271950}) is provided by the class of hypersurfaces of revolution around the vertical axis, say $(S_E)_{E>0}$, whose profile is given by the vertical symmetrization of the curve $(s,t_E(s))$, where 
    $$t_E:\left[E^{\frac{1}{2n-1}},\infty\right]\longrightarrow [0,\infty],\qquad t_E(s)=\int_{E^{\frac{1}{2n-1}}}^s\frac{E\tau}{\sqrt{\tau^{4n-2}-E^2}}\,d\tau.$$
    Each $S_E$ is complete and minimal. Moreover, it is non-characteristic, as it does not intersect the vertical axis (cf. \cite{MR2271950}). Since hypersurfaces of revolution are horizontally umbilic (cf. \cite{MR3794892}), $S_E$ satisfies \eqref{taiwansplitting}. In this case, differently from \Cref{newnewverticalhyperex}, a tedious computation shows that
    \begin{equation}\label{newnewellpercatenoid}
        \nabla_{J(\vh)}\vh=-2E|z|^{-4}J(\vh),
    \end{equation}
    where we have set $|z|^2=x_1^2+\cdots+x_n^2+y_1^2+\cdots+y_n^2.$
\end{example}
On the other hand, we have the following important example in the class of minimal, characteristic hypersurfaces.
\begin{example}[Hyperbolic paraboloids, $\mathrm{I}$]\label{newnewhyperparauno}
    Let $S\subseteq\hh^n$ be the vertical graph of the function $$u(x_1,\ldots,x_n,,y_1,\ldots,y_n)=x_1y_1+x_2y_2+\cdots+x_ny_n.$$
      $S$ is a complete, minimal hypersurface with $S_0=\{(z,t)\in\hh^n\,:\,x_1=\cdots=x_n=0\}$. A direct computation reveals that $\nabla_{J(\vh)}\vh\equiv 0$ on $S\setminus S_0$, whence $S$ verifies \eqref{taiwansplitting}. We stress that $S$ is not horizontally umbilic, since its characteristic points are not isolated (cf. \cite[Proposition 4.1]{MR3794892}). 
   
\end{example}
\subsection{Tangential Laplace-Beltrami operators}
The general approach described in \Cref{sectanpseucon} allows to associate to a function $f\in C^\infty(S)$ a natural notion of tangential Laplacian, namely
\begin{equation}\label{hortanlapconcon}
    \Delta^{\hhh,S}f=\trace\hess^{\hhh, S}f.
\end{equation}
On the other hand, the authors  of \cite{MR2354992} considered a Laplace-Beltrami operator on $S$ of the form
\begin{equation}\label{hortanlapl}
\sum_{i=1}^{2n}\nabla^{\hh,S}_i\nabla^{\hh,S}_i f,
\end{equation}
where the \emph{horizontal tangential derivatives} $
    \nabla^{\hh,S}_if=Z_i f-\langle\nabla^\hhh f,\vh\rangle\vh_i$
for any $i=1,\ldots,2n$ do not depend on the smooth extension of $f$ (cf. \cite{MR2354992}).
As pointed out in \cite{MR2354992}, the operator defined in \eqref{hortanlapl}, differently from the Riemannian framework, is not in general self-adjoint. To this aim, the authors of \cite{MR2354992} introduced a \emph{modified} version of \eqref{hortanlapl}, the so-called \emph{modified horizontal tangential Laplacian}
\begin{equation}\label{modohoritanlaplaperintro}
    \hat{\Delta}^{\hhh,S} f=\sum_{i=1}^{2n}\nabla^{\hh,S}_i\nabla^{\hh,S}_i f+2\alpha\left\langle \nabla^\hhh f,J(\vh)\right\rangle.
\end{equation}
The most relevant feature of $\hat{\Delta}^{\hhh,S} $ is that it is indeed self-adjoint (cf. \cite[Corollary 11.4]{MR2354992}), so that the following integration-by-parts formula holds. 
\begin{proposition}
  Let $\varphi\in C^1_c(S)$ and $\psi\in C^2(S)$. Then
\begin{equation}\label{ibpformula}
        \int_S\varphi\,\hat\Delta^{\hhh,S}\psi\,d \sigma_\hhh=-\int_S\langle \nabla^{\hhh,S}\varphi,\nabla^{\hhh, S}\psi\rangle\,d \sigma_\hhh
    \end{equation}
\end{proposition}
In order to exploit \eqref{ibpformula}, in the next proposition we show that \eqref{hortanlapconcon} and
 \eqref{hortanlapl} agree.
 \begin{proposition}\label{lapbelt}
Let $f\in C^2(S)$. Set  $g_\hh^{i,j}=\delta_{i,j}-\vh_i\vh_j$ for any $i,j=1,\ldots,2n$. Then
    \begin{equation*}
        \Delta^{\hhh,S} f=\sum_{i=1}^{2n}\nabla^{\hh,S}_i\nabla^{\hh,S}_i f=\sum_{i,j=1}^{2n}g_\hh^{i,j}Z_iZ_j f-H\langle\nabla^\hhh f,\vh\rangle,
    \end{equation*}
\end{proposition}
\begin{proof}
Recalling that $|\vh|=1$,
    \begin{equation*}
        \begin{split}
    \sum_{i=1}^{2n}\nabla^{\hh,S}_i\nabla^{\hh,S}_i f
            &=\sum_{i=1}^{2n}Z_iZ_i f-\langle \nabla^\hhh f,\vh\rangle\sum_{i=1}^{2n}Z_i\vh_i-\sum_{i,j=1}^{2n}Z_iZ_j f\vh_i\vh_j-\sum_{i,j=1}^{2n}Z_jfZ_i\vh_j\vh_i\\
            &\quad-\sum_{i}^{2n}\langle\nabla^\hhh Z_i f,\vh\rangle\vh_i+\sum_{i=1}^{2n}\langle\nabla^\hhh(\langle\nabla^\hhh f,\vh\rangle\vh_i),\vh\rangle\vh_i\\
&\overset{\eqref{hvarieespressioni}}{=}\sum_{i,j=1}^{2n}g_\hh^{i,j}Z_iZ_j f-H\langle \nabla^\hhh f,\vh\rangle-\sum_{i,j=1}^{2n}Z_jfZ_i\vh_j\vh_i-\sum_{i,j=1}^{2n}Z_jZ_if\vh_i\vh_j\\
&\quad+\sum_{i,j,k=1}^{2n}Z_jZ_kf(\vh_i)^2\vh_j\vh_k+\sum_{i,j,k=1}^{2n}Z_kfZ_j\vh_k(\vh_i)^2\vh_j+\sum_{i,j,k=1}^{2n}Z_kf\vh_kZ_j\vh_i\vh_i\vh_j\\
&=\sum_{i,j=1}^{2n}g_\hh^{i,j}Z_iZ_j f-H \langle \nabla^\hhh f,\vh\rangle-\sum_{i,j=1}^{2n}Z_jfZ_i\vh_j\vh_i-\sum_{i,j=1}^{2n}Z_jZ_if\vh_i\vh_j\\
&\quad+\sum_{j,k=1}^{2n}Z_jZ_kf\vh_j\vh_k+\sum_{j,k=1}^{2n}Z_kfZ_j\vh_k\vh_j\\
&=\sum_{i,j=1}^{2n}g_\hh^{i,j}Z_iZ_j f-H \langle \nabla^\hhh f,\vh\rangle.
        \end{split}
    \end{equation*}
  Let now $\En_1,\ldots,\En_{2n-1}$ be a local orthonormal frame of $\hhh TS$. Set $a_i^j=\langle\En_i,Z_j\rangle$ for any $i=1,\ldots,2n-1$ and any $j=1,\ldots,2n$.
        Then
        \begin{equation}\label{coeffequations}
            \sum_{k=1}^{2n}a_i^ka_j^k=\delta_{ij},\qquad\sum_{k=1}^{2n}a_i^k\vh_k=0\qquad\text{and}\qquad\sum_{k=1}^{2n-1}a_k^la_k^m=g_\hh^{l,m}
        \end{equation}
        for any $i,j=1,\ldots,2n-1$ and any $l,m=1,\ldots,2n$. We conclude that
        \begin{equation*}
            \begin{split}
                \Delta^{\hhh,S} f
                &=\sum_{j=1}^{2n-1}E_j(E_j f)-\sum_{j=1}^{2n-1}\nabla^S_{E_j}E_j f\\
                &=\sum_{j=1}^{2n-1}\sum_{h,k=1}^{2n}a_j^h a_j^kZ_hZ_kf+\sum_{j=1}^{2n-1}\sum_{h,k=1}^{2n}a_j^hZ_h(a_j^k)Z_kf\\
                &\quad-\sum_{j=1}^{2n-1}\langle\nabla_{E_j}E_j,\nabla f\rangle+\sum_{j=1}^{2n-1}\langle\nabla_{E_j}E_j,\vh\rangle \langle\nabla f,\vh\rangle\\
                 &\overset{\eqref{coeffequations}}{=}\sum_{h,k=1}^{2n}g_\hh^{h,k}Z_hZ_kf+\sum_{j=1}^{2n-1}\sum_{h,k=1}^{2n}a_j^hZ_h(a_j^k)Z_kf\\
                &\quad-\sum_{j=1}^{2n-1}\sum_{h,k=1}^{2n}a_j^h\langle\nabla_{Z_h}a_j^kZ_k,\nabla f\rangle-H \langle\nabla f,\vh\rangle\\
                    &=\sum_{h,k=1}^{2n}g_\hh^{h,k}Z_hZ_kf-H \langle\nabla f,\vh\rangle,
            \end{split}
        \end{equation*}
        whence the thesis follows. 
\end{proof}
\subsection{The commutation tensor}
We know from \Cref{commutareh} how the commutation tensor $C$ intervenes in the lack of commutativity of $h$. Next we discuss how it affects the commutation of the covariant derivative of $h$. First, $C$ evolves along tangent vector fields as follows.
\begin{proposition}\label{Csimmetric}
    Let $X\in\Gamma(TS)$ and let $Y,Z\in\Gamma(\hhh TS)$. Then
    \begin{equation}\label{commutatortensorpersimons}
        \left(\nabla^S_X C\right)(Y,Z)=C(Z,\vh)h(X,Y)-C(Y,\vh)h(X,Z).
    \end{equation}
\end{proposition}
\begin{proof}
Let $X,Y,Z$ be as in the statement. Then, by \eqref{jcommutaeq},
\begin{equation*}
    \begin{split}
        \nabla^S_XC(Y,Z)
        &=X\langle J(Y),Z \rangle +\langle \nabla^S_X Y,J(Z) \rangle -\langle J(Y),\nabla^S_XZ\rangle\\
        &=\langle \nabla_X J(Y),Z\rangle+\langle J(Y),\nabla_X Z\rangle+\langle \nabla_X Y,J(Z))-\langle J(Y),\nabla_XZ\rangle\\
        &\quad-\langle \nabla_X Y,\vh\rangle\langle J(Z),\vh\rangle+\langle \nabla_XZ,\vh\rangle\langle J(Y),\vh\rangle\\
        &=C(Z,\vh)h(X,Y)-C(Y,\vh)h(X,Z).
    \end{split}
\end{equation*}
\end{proof}
 Thanks to \Cref{commutareh} and \Cref{Csimmetric}, we describe the lack of commutativity of $\nabla ^Sh$.
 \begin{proposition}\label{codazzifake}
     Let $X,Y,Z\in\Gamma(\hhh TS)$. Then
     \begin{equation}\label{simmetryandnabla}
         \nabla^S_Xh(Y,Z)=\nabla^S_Xh(Z,Y)+2(X\alpha) C(Z,Y)+2\alpha C(Y,\vh)h(X,Z)-2\alpha C(Z,\vh)h(X,Y).
     \end{equation}
 \end{proposition}
 \begin{proof}
     Fix $X,Y,Z$ as in the statement. Then, by \Cref{commutareh} and \Cref{Csimmetric},
     \begin{equation*}
         \begin{split}
             \nabla^S_Xh(Y,Z)&=X(h(Y,Z))-h(\nabla^S_XY,Z)-h(Y,\nabla^S_XZ)\\
             &=X(h(Z,Y))+2X\alpha C(Z,Y)+2\alpha X(C(Z,Y))\\
             &\quad-h(Z,\nabla^S_XY)-2\alpha C(Z,\nabla_X^SY)-h(\nabla^S_XZ,Y)-2\alpha C(\nabla^S_XZ,Y)\\
             &=\nabla^S_Xh(Z,Y)+2X\alpha C(Z,Y)+2\alpha\nabla^S_X C(Z,Y)\\
             &=\nabla^S_Xh(Z,Y)+2X\alpha C(Z,Y)+2\alpha C(Y,\vh)h(X,Z)-2\alpha C(Z,\vh)h(X,Y).
         \end{split}
     \end{equation*}
 \end{proof} 
 
 \subsection{Gauss-Codazzi equations}
With the next result we derive the sub-Riemannian counterpart of the classical Gauss-Codazzi equations. We refer to \cite{santos} for a proof, which we include anyway for the sake of completeness.
 \begin{proposition}[Gauss-Codazzi equations]\label{gcsub}
     Let $X,Y,Z,W\in\Gamma(\hhh TS)$. Then the \emph{Gauss equation}
\begin{equation}\label{gaussssssub}
    R^S(X,Y,Z,W)=h(Y,Z)h(X,W)-h(X,Z)h(Y,W)
\end{equation}
and the \emph{Codazzi equation}
\begin{equation}\label{codazzipersimons}
    \left(\nabla^S_Yh\right)(X,Z)=\left(\nabla^S_Xh\right)(Y,Z)+2C(X,Y) h(\s,Z)
\end{equation}
hold.
 \end{proposition}
 \begin{proof}
We know that $   \nabla_XY=\nabla^S_XY-h(X,Y)\vh.$
Hence
\begin{equation*}
\begin{split}
    \nabla_X\nabla_YZ&=\nabla_X\nabla_Y^SZ-\nabla_X(h(Y,Z)\vh)=\nabla^S_X\nabla^S_YZ-h(X,\nabla^S_YZ)\vh-X(h(Y,Z))\vh-h(Y,Z)A(X).
\end{split}
\end{equation*}
Similarly,
\begin{equation*}
    -\nabla_Y\nabla_XZ=-\nabla^S_Y\nabla^S_XZ+h(Y,\nabla^S_XZ)\vh+Y(h(X,Z))\vh+h(X,Z)A(Y)
\end{equation*}
Moreover,
\begin{equation*}
    -\nabla_{[X,Y]}Z=-\nabla^S_{[X,Y]}Z+h([X,Y],Z)\vh
\end{equation*}
Summing the three equations term by term we get that
\begin{equation*}
\begin{split}
     0&\overset{R\equiv 0}{=}R^S(X,Y)Z+h(X,Z)A(Y)-h(Y,Z)A(X)\\
     &\quad+\left(Y(h(X,Z))-h(\nabla_Y^SX,Z)-h(X,\nabla^S_YZ)\right)\vh\\
     &\quad\left(-X(h(Y,Z))+h(   \nabla^S_XY,Z)+h(Y,\nabla^S_XZ)\right)\vh\\
     &\quad+h(\nabla_Y^SX-\nabla^S_XY+[X,Y],Z)\vh\\
     &=R^S(X,Y)Z+h(X,Z)A(Y)-h(Y,Z)A(X)\\
     &\quad+\nabla^S_Yh(X,Z)\vh-\nabla^S_Xh(Y,Z)\vh-h(\tor_{\nabla^S} (X,Y),Z)\vh\\
     &=R^S(X,Y)Z+h(X,Z)A(Y)-h(Y,Z)A(X)\\
     &\quad+\nabla^S_Yh(X,Z)\vh-\nabla^S_Xh(Y,Z)\vh-2\langle J(X),Y\rangle h(S,Z)\vh.
\end{split}
\end{equation*}
The thesis follows projecting the previous identity either on $W$ or on $\vh$.
 \end{proof}
In the following, we shall also need the following Codazzi equation for the symmetric form $\tilde h$.
\begin{proposition}[Codazzi equation for $\th$]\label{codazzipertildehpropositionnnnnnnnnnnnnn}
    Let $X,Y,Z\in\Gamma(\hhh TS)$. Then
    \begin{equation*}
    \begin{split}
         \ns_Y \th &(X,Z)-\ns_X\th(Y,Z)=2(Z\alpha)C(X,Y)+(Y\alpha)C(X,Z)-(X\alpha)C(Y,Z)\\
             &\quad +2\alpha^2 C(\vh,Z)C(X,Y)+\alpha C(\vh,X)h(Y,Z)-\alpha C(\vh,Y)h(X,Z). 
    \end{split}
    \end{equation*}
\end{proposition}
\begin{proof}
    In view of \Cref{commutareh}, \Cref{Csimmetric} and \Cref{gcsub},
    \begin{equation*}
        \begin{split}
             \ns_Y \th &(X,Z)-\ns_X\th(Y,Z)=\ns_Y h (X,Z)-\ns_Xh(Y,Z)+\ns_Y(\alpha C(X,Z))-\ns_X(\alpha C(Y,Z))\\
             &=2C(X,Y)h(\s,Z)+(Y\alpha)C(X,Z)-(X\alpha)C(Y,Z)\\
             &\quad +\alpha\ns_Y C(X,Z)-\alpha\ns_XC(Y,Z)\\
             &\overset{\eqref{accaesse}}{=}2(Z\alpha)C(X,Y)+4\alpha^2C(X,Y)C(\vh,Z) +(Y\alpha)C(X,Z)-(X\alpha)C(Y,Z)\\
             &\quad +\alpha C(Z,\vh)h(Y,X)-\alpha C(X,\vh)h(Y,Z)\\
             &\quad -\alpha C(Z,\vh)h(X,Y)+\alpha C(Y,\vh)h(X,Z)\\
             &=2(Z\alpha)C(X,Y)+4\alpha^2C(X,Y)C(\vh,Z)+(Y\alpha)C(X,Z)-(X\alpha)C(Y,Z)\\
             &\quad +2\alpha^2 C(Z,\vh)C(X,Y)-\alpha C(X,\vh)h(Y,Z)+\alpha C(Y,\vh)h(X,Z)\\
             &=2(Z\alpha)C(X,Y)+(Y\alpha)C(X,Z)-(X\alpha)C(Y,Z)\\
             &\quad +2\alpha^2 C(\vh,Z)C(X,Y)+\alpha C(\vh,X)h(Y,Z)-\alpha C(\vh,Y)h(X,Z).            
        \end{split}
    \end{equation*}
\end{proof}

\begin{corollary}
        Let $X,Y\in\Gamma(\hhh TS)$. Then
        \begin{equation}\label{codazzihtildeeqaba}
        \begin{split}
             \nabla^S_X\tilde h(Y,Y)&=\nabla^S_Y\tilde h(X,Y)+3Y\alpha C(Y,X)+3\alpha^2\langle J(\vh),Y\rangle C(Y,X)\\
             &\quad+\alpha\langle J(\vh),Y\rangle \tilde h(Y,X)-\alpha\langle J(\vh),X\rangle \tilde h(Y,Y).
        \end{split}
        \end{equation}
\end{corollary}
 \subsection{Variation formulas}\label{varformsectionnnnnnn}
 Let $S\subseteq\hn$ be a smooth, embedded, non-characteristic hypersurface without boundary. Assume in addition that $S$ is two-sided. Let $\Om\subseteq\hh^n$ be an open bounded set such that $\Om\cap S\neq\emptyset$, and let $\xi\in C^1_c(\Om)$. Then it is known (cf. \cite{MR3319952,MR4316814,santos,tesisimone})
 that
   \begin{equation}\label{firstvariation}
        \frac{d}{dt}\sigma_{\hhh,t}(\Om)\Big\vert_{t=0}=\int_SH\xi\,d\sigma_\hhh
    \end{equation}
    and
    \begin{equation}\label{secondvariation}
        \frac{d^2}{dt^2}\sigma_{\hhh,t}(\Om)\Big\vert_{t=0}=\int_S\left(|\nabla^{\hhh,S}\xi|^2-\xi^2\left(q-H^2\right)\right)\,d\sigma_\hhh,
    \end{equation}
    where
    \begin{equation}\label{q}
        q=\sum_{h,k=1}^{2n}Z_h(\vh_k)Z_k(\vh_h)+4\langle \nabla\alpha,J(\vh)\rangle+4n\alpha^2
    \end{equation}
and where by $\sigma_{\hhh,t}$ we denote the horizontal surface measure associated with the smooth variation $E_t$ along the vector field $\xi\vh$.
Observe that $q$ does not depend on the chosen unitary extension of $\vh$ (cf. \cite{ruled}). Moreover, in view of \cite[Proposition 5.1]{ruled} and \eqref{normatildeetildeh}, 
\begin{equation}\label{variq}
    \begin{split}
        q&=|h|^2+4\langle \nabla\alpha,J(\vh)\rangle+4\alpha^2=|\tilde h|^2+4\langle \nabla\alpha,J(\vh)\rangle+2(n+1)\alpha^2.
    \end{split}
\end{equation}
As customary (cf. e.g. \cite{MR3406514}) we say that $S$ is \emph{area stationary} whenever the quantity in \eqref{firstvariation} vanishes for any $\Om$ and $\xi$ as above, and that $S$ is \emph{stable} if it is area stationary and the quantity in \eqref{secondvariation} is non-negative for any $\Om$ and $\xi$ as above. Notice that $S$ is minimal if and only if it is area stationary. In particular, when $S$ is stable, the \emph{stability inequality}
\begin{equation}\label{stabilityinequality}
        \int_Sq\xi^2\,d\sigma_\hhh\leq\int_S|\nabla^{\hhh, S}\xi|^2\,d\sigma_\hhh
    \end{equation}
 holds for any $\xi\in C^1_c(S)$.
\section{Further properties of the second fundamental forms}\label{adhoccomputsectionnnnnnn}
 In this section we establish some properties of $h$ and $\tilde h$ which will be useful in the next section. %
 In order to facilitate a more conscious reading we would recommend the reader to skip directly to \Cref{secsimkatnew}, and if necessary to go back to this section in accordance with the references to the latter. Through this section we assume $S$ is smooth, immersed, non-characteristic hypersurface.  
\begin{proposition}\label{katofirst}
   Let $\E_1,\ldots,\E_{2n+1}$ be as in \eqref{diagontildeh}.
     Then 
\begin{equation}\label{tildehinejejeiejei}
    \tilde h_p(\E_j,\pi(J(\E_j)))=0,
\end{equation}
so that
\begin{equation}\label{hnonsimejjej}
    h_p(\E_j,\pi (J(\E_j)))=-\alpha+\alpha\langle J(\vh),\E_j\rangle^2\qquad\text{and}\qquad h_p(\pi(J(\E_j)),\E_j)=\alpha-\alpha\langle J(\vh),\E_j\rangle^2.
\end{equation}
    In addition
    \begin{equation}\label{comesetildehmetrica}
        \tilde h_p\left(\nabla^S_{X}\E_i,\E_i\right)=0
    \end{equation}
    for any $i=1,\ldots,2n-1$
    and any $X\in\Gamma(TS)$. Moreover, 
    \begin{equation}\label{comesetildehmetricaboh}
        \sum_{i=1}^{2n-1}Y\left(\tilde h\left(\nabla^S_{X}\E_i,\E_i\right)\right)(p)=0
    \end{equation}
    for any $X,Y\in\Gamma(TS)$. Finally,
\begin{equation}\label{passaggioinproofdihessianaetracca}
             \sum_{i=1}^{2n-1}\left(\tilde h_p(\nabla^S_X\nabla^S_Y\E_i,\E_i)+\tilde h_p(\nabla^S_Y\E_i,\nabla^S_X\E_i)\right)=0
         \end{equation}
          for any $X,Y\in\Gamma(TS)$.
\end{proposition}
\begin{proof}
To prove \eqref{tildehinejejeiejei}, notice that
\begin{equation*}
    \tilde h(\E_j,\pi(J(\E_j)))=\sum_{k=1}^{2n-1}\langle \pi(J(\E_j)),\E_k\rangle\tilde h(\E_j,\E_k)=\langle J(\E_j),\E_j\rangle\tilde h(\E_j,\E_j)=0, 
\end{equation*}
while \eqref{hnonsimejjej} follows from \eqref{pijeixinhprmo} and \eqref{theh}.
    Fix $i=1,\ldots,2n-1$. Then
    \begin{equation*}
        \begin{split}
            \tilde h\left(\nabla^S_{X}\E_i,\E_i\right)
            =\sum_{k=1}^{2n-1}\left\langle\nabla^S_{X}\E_i,\E_k\right\rangle \th\left(\E_k,\E_i\right)=\lambda_i\left\langle\nabla^S_{X}\E_i,\E_i\right\rangle =0,
        \end{split}
    \end{equation*}
  whence \eqref{comesetildehmetrica} follows.  Moreover, \eqref{comesetildehmetricaboh} follows since
    \begin{equation*}
        \begin{split}
              \sum_{i=1}^{2n-1}Y\Big(\tilde h\Big(\nabla^S_{X}&\E_i,\E_i\Big)\Big)(p)=\sum_{i,k=1}^{2n-1}Y\left(\left\langle\nabla^S_{X}\E_i,\E_k\right\rangle \th\left(\E_k,\E_i\right)\right)(p)\\
              &=\sum_{i,k=1}^{2n-1}\tilde h_p(\E_i,\E_k) Y\left(\left\langle\nabla^S_{X}\E_i,\E_k\right\rangle \right)(p)+\sum_{i,k=1}^{2n-1}\left\langle\nabla^S_{X}\E_i,\E_k\right\rangle Y(\tilde h(\E_i,\E_k))(p) \\
              &\overset{\eqref{diagontildeh}}{=}\sum_{i=1}^{2n-1}\lambda_i Y\left(\left\langle\nabla^S_{X}\E_i,\E_i\right\rangle \right)(p)+\sum_{i,k=1}^{2n-1}\left\langle\nabla^S_{X}\E_i,\E_k\right\rangle Y(\tilde h(\E_i,\E_k))(p) \\
              &\overset{\eqref{metricequationnablas}}{=}\sum_{i,k=1}^{2n-1}\left\langle\nabla^S_{X}\E_i,\E_k\right\rangle Y(\tilde h(\E_i,\E_k))(p) \\
                  &\overset{\eqref{metricequationnablas}}{=}-\sum_{i,k=1}^{2n-1}\left\langle\nabla^S_{X}\E_i,\E_k\right\rangle Y(\tilde h(\E_i,\E_k))(p)\\
                 &=0,
        \end{split}
    \end{equation*}
    the semi-last equality following by the symmetry of $\tilde h$. Finally,  
         \begin{equation*}
             \begin{split}
                  \sum_{i=1}^{2n-1}\Big(\tilde h(\nabla^S_X&\nabla^S_Y\E_i,\E_i)+\tilde h(\nabla^S_Y\E_i,\nabla^S_X\E_i)\Big)\\
                  &=\sum_{i,j=1}^{2n-1}\langle \nabla^S_X\nabla^S_Y\E_i,\E_j\rangle\tilde h (\E_i,\E_j)+\sum_{i,j,k=1}^{2n-1}\langle\nabla^S_Y\E_i,\E_j\rangle\langle\nabla^S_X\E_i,\E_k\rangle\tilde h(\E_j,\E_k)\\
          &\overset{\eqref{diagontildeh}}{=}\sum_{i=1}^{2n-1}\lambda_i\langle \nabla^S_X\nabla^S_Y\E_i,\E_i\rangle+\sum_{i,j=1}^{2n-1}\lambda_j\langle\nabla^S_Y\E_i,\E_j\rangle\langle\nabla^S_X\E_i,\E_j\rangle\\
          &\overset{\eqref{metricequationnablas}}{=}-\sum_{i=1}^{2n-1}\lambda_i\langle \nabla^S_Y\E_i,\nabla^S_X\E_i\rangle+\sum_{i,j=1}^{2n-1}\lambda_j\langle\nabla^S_Y\E_i,\E_j\rangle\langle\nabla^S_X\E_i,\E_j\rangle  \\
          &=\sum_{i,j=1}^{2n-1}\langle\nabla^S_Y\E_i,\E_j\rangle\langle\nabla^S_X\E_i,\E_j\rangle (\lambda_j-\lambda_i).
             \end{split}
         \end{equation*}
         From one hand, exchanging the indices in the previous equation, we get that
         \begin{equation}\label{iejinproofuno}
             \sum_{i,j=1}^{2n-1}\langle\nabla^S_Y\E_i,\E_j\rangle\langle\nabla^S_X\E_i,\E_j\rangle (\lambda_j-\lambda_i)=\sum_{i,j=1}^{2n-1}\langle\nabla^S_Y\E_j,\E_i\rangle\langle\nabla^S_X\E_j,\E_i\rangle (\lambda_i-\lambda_j).
         \end{equation}
         From the other hand, recalling \eqref{metricequationnablas}, we infer that
         \begin{equation}\label{iejinproofdue}
              \sum_{i,j=1}^{2n-1}\langle\nabla^S_Y\E_i,\E_j\rangle\langle\nabla^S_X\E_i,\E_j\rangle (\lambda_j-\lambda_i)=\sum_{i,j=1}^{2n-1}\langle\nabla^S_Y\E_j,\E_i\rangle\langle\nabla^S_X\E_j,\E_i\rangle (\lambda_j-\lambda_i).
         \end{equation}
           Therefore, combining \eqref{iejinproofuno} and \eqref{iejinproofdue}, \eqref{passaggioinproofdihessianaetracca} follows.
\end{proof}
\begin{proposition}\label{prophessianaetraccia}
     Let $X,Y\in\Gamma(\hhh TS)$. Then
     \begin{equation}\label{nablahtrace}
         \trace \nabla^S_X h(\cdot,\cdot)= XH
     \end{equation}
     and
     \begin{equation}\label{hesshtrace}
         \trace \hess^{\hhh,S}h(X,Y,\cdot,\cdot)=\hess^{\hhh,S}H (X,Y).
     \end{equation}
     \end{proposition}
     \begin{proof}
         Fix $X,Y$ as in the statement. Let $p\in S$. Being the trace operator independent of the choice of the orthonormal basis, we let  Let $\E_1,\ldots,\E_{2n+1}$ be as in \eqref{diagontildeh}.
         To prove \eqref{nablahtrace}, we observe that
         \begin{equation*}
             \trace \nabla^S_X h(\cdot,\cdot)=\sum_{i=1}^{2n-1}\nabla^S_Xh(\E_i,\E_i)=XH-2\sum_{i=1}^{2n-1}\tilde h(\nabla^S_X\E_i,\E_i)\overset{\eqref{comesetildehmetrica}}{=}XH.
         \end{equation*}         
         Notice that
\begin{equation}\label{hessianadihinproof}
             \hess^{\hhh,S}H(X,Y)=\nabla^S_X\nabla^S_Y H=\nabla^S_X(YH)=XYH-\nabla^S_X Y H.
         \end{equation}
         On the other hand, exploiting \Cref{katofirst},
               \begin{equation*}
             \begin{aligned}
             \trace \hess
                & =\sum_{i=1}^{2n-1}\nabla^S_X\nabla^S_Y h(\E_i,\E_i)\\
                 &=\sum_{i=1}^{2n-1}\nabla^S_X\left(Yh(\E_i,\E_i)-2\tilde h(\nabla^S_Y\E_i,\E_i)\right)\\
                 &=\sum_{i=1}^{2n-1}\Big(XYh(\E_i,\E_i)-\nabla^S_X Y h(\E_i,\E_i)-2Y\tilde h(\nabla^S_X\E_i,\E_i)\\
                 &\quad-2X\tilde h(\nabla^S_Y\E_i,\E_i)+2\tilde h(\nabla^S_{\nabla^S_X Y}\E_i,\E_i)+2\tilde h(\nabla^S_Y\nabla^S_X\E_i,\E_i)+2\tilde h(\nabla^S_Y\E_i,\nabla^S_X\E_i)\Big)\\
                 &\overset{\eqref{comesetildehmetrica}}{=}\sum_{i=1}^{2n-1}\Big(XYh(\E_i,\E_i)-\nabla^S_X Y h(\E_i,\E_i)-2Y\tilde h(\nabla^S_X\E_i,\E_i)\\
                 &\quad-2X\tilde h(\nabla^S_Y\E_i,\E_i)+2\tilde h(\nabla^S_Y\nabla^S_X\E_i,\E_i)+2\tilde h(\nabla^S_Y\E_i,\nabla^S_X\E_i)\Big)\\
                 &\overset{\eqref{comesetildehmetricaboh}}{=}\sum_{i=1}^{2n-1}\Big(XYh(\E_i,\E_i)-\nabla^S_X Y h(\E_i,\E_i)+2\tilde h(\nabla^S_Y\nabla^S_X\E_i,\E_i)+2\tilde h(\nabla^S_Y\E_i,\nabla^S_X\E_i)\Big)\\
                 &\overset{\eqref{passaggioinproofdihessianaetracca}}{=}\sum_{i=1}^{2n-1}\Big(XYh(\E_i,\E_i)-\nabla^S_X Y h(\E_i,\E_i)\Big)\\
                 &\overset{\eqref{hessianadihinproof}}{=}\hess^{\hhh,S}H(X,Y).
             \end{aligned}
         \end{equation*}
     \end{proof}
     
     \begin{proposition}\label{lapnormelapnonnorm}
       It holds that        \begin{equation}\label{primaconsdisimonsform}
             \frac{1}{2}\Delta^{\hhh,S}|\tilde h|^2=|\nabla^{\hhh,S}\tilde h|^2+\left\langle\tilde h,\Delta^{\hhh,S}\tilde h\right\rangle,
         \end{equation}
         where $\left\langle \cdot, \cdot \right\rangle$ is defined in \cref{sectanpseucon}.
         \begin{proof}
          Take $\E_1,\ldots,\E_{2n-1}$ as in \eqref{diagontildeh}.
             \begin{equation*}
             \begin{split}
                    |\nabla^{\hhh,S}\tilde h|^2
                    &=\sum_{i,j,k=1}^{2n-1}\left(\nabla^S_{\E_i}\tilde h(\E_j,\E_k)\right)^2\\
                 &=\sum_{i,j,k=1}^{2n-1}\left(\E_i\tilde h(\E_j,\E_k)-\tilde h(\nabla^S_{\E_i}\E_j,\E_k)-\tilde h(\E_j,\nabla^S_{\E_i}\E_k)\right)^2\\
                 &=\sum_{i,j,k=1}^{2n-1}\left(\E_i\tilde h(\E_j,\E_k)\right)^2+2\sum_{i,j,k=1}^{2n-1}\left(\tilde h(\nabla^S_{\E_i}\E_j,\E_k)\right)^2\\
                 &\quad-4\sum_{i,j,k=1}^{2n-1}\E_i\tilde h(\E_j,\E_k)\tilde h(\nabla^S_{\E_i}\E_j,\E_k)+2\sum_{i,j,k=1}^{2n-1}\tilde h(\nabla^S_{\E_i}\E_j,\E_k)\tilde h(\E_j,\nabla^S_{\E_i}\E_k)\\
                 &=\sum_{i,j,k=1}^{2n-1}\left(\E_i\tilde h(\E_j,\E_k)\right)^2+2\sum_{i,j,k=1}^{2n-1}\langle\nabla^S_{\E_i}\E_j,\E_k\rangle^2\tilde h(\E_k,\E_k)^2\\
                 &\quad-4\sum_{i,j,k=1}^{2n-1}\tilde h(\E_k,\E_k)\E_i\tilde h(\E_j,\E_k)\langle\nabla^S_{\E_i}\E_j,\E_k\rangle\\
                 &\quad+2\sum_{i,j,k=1}^{2n-1}\tilde h(\E_j,\E_j)\tilde h(\E_k,\E_k)\langle\nabla^S_{\E_i}\E_j,\E_k\rangle\langle\nabla^S_{\E_i}\E_k,\E_j\rangle\\
                    &\overset{\eqref{metricequationnablas}}{=}\sum_{i,j,k=1}^{2n-1}\left(\E_i\tilde h(\E_j,\E_k)\right)^2+2\sum_{i,j,k=1}^{2n-1}\langle\nabla^S_{\E_i}\E_j,\E_k\rangle^2\tilde h(\E_k,\E_k)^2\\
                 &\quad+4\sum_{i,j,k=1}^{2n-1}\tilde h(\E_j,\E_j)\E_i\tilde h(\E_j,\E_k)\langle\nabla^S_{\E_i}\E_j,\E_k\rangle-2\sum_{i,j,k=1}^{2n-1}\tilde h(\E_j,\E_j)\tilde h(\E_k,\E_k)\langle\nabla^S_{\E_i}\E_j,\E_k\rangle^2
             \end{split}
             \end{equation*}
             On the other hand,
             \begin{equation*}
                 \begin{split}
                     \sum_{j=1}^{2n-1}&\tilde h(\E_j,\E_j)\left(\Delta^{\hhh,S}\tilde h\right)(\E_j,\E_j)= \sum_{i,j=1}^{2n-1}\tilde h(\E_j,\E_j)\nabla^S_{\E_i}\nabla^S_{\E_i}\tilde h(\E_j,\E_j)\\
                     &=\sum_{i,j=1}^{2n-1}\tilde h(\E_j,\E_j)\nabla^S_{\E_i}\left(\E_i\tilde h(\E_j,\E_j)-2\tilde h(\nabla^S_{\E_i}\E_j,\E_j)\right)\\
                      &=\sum_{i,j=1}^{2n-1}\tilde h(\E_j,\E_j)\E_i\E_i\tilde h(\E_j,\E_j)-2\sum_{i,j=1}^{2n-1}\tilde h(\E_j,\E_j)\E_i\tilde h(\nabla^S_{\E_i}\E_j,\E_j)\\
                      &\quad-\sum_{i,j=1}^{2n-1}\tilde h(\E_j,\E_j)\nabla^S_{\E_i}\E_i\tilde h(\E_j,\E_j)+2\sum_{i,j=1}^{2n-1}\tilde h(\E_j,\E_j)\tilde h\left(\nabla^S_{\nabla^S_{\E_i}\E_i}\E_j,\E_j\right)\\
                      &\quad-\sum_{i,j=1}^{2n-1}\tilde h(\E_j,\E_j)\E_i\tilde h(\nabla^S_{\E_i}\E_j,\E_j)+2\sum_{i,j=1}^{2n-1}\tilde h(\E_j,\E_j)\tilde h(\nabla^S_{\E_i}\nabla^S_{\E_i}\E_j,\E_j)\\
                      &\quad-\sum_{i,j=1}^{2n-1}\tilde h(\E_j,\E_j)\E_i\tilde h(\E_j,\nabla^S_{\E_i}\E_j)+2\sum_{i,j=1}^{2n-1}\tilde h(\E_j,\E_j)\tilde h(\nabla^S_{\E_i}\E_j,\nabla^S_{\E_i}\E_j)\\
                      &\overset{\eqref{comesetildehmetrica}}{=}\sum_{j=1}^{2n-1}\tilde h(\E_j,\E_j)\Delta^{\hhh,S}(\tilde h(\E_j,\E_j))-4\sum_{i,j=1}^{2n-1}\tilde h(\E_j,\E_j)\E_i\tilde h(\nabla^S_{\E_i}\E_j,\E_j)\\
                      &\quad+2\sum_{i,j=1}^{2n-1}\tilde h(\E_j,\E_j)\tilde h(\nabla^S_{\E_i}\nabla^S_{\E_i}\E_j,\E_j)+2\sum_{i,j=1}^{2n-1}\tilde h(\E_j,\E_j)\tilde h(\nabla^S_{\E_i}\E_j,\nabla^S_{\E_i}\E_j).
                 \end{split}
             \end{equation*}
             Notice that
             \begin{equation*}
             \begin{split}
                 -4\sum_{i,j=1}^{2n-1}&\tilde h(\E_j,\E_j)\E_i\tilde h(\nabla^S_{\E_i}\E_j,\E_j)=-4\sum_{i,j,k=1}^{2n-1}\tilde h(\E_j,\E_j)\E_i\left(\langle\nabla^S_{\E_i}\E_j,\E_k\rangle\tilde h(\E_k,\E_j)\right)\\
                 &\overset{\eqref{diagontildeh}}{=}-4\sum_{i,j=1}^{2n-1}\tilde h(\E_j,\E_j)^2\E_i\left(\langle\nabla^S_{\E_i}\E_j,\E_j\rangle\right)-4\sum_{i,j,k=1}^{2n-1}\tilde h(\E_j,\E_j)\E_i \tilde h(\E_k,\E_j)\langle\nabla^S_{\E_i}\E_j,\E_k\rangle\\
                 &\overset{\eqref{metricequationnablas}}{=}-4\sum_{i,j,k=1}^{2n-1}\tilde h(\E_j,\E_j)\E_i \tilde h(\E_j,\E_k)\langle\nabla^S_{\E_i}\E_j,\E_k\rangle.
             \end{split}
                        \end{equation*}
                   Moreover, 
               \begin{equation*}
                \begin{split}
                2\sum_{i,j=1}^{2n-1}\tilde h(\E_j,\E_j)\tilde h(\nabla^S_{\E_i}\nabla^S_{\E_i}\E_j,\E_j)&\overset{\eqref{diagontildeh}}{=}   2\sum_{i,j=1}^{2n-1}\tilde h(\E_j,\E_j)^2\langle\nabla^S_{\E_i}\nabla^S_{\E_i}\E_j,\E_j\rangle\\
                             &\overset{\eqref{metricequationnablas}}{=}-2\sum_{i,j=1}^{2n-1}\tilde h(\E_j,\E_j)^2\langle\nabla^S_{\E_i}\E_j,\nabla^S_{\E_i}\E_j\rangle\\
                             &=-2\sum_{i,j,k=1}^{2n-1}\tilde h(\E_j,\E_j)^2\langle\nabla^S_{\E_i}\E_j,\E_k\rangle^2\\
                             &\overset{\eqref{metricequationnablas}}{=}-2\sum_{i,j,k=1}^{2n-1}\tilde h(\E_k,\E_k)^2\langle\nabla^S_{\E_i}\E_j,\E_k\rangle^2.\\
                            \end{split}
                        \end{equation*}
                        Finally,
                        \begin{equation*}
           \begin{split}           2\sum_{i,j=1}^{2n-1}\tilde h(\E_j,\E_j)\tilde h(\nabla^S_{\E_i}\E_j,\nabla^S_{\E_i}\E_j)&=2\sum_{i,j,k=1}^{2n-1}\tilde h(\E_j,\E_j)\tilde h(\E_k,\E_k)\langle\nabla^S_{\E_i}\E_j,\E_k\rangle^2.
                            \end{split}
                        \end{equation*}
    Therefore we infer that
    \begin{equation}\label{moltoutilepersimonsmoduloeq}
        |\nabla^{\hhh,S}\tilde h|^2+\sum_{j=1}^{2n-1}\tilde h(\E_j,\E_j)\left(\Delta^{\hhh,S}\tilde h\right)(\E_j,\E_j)=\sum_{i,j,k=1}^{2n-1}\left(\E_i\tilde h(\E_j,\E_k)\right)^2+\sum_{j=1}^{2n-1}\tilde h(\E_j,\E_j)\Delta^{\hhh,S}(\tilde h(\E_j,\E_j)).
    \end{equation}
    We conclude noticing that
    \begin{equation*}
        \begin{split}
            \frac{1}{2}\Delta^{\hhh,S}|\tilde h|^2
            &=\frac{1}{2}\sum_{i,j,k=1}^{2n-1}\nabla^S_{\E_i}\nabla^S_{\E_i}\left(\tilde h(\E_j,\E_k)^2\right)\\
            &=\sum_{i,j,k=1}^{2n-1}\nabla^S_{\E_i}\left(\tilde h(\E_j,\E_k)\E_i\tilde  h(\E_j,\E_k)\right)\\
            &=\sum_{i,j,k=1}^{2n-1}\left(\E_i\tilde h(\E_j,\E_k)\right)^2+\sum_{j=1}^{2n-1}\tilde h(\E_j,\E_j)\Delta^{\hhh,S}(\tilde h(\E_j,\E_j))\\
            &\overset{\eqref{moltoutilepersimonsmoduloeq}}{=}|\nabla^{\hhh,S}\tilde h|^2+\sum_{j=1}^{2n-1}\tilde h(\E_j,\E_j)\left(\Delta^{\hhh,S}\tilde h\right)(\E_j,\E_j).
        \end{split}
    \end{equation*}
         \end{proof}
     \end{proposition}
     \begin{proposition}\label{laplacianodihmodificato}
       Let $\En_1,\ldots, \En_{2n-1}$ be any local orthonormal frame of $\hhh TS$.  It holds that
         \begin{equation}\label{migliorie1}
              \langle \nabla^S|\tilde h|^2,J(\vh)\rangle=4\alpha \left|\nabla_{J(\vh)}\vh  \right|^2-4\alpha|\tilde h|^2+2 \sum_{j,k=1}^{2n-1} \tilde h (\En_j,\En_k),\langle\nabla_{\En_j}\nabla_{J(\vh)}\vh,\En_k \rangle.
         \end{equation} 
     \end{proposition}
     \begin{proof}
    Take $\E_1,\ldots,\E_{2n-1}$ as in \eqref{diagontildeh}. Then
          \begin{equation*}
              \begin{split}
                   \langle \nabla^S|\tilde h|^2,J(\vh)\rangle&=2\sum_{j,k=1}^{2n-1}\tilde h(\E_j,\E_k)J(\vh)\left(\tilde h(\E_j,\E_k)\right)\\
                   &\overset{\eqref{diagontildeh}}{=}2\sum_{j=1}^{2n-1}\tilde h(\E_j,\E_j)J(\vh)\left(\tilde h(\E_j,\E_j)\right)\\
                   &=2\sum_{j=1}^{2n-1}\tilde h(\E_j,\E_j)\langle\nabla_{J(\vh)}\nabla_{\E_j}\vh,\E_j\rangle+2\sum_{j=1}^{2n-1}\tilde h(\E_j,\E_j)\langle\nabla_{\E_j}\vh,\nabla_{J(\vh)}\E_j\rangle\\
                   &\overset{R\equiv 0,\,\eqref{pseudotorsion}}{=}2\sum_{j=1}^{2n-1}\tilde h(\E_j,\E_j)\langle\nabla_{\E_j}\nabla_{J(\vh)}\vh,\E_j\rangle+2\sum_{j=1}^{2n-1}\tilde h(\E_j,\E_j)\langle\nabla_{[J(\vh),\E_j]}\vh,\E_j\rangle\\
                   &\quad+2\sum_{j=1}^{2n-1}\tilde h(\E_j,\E_j)\langle\nabla_{\E_j}\vh,\nabla_{\E_j}J(\vh)\rangle+2\sum_{j=1}^{2n-1}\tilde h(\E_j,\E_j)\langle\nabla_{\E_j}\vh,[J(\vh),\E_j]\rangle\\
                   &\overset{\eqref{jcommutaeq},\eqref{inhsiiiiiiiii}}{=}2\sum_{j,k=1}^{2n-1}\tilde h(\E_j,\E_k)\langle\nabla_{\E_j}\nabla_{J(\vh)}\vh,\E_k\rangle+4\sum_{j=1}^{2n-1}\tilde h(\E_j,\E_j)\tilde h([J(\vh),\E_j],\E_j).
              \end{split}
          \end{equation*}
        We conclude noticing that
        \begin{equation*}
            \begin{split}
            4\sum_{j=1}^{2n-1}\tilde h(\E_j,\E_j)\tilde h([J(\vh),\E_j],\E_j)&\overset{\eqref{diagontildeh}}{=}4\sum_{j=1}^{2n-1}\tilde h(\E_j,\E_j)^2\langle [J(\vh),\E_j],\E_j\rangle\\
                &\overset{\eqref{commutatoreesecondaforma}}{=}4\sum_{j=1}^{2n-1}\tilde h(\E_j,\E_j)^2 h(\E_j,\pi(J(\E_j)))\\
                &\overset{\eqref{hnonsimejjej}}{=}-4\alpha\sum_{j=1}^{2n-1}\tilde h(\E_j,\E_j)^2+4\alpha\sum_{j=1}^{2n-1} \langle J(\vh),\E_j\rangle^2 \tilde h(\E_j,\E_j)^2\\
                &=-4\alpha|\tilde h|^2+4\alpha\left|\nabla_{J(\vh)}\vh \right|^2.
            \end{split}
        \end{equation*}
     \end{proof}

 \section{Simons formulas and Kato inequalities}\label{secsimkatnew}
\subsection{The full Simons identity}
This section constitutes the core of the paper. Namely, we establish Simons formulas and Kato inequalities for minimal hypersurfaces in $\hh^n$. 
First, we provide the proof \Cref{newnewfullsimonsintro}, establishing a full Simons identity for the $(2,0)$-horizontal tensor field $\Delta^{\hhh,S} h$ associated with a minimal hypersurface and generalizing its Celebrated Riemannian counterpart (cf. \cite{MR233295}) to the sub-Riemannian setting. Namely, we show that
\begin{equation}\label{simonsformula}
         \begin{split}
              \Delta^{\hhh,S}h(X,Y)&=-qh(X,Y)+8\alpha^2h(X,Y)\\
              &\quad+4\hess^{\hhh,S}\alpha(\pi(J(X)),Y)+4\hess^{\hhh,S}\alpha(X,\pi(J(Y)))\\
             &\quad+\Big(16\alpha\pi(J(X))\alpha-8\alpha^2h(X,J(\vh))+4\left(\nabla_X\vh\right)\alpha\Big)\langle Y,J(\vh)\rangle\\
&\quad-2X\alpha h(Y,J(\vh))-2Y\alpha h( X,J(\vh))\\
             &\quad+2\alpha h(Y,\nabla_{\pi(J(X))}\vh)-2\alpha  \langle\nabla_X\nabla_{J(\vh)}\vh,Y\rangle-4\alpha^2h(\pi(J(X)),\pi(J(Y)))\\
              &\quad+2\alpha\left\langle J\left(\nabla _X\vh\right),\nabla_Y\vh\right\rangle
         \end{split}
     \end{equation}
for any $X,Y \in\Gamma(\hhh TS)$.
   \begin{proof}[Proof of \Cref{newnewfullsimonsintro}]
    Let $\En_1,\ldots, \En_{2n-1}$ be a local orthonormal frame of $\hhh TS$. Let $h_\s$ and $C_\vh$ be the $(1,0)$-tensor fields defined by  $h_\s(Z)=h(\s,Z)$ and $C_\vh (Z)=C(Z,\vh)$
    for any $Z\in\Gamma(\hhh T S)$. Then
  \begin{equation*}
        \begin{split}
\Delta^{\hhh,S}
&=\sum_{i=1}^{2n-1}\nabla^S_{\En_i}(\nabla^S_{\En_i}h(X,Y))\\
&\overset{\eqref{codazzipersimons}}{=}\sum_{i=1}^{2n-1}\nabla^S_{\En_i}(\nabla^S_{X}h(\En_i,Y)+2C(X,\En_i)h_\s (Y))\\
&=\sum_{i=1}^{2n-1}\hess^{\hhh,S}h(\En_i,X,\En_i,Y)+\underbrace{\sum_{i=1}^{2n-1}\nabla^S_{\En_i}(2C(X,\En_i)h_\s(Y))}_{\mathrm I}\\
&\overset{\eqref{commutatensoreconriemannpersimons}}{=}\sum_{i=1}^{2n-1}\hess^{\hhh,S}h(X,\En_i,\En_i,Y)+\underbrace{\sum_{i=1}^{2n-1}h(R^S(X,\En_i)\En_i,Y)+\sum_{i=1}^{2n-1}h(\En_i,R^S(X,\En_i)Y)}_{\mathrm {II}}\\
&\quad+\underbrace{\sum_{i=1}^{2n-1}2C(X,\En_i)\nabla^S_\s h(\En_i,Y)}_{\mathrm {III}}+\,\mathrm{I}\\
&=\sum_{i=1}^{2n-1}\nabla^S_X(\nabla^S_{\En_i}h(\En_i,Y))+\mathrm I+\mathrm {II}+ \mathrm{III}\\
&\overset{\eqref{simmetryandnabla}}{=}\sum_{i=1}^{2n-1}\nabla^S_X(\nabla^S_{\En_i}h(Y,\En_i))+\underbrace{\sum_{i=1}^{2n-1}\nabla^S_X(2\En_i\alpha C(Y,\En_i))}_{\mathrm{IV}}+\underbrace{\sum_{i=1}^{2n-1}\nabla^S_X(2\alpha C_\vh(\En_i) h(\En_i,Y))}_{\mathrm{V}}\\
&\quad\underbrace{-\sum_{i=1}^{2n-1}\nabla^S_X(2\alpha C_\vh(Y)h(\En_i,\En_i))}_{\mathrm{VI}}+\mathrm I+\mathrm {II}+ \mathrm{III}\\
&\overset{\eqref{codazzipersimons}}{=}\underbrace{\sum_{i=1}^{2n-1}\nabla^S_X(\nabla^S_{Y}h(\En_i,\En_i))}_{\mathrm{VII}}+\underbrace{\sum_{i=1}^{2n-1}\nabla^S_X(2C(Y,\En_i)h_\s(\En_i))}_{\mathrm{VIII}}+\mathrm I+\mathrm {II}+ \mathrm{III}+\mathrm{IV}+\mathrm{V}+\mathrm{VI}\\
&=\mathrm I+\mathrm {II}+ \mathrm{III}+\mathrm{IV}+\mathrm{V}+\mathrm{VI}+\mathrm{VII}+\mathrm{VIII}.
        \end{split}
    \end{equation*}
    We compute $\mathrm{I},\ldots,\mathrm{VIII}$. Let $\Es_1,\ldots,\Es_{2n-1}$ be a local orthonormal frame of $\hhh TS$ such that
    \begin{equation}\label{baseadattatapersimons}
        \Es_n=J(\vh)\qquad\text{and}\qquad J(\Es_1)=\Es_{n+1},\ldots,J(\Es_{n-1})=\Es_{2n-1}.
    \end{equation}
    
    \textbf{Computation of $\mathrm{I}$}.        
    \begin{equation*}
        \begin{split}
            \mathrm{I}
            &=\sum_{i=1}^{2n-1}\nabla^S_{\Es_i}(2C(X,\Es_i)h_\s(Y))\\
            &\overset{\eqref{leibnizpertensori}}{=}\sum_{i=1}^{2n-1}2h(\s,Y)\nabla^S_{\Es_i}C(X,\Es_i)+\sum_{i=1}^{2n-1}2C(X,\Es_i)\nabla^S_{\Es_i}h_\s(Y)\\
            &\overset{\eqref{commutatortensorpersimons},\eqref{proiezionepersimons}}{=}\sum_{i=1}^{2n-1}2h(\s,Y)C(\Es_i,\vh)h(\Es_i,X)-\sum_{i=1}^{2n-1}2h(\s,Y)C(X,\vh)h(\Es_i,\Es_i)+2 \nabla^S_{\pi(J(X))}h_\s(Y)\\
            &\overset{H\equiv 0}{=}\sum_{i=1}^{2n-1}2h(\s,Y)C(\Es_i,\vh)h(\Es_i,X)+2 \nabla^S_{\pi(J(X))}h_\s(Y)\\
            &\overset{\eqref{baseadattatapersimons}}{=}-2h(\s,Y)h(J(\vh),X)+2 \nabla^S_{\pi(J(X))}h_\s(Y).
        \end{split}
    \end{equation*}
    Notice that
    \begin{equation*}
        \begin{split}
-2h(\s,Y)h(J(\vh),X)\overset{\eqref{comecommutah}}{=}-2h(\s,Y)h(X,J(\vh))\overset{\eqref{accaesse}}{=}-2Y\alpha h(X,J(\vh))-4\alpha^2h(X,J(\vh))\langle Y,J(\vh)\rangle.
        \end{split}
    \end{equation*}
    On the other hand,
    \begin{equation*}
        \begin{split}
            2 \nabla^S_{\pi(J(X))}h_\s(Y)&=2\pi(J(X))(h(\s,Y))-2h(\s,\nabla^S_{\pi(J(X))}Y)\\
            &\overset{\eqref{accaesse}}{=}2\pi(J(X))(Y\alpha+2\alpha^2\langle Y,J(\vh)\rangle-2\nabla^S_{\pi(J(X))}Y\alpha-4\alpha^2\langle\nabla^S_{\pi(J(X))}Y,J(\vh)\rangle\\
            &=2\hess^{\hhh,S}\alpha(\pi(J(X)),Y)+8\alpha\pi(J(X))\alpha\langle Y,J(\vh)\rangle\\
            &\quad+4\alpha^2\pi(J(X))\langle Y,J(\vh)\rangle-4\alpha^2\langle\nabla_{\pi(J(X))}Y,J(\vh)\rangle\\
            &=2\hess^{\hhh,S}\alpha(\pi(J(X)),Y)+8\alpha\pi(J(X))\alpha\langle Y,J(\vh)\rangle+4\alpha^2\langle Y,\nabla_{\pi(J(X))}J(\vh)\rangle\\
&\overset{\eqref{jcommutaeq}}{=}2\hess^{\hhh,S}\alpha(\pi(J(X)),Y)+8\alpha\pi(J(X))\alpha\langle Y,J(\vh)\rangle-4\alpha^2\left\langle J(Y),\nabla_{\pi(J(X))}\vh\right\rangle\\
&=2\hess^{\hhh,S}\alpha(\pi(J(X)),Y)+8\alpha\pi(J(X))\alpha\langle Y,J(\vh)\rangle-4\alpha^2h(\pi(J(X)),\pi(J(Y))).
        \end{split}
    \end{equation*}
    In conclusion,
    \begin{equation*}
        \begin{split}
            \mathrm{I}&=2\hess^{\hhh,S}\alpha(\pi(J(X)),Y)+8\alpha\pi(J(X))\alpha\langle Y,J(\vh)\rangle-4\alpha^2h(\pi(J(X)),\pi(J(Y)))\\
            &\quad-2Y\alpha h(X,J(\vh))-4\alpha^2h(X,J(\vh))\langle Y,J(\vh)\rangle.
        \end{split}
    \end{equation*}
     \textbf{Computation of $\mathrm{VI}$.}
     \begin{equation*}
     \begin{split}
          \mathrm{VI}&=-\sum_{i=1}^{2n-1}\nabla^S_X(2\alpha C_\vh(Y)h(\En_i,\En_i))\\
          &=-\sum_{i=1}^{2n-1}\nabla_{X}^S(2\alpha C_\vh(Y))h(\En_i,\En_i)-2\alpha C(Y,\vh)\sum_{i=1}^{2n-1}\nabla^S_Xh(\En_i,\En_i)\\
          &=-\nabla^S_{X}(2\alpha C_\vh(Y))H-2\alpha C(Y,\vh)\trace \nabla^S_X h(\cdot,\cdot)\\
          &\overset{\eqref{nablahtrace}}{=}-\nabla_{X}^S(2\alpha C_\vh(Y))H-2\alpha C(Y,\vh)XH\\
          &\overset{H\equiv 0}{=}0.
     \end{split}
     \end{equation*}
       \textbf{Computation of $\mathrm{VII}$.}
       Thanks to \Cref{prophessianaetraccia}, we infer that
     \begin{equation*}
         \mathrm{VII}=\trace \hess^{\hhh,S}h(X,Y,\cdot,\cdot)=\hess^{\hhh,S}H(X,Y)\overset{H\equiv 0}{=}0.
     \end{equation*}
     \textbf{Computation of $\mathrm{VIII}$.}
     \begin{equation*}
         \begin{split}
             \mathrm{VIII}&=\trace \nabla^S_X(2C(Y,\cdot)h_\s(\cdot))\\
             &=\sum_{i=1}^{2n-1}\nabla^S_X(2C(Y,\Es_i)h_\s(\Es_i))\\
             &\overset{\eqref{leibnizpertensori}}{=}\sum_{i=1}^{2n-1}  2h(\s,\Es_i)\nabla^S_XC(Y,\Es_i)+\sum_{i=1}^{2n-1}2C(Y,\Es_i)\nabla^S_Xh_\s(\Es_i)\\
             &\overset{\eqref{commutatortensorpersimons}}{=}\sum_{i=1}^{2n-1}  2h(\s,\Es_i)C(\Es_i,\vh)h(X,Y)-\sum_{i=1}^{2n-1}2h(\s,\Es_i)C(Y,\vh)h(X,\Es_i)+\sum_{i=1}^{2n-1}2C(Y,\Es_i)\nabla^S_Xh_\s(\Es_i)\\
             &\overset{\eqref{baseadattatapersimons},\eqref{proiezionepersimons}}{=}-2h(\s,J(\vh))h(X,Y)-2h(\s,\nabla_X\vh)C(Y,\vh)+2\nabla^S_Xh_\s(\pi(J(Y))).\\
       \end{split}
     \end{equation*}
     Notice that
     \begin{equation*}   -2h(\s,J(\vh))h(X,Y)\overset{\eqref{accaesse}}{=}\left(-2\langle \nabla\alpha,J(\vh)\rangle-4\alpha^2\right)h(X,Y).
     \end{equation*}
     On the other hand,
     \begin{equation*}
         -2h(\s,\nabla_X\vh)C(Y,\vh)=2\left(\nabla_X\vh\right)\alpha\langle Y,J(\vh)\rangle +4\alpha^2h(X,J(\vh))\langle Y,J(\vh)\rangle.
     \end{equation*}
     Finally,
          \begin{equation*}
         \begin{split}
             2\nabla^S_Xh_\s(\pi(J(Y)))&=2X\left(h(\s,\pi(J(Y)))\right)-2h(\s,\nabla^S_X\pi(J(Y)))\\                     &\overset{\eqref{accaesse},\eqref{pijeixinhprmo}}{=}2\hess^{\hhh,S}\alpha (X,\pi(J(Y)))+4\alpha^2\langle\pi(J(Y)),\nabla_XJ(\vh)\rangle\\
             &=2\hess^{\hhh,S}\alpha (X,\pi(J(Y)))-4\alpha^2 h(X,J(\pi(J(Y))))\\
             &\overset{\eqref{pijeixinhprmo}}{=}2\hess^{\hhh,S}\alpha (X,\pi(J(Y)))+4\alpha^2 h(X,Y)-4\alpha^2h(X,J(\vh))\langle Y,J(\vh)\rangle.
         \end{split}
     \end{equation*}
     Putting the previous equations together, we conclude that
     \begin{equation*}
         \mathrm{VIII}=-2\langle \nabla\alpha,J(\vh)\rangle h(X,Y)+2\left(\nabla_X\vh\right)\alpha\langle Y,J(\vh)\rangle +2\hess^{\hhh,S}\alpha (X,\pi(J(Y))).
     \end{equation*}
     \textbf{Computation of $\mathrm{II}$.}
     \begin{equation*}
         \begin{split}
         \mathrm{II}             &=\sum_{i,j=1}^{2n-1}R^S(X,\Es_i,\Es_i,\Es_j)h(\Es_j,Y)+\sum_{i,j=1}^{2n-1}R^S(X,\Es_i,Y,\Es_j)h(\Es_i,\Es_j)\\
             &\overset{\eqref{gaussssssub}}{=}\sum_{i,j=1}^{2n-1}h(\Es_i,\Es_i)h(X,\Es_j)h(\Es_j,Y)-\sum_{i,j=1}^{2n-1}h(X,\Es_i)h(\Es_i,\Es_j)h(\Es_j,Y)\\
             &\quad+\sum_{i,j=1}^{2n-1}h(\Es_i,Y)h(X,\Es_j)h(\Es_i,\Es_j)-\sum_{i,j=1}^{2n-1}h(X,Y)h(\Es_i,\Es_j)^2\\
             &\overset{H\equiv 0}{=}\sum_{i,j=1}^{2n-1}h(\Es_i,Y)h(X,\Es_j)(h(\Es_i,\Es_j)-h(\Es_j,\Es_i))-|h|^2h(X,Y)\\
             &\overset{\eqref{comecommutah}}{=}2\alpha\sum_{i,j=1}^{2n-1}h(\Es_i,Y)h(X,\Es_j)C(\Es_j,\Es_i)-|h|^2h(X,Y)\\
             &\overset{\eqref{comecommutah}}{=}2\alpha\sum_{i,j=1}^{2n-1}h(Y,\Es_i)h(X,\Es_j)C(\Es_j,\Es_i)+4\alpha^2\sum_{i,j=1}^{2n-1}h(X,\Es_j)C(Y,\Es_i)C(\Es_j,\Es_i)-|h|^2h(X,Y).\\
         \end{split}
     \end{equation*}
     Notice that
     \begin{equation*}
         \begin{split}
             \sum_{i,j=1}^{2n-1}h(Y,\Es_i)h(X,\Es_j)C(\Es_j,\Es_i)&=\sum_{j=1}^{2n-1}h(X,\Es_j)\left\langle\nabla_Y\vh,\sum_{i=1}^{2n-1}\langle J(\Es_j),\Es_i\rangle \Es_i\right\rangle\\
             &\overset{\eqref{metricequationnablas}}{=}\sum_{j=1}^{2n-1}h(X,\Es_j)\langle\nabla_Y\vh,J(\Es_j)\rangle\\
             &=-\left\langle\nabla_X\vh,\sum_{j=1}^{2n-1}\langle J(\nabla_Y\vh),\Es_j\rangle \Es_j\right\rangle\\
             &=-\langle \nabla_X\vh,J(\nabla_Y\vh)\rangle+\langle\nabla_X\vh,\vh\rangle\langle J(\nabla_Y\vh),\vh\rangle\\
             &\overset{\eqref{metricequationnablas}}{=}-\langle \nabla_X\vh,J(\nabla_Y\vh)\rangle.
         \end{split}
     \end{equation*}
     On the other hand,
     \begin{equation*}
         \begin{split}
           \sum_{i,j=1}^{2n-1}h(X,\Es_j)C(Y,\Es_i)&C(\Es_j,\Es_i)=  \sum_{j=1}^{2n-1}h(X,\Es_j)\left\langle J(Y),\sum_{i=1}^{2n-1}\langle J(\Es_j),\Es_i\rangle \Es_i\right\rangle\\
           &=\sum_{j=1}^{2n-1}h(X,\Es_j)\left\langle J(Y), J(\Es_j)\right\rangle-\sum_{j=1}^{2n-1}h(X,\Es_j)\left\langle J(Y),\vh\right\rangle\langle J(\Es_j),\vh\rangle\\
           &=\sum_{j=1}^{2n-1}h(X,\Es_j)\left\langle Y, \Es_j\right\rangle+h(X,J(\vh))\left\langle J(Y),\vh\right\rangle\\
           &=h(X,Y)-h(X,J(\vh))\left\langle Y,J(\vh)\right\rangle.
         \end{split}
     \end{equation*}
     Therefore we conclude that
     \begin{equation*}
         \mathrm{II}=\left(-|h|^2+4\alpha^2\right)h(X,Y)-4\alpha^2h(X,J(\vh))\left\langle Y,J(\vh)\right\rangle-2\alpha\langle \nabla_X\vh,J(\nabla_Y\vh)\rangle.
     \end{equation*}
     \textbf{Computation of $\mathrm{IV}$.} 
     \begin{equation*}
         \begin{split}
             \mathrm{IV}&=\sum_{i=1}^{2n-1}2\nabla^S_X(\Es_i\alpha)C(Y,\Es_i)+\sum_{i=1}^{2n-1}2\Es_i\alpha\nabla^S_XC(Y,\Es_i)\\
             &\overset{\eqref{commutatortensorpersimons}}{=}2\nabla^S_X \left(\pi(J(Y)) \alpha\right)+\sum_{i=1}^{2n-1}2 G_i\alpha C(\Es_i,\vh)h(X,Y)-\sum_{i=1}^{2n-1}2\Es_i\alpha C(Y,\vh)h(X,\Es_i)\\
             &=2\hess^{\hhh,S}\alpha(X,\pi(J(Y)))-2\langle \nabla\alpha,J(\vh)\rangle h(X,Y)+  2 \langle Y,J(\vh)\rangle 
 \left(\nabla_X\vh\right)\alpha.
         \end{split}
     \end{equation*}
     \textbf{Computation of $\mathrm{V}$}
     \begin{equation*}
         \begin{split}
             \mathrm{V}&=\sum_{i=1}^{2n-1}2X\alpha C(\Es_i,\vh)h(\Es_i,Y)+\sum_{i=1}^{2n-1}2\alpha X(C(\Es_i,\vh)h(\Es_i,Y))\\
             &\quad-\sum_{i=1}^{2n-1}2\alpha C(\nabla^S_X\Es_i,\vh)h(\Es_i,Y)-\sum_{i=1}^{2n-1}2\alpha C(\Es_i,\vh)h(\nabla^S_X\Es_i,Y)-\sum_{i=1}^{2n-1}2\alpha C(\Es_i,\vh)h(\Es_i,\nabla^S_X Y)\\
             &\overset{\eqref{comecommutah}}{=}-2X\alpha h(Y,J(\vh))-2\alpha X h( J(\vh),Y)+\sum_{i=1}^{2n-1}2\alpha\langle \nabla^S_X\Es_i,J(\vh)\rangle h(\Es_i,Y)\\
             &\quad+ 2\alpha h(\nabla^S_XJ(\vh),Y)+2\alpha h(J(\vh),\nabla^S_X Y)\\
             &=-2X\alpha h(Y,J(\vh))-2\alpha X \langle\nabla_{J(\vh)}\vh,Y\rangle-\sum_{i=1}^{2n-1}2\alpha\langle \Es_i,\nabla^S_X J(\vh)\rangle h(\Es_i,Y)\\
             &\quad+ 2\alpha h(\nabla^S_XJ(\vh),Y)+2\alpha h(J(\vh),\nabla^S_X Y)\\
             &=-2X\alpha h(Y,J(\vh))-2\alpha X \langle\nabla_{J(\vh)}\vh,Y\rangle+2\alpha h(J(\vh),\nabla^S_X Y)\\
             &=-2X\alpha h(Y,J(\vh))-2\alpha  \langle\nabla^S_X\nabla_{J(\vh)}\vh,Y\rangle\\
             &=-2X\alpha h(Y,J(\vh))-2\alpha  \langle\nabla_X\nabla_{J(\vh)}\vh,Y\rangle.
         \end{split}
     \end{equation*}
   \textbf{Computation of $\mathrm{III}$.}
   \begin{equation*}
       \begin{split}
    \mathrm{III}&=2\nabla^S_\s h(\pi(J(X)),Y)\\
    &=   2\s\langle \nabla_{\pi(J(X))}\vh,Y\rangle-2\langle \nabla_{\nabla^S_\s \pi(J(X))}\vh,Y\rangle-2\langle\nabla _{\pi(J(X))}\vh,\nabla^S_\s Y)\\
    &= 2\langle \nabla^S_\s\nabla_{\pi(J(X))}\vh,Y\rangle+2\langle \nabla_{\pi(J(X))}\vh,\nabla^S_\s Y\rangle-2\langle \nabla_{\nabla^S_\s \pi(J(X))}\vh,Y\rangle-2\langle\nabla _{\pi(J(X))}\vh,\nabla^S_\s Y)\\
    &= 2\langle \nabla_\s\nabla_{\pi(J(X))}\vh,Y\rangle-2\langle \nabla_{\nabla_\s \pi(J(X))}\vh,Y\rangle+2\langle \nabla_\s \pi(J(X)),\vh\rangle\langle\nabla_\vh\vh,Y\rangle\\
&=2 \langle R(\s,\pi(J(X)))\vh,Y\rangle+2\langle \nabla_{\pi(J(X))}\nabla_\s\vh,Y\rangle+2\langle \nabla_{[\s,\pi(J(X))]}\vh,Y\rangle\\
&\quad-2\langle \nabla_{\nabla_\s \pi(J(X))}\vh,Y\rangle+2\langle \nabla_\s \pi(J(X)),\vh\rangle\langle\nabla_\vh\vh,Y\rangle\\
&\overset{R\equiv 0}{=}2\langle \nabla_{\pi(J(X))}\nabla_\s\vh,Y\rangle+2\langle \nabla_{[\s,\pi(J(X))]-\nabla_\s \pi(J(X))}\vh,Y\rangle+2\langle \nabla_\s \pi(J(X)),\vh\rangle\langle\nabla_\vh\vh,Y\rangle\\
&\overset{\eqref{nablaessenu},\eqref{propvh5}}{=} 2\langle \nabla_{\pi(J(X))}\nabla^\hhh\alpha,Y\rangle+4\langle \nabla_{\pi(J(X))}\alpha^2 J(\vh),Y\rangle+2\langle\nabla_{\tor_\nabla(\pi(J(X)),\s)}\vh,Y\rangle\\
&\quad-2\langle\nabla_{\nabla_{\pi(J(X))}\s}\vh,Y\rangle+4\alpha\langle Y,J(\vh)\rangle h(\s,\pi(J(X))).
       \end{split}
   \end{equation*}
   Notice that 
   \begin{equation*}
       2\langle \nabla_{\pi(J(X))}\nabla^\hhh\alpha,Y\rangle=2\langle \nabla^S_{\pi(J(X))}\nabla^\hhh\alpha,Y\rangle=2\hess^{\hhh,S}\alpha(\pi(J(X)),Y).
   \end{equation*}
   Moreover,
   \begin{equation*}
   \begin{split}
         4\langle \nabla_{\pi(J(X))}\alpha^2 J(\vh),Y\rangle&=8\alpha \langle Y,J(\vh)\rangle\pi(J(X))\alpha+4\alpha^2\langle \nabla_{\pi(J(X))} J(\vh),Y\rangle\\
&=8\alpha \langle Y,J(\vh)\rangle \pi(J(X))\alpha-4\alpha^2h(\pi(J(X)),\pi(J(Y))).\\
   \end{split}
   \end{equation*}
   In addition,
   \begin{equation*}
       \begin{split}
           2\langle\nabla_{\tor_\nabla(\pi(J(X)),\s)}\vh,Y\rangle&\overset{\eqref{pseudotorsion}}{=}-4\langle\pi(J(X)),J(\s)\rangle\langle\nabla_T\vh,Y\rangle\\
           &=-4\langle\pi(J(X)),J(T)\rangle\langle\nabla_T\vh,Y\rangle+4\alpha\langle\pi(J(X)),J(\vh)\rangle\langle\nabla_T\vh,Y\rangle\\
           &\overset{\eqref{pijeixinhprmo}}{=}0.
       \end{split}
   \end{equation*}
   Observe that
   \begin{equation*}
       \begin{split}
           -2\langle\nabla_{\nabla_{\pi(J(X))}\s}\vh,Y\rangle&\overset{\eqref{phflat}}{=}2\langle\nabla_{\nabla_{\pi(J(X))}\alpha\vh}\vh,Y\rangle\\
           &=2\pi(J(X))\alpha\langle\nabla_\vh\vh,Y\rangle+2\alpha\langle\nabla_{\nabla_{\pi(J(X))}\vh}\vh,Y\rangle\\
           &\overset{\eqref{propvh5}}{=}-4\alpha\pi(J(X))\alpha\langle Y,J(\vh)\rangle+2\alpha h(\nabla_{\pi(J(X))}\vh,Y)\\
           &\overset{\eqref{comecommutah}}{=}-4\alpha\pi(J(X))\alpha\langle Y,J(\vh)\rangle+2\alpha h(Y,\nabla_{\pi(J(X))}\vh)+4\alpha^2h(\pi(J(X)),\pi(J(Y))).
       \end{split}
   \end{equation*}
   Finally,
   \begin{equation*}
       4\alpha\langle Y,J(\vh)\rangle h(\s,\pi(J(X)))\overset{\eqref{accaesse},\eqref{pijeixinhprmo}}{=}4\alpha\pi(J(X))\alpha\langle Y,J(\vh)\rangle .
   \end{equation*}
   In conclusion, we infer that
   \begin{equation*}
       \mathrm{III}=2\hess^{\hhh,S}\alpha(\pi(J(X)),Y)+8\alpha\pi(J(X))\alpha\langle Y,J(\vh)\rangle+2\alpha h(Y,\nabla_{\pi(J(X))}\vh).
   \end{equation*}
The thesis follows from adding the terms that we have just computed.
\end{proof}
\subsection{Contracted Simons formulas}\label{secsimineq}
Combining \Cref{newnewfullsimonsintro} with \Cref{lapnormelapnonnorm}, \Cref{laplacianodihmodificato}, we provide contracted Simons formula for $\hat\Delta^{\hhh,S}|\tilde h|^2$. In order to handle the Hessian term appearing in the second line of \eqref{simonsformula}, we assume throughout this section that  \begin{equation}\label{graddialfasololungojeinuconk}\tag{$\mathrm{P}2$}
             \nabla^{\hhh,S}\alpha\equiv \langle\nabla\alpha,J(\vh)\rangle J(\vh).         \end{equation} 
             Condition \eqref{graddialfasololungojeinuconk} is motivated by the following relevant instances.
\begin{example}[Vertical hypersurfaces, $\mathrm{I}$]\label{newnewverticalhypuno}
Vertical hyperplanes satisfy \eqref{graddialfasololungojeinuconk}, since $\alpha\equiv 0$.    
More generally, we say that a hypersurface is \emph{vertical} if $T$ is tangent to $S$ at every point. As for vertical hyperplanes, $\alpha\equiv 0$, whence vertical hypersurfaces (not neessarily minimal) satisfy \eqref{graddialfasololungojeinuconk}.
\end{example}
\begin{example}[Catenoidal hypersurfaces, $\mathrm{II}$]\label{newnewexamplecatenoiddue}
Every catenoidal hypersurface arising from \Cref{newnewexamplecatenoid} satisfy \eqref{graddialfasololungojeinuconk}. More generally, every horizontally umbilic hypersurface (not necessarily minimal) satisfies \eqref{graddialfasololungojeinuconk} (cf. \cite[Proposition 4.2]{MR3794892}). We stress that differently from the previous examples, $\alpha=0$ only when $S_E$ intersect the horizontal hyperplane $\{t=0\}$.
\end{example}
Among minimal characteristic hypersurfaces, we have the following important example.
\begin{example}[Hyperbolic paraboloids, $\mathrm{II}$]\label{newnewhyperparadue}
    Let $S$ be the hyperbolic paraboloid as in \Cref{newnewhyperparauno}. A direct computation shows that $S\setminus S_0$ verifies \eqref{graddialfasololungojeinuconk}.   
\end{example}
Precisely, we recall that 
 condition \eqref{graddialfasololungojeinuconk} implies that, on a minimal hypersurface,
    \begin{equation}\label{newnewcontractedsimons}
    \begin{split}
         \frac{1}{2}\hat\Delta^{\hhh,S}|\tilde h|^2
        & =|\nabla^{\hhh,S}\tilde h|^2-q|\tilde h|^2+6\alpha^2|\tilde h|^2-6\alpha^2\left|\nabla_{J(\vh)}\vh\right|^2+4J(\vh)\alpha\,|\tilde h|^2-4J(\vh)\alpha\,\ell^2\\
        &\quad-\left(4J(\vh)\alpha+6\alpha^2\right) \left\langle\tilde h,\tilde h_J\right\rangle,
    \end{split}
\end{equation}
where $\tilde h_J$ is the $(2,0)$-tensor field defined by $\tilde h_J(X,Y)=\tilde h(\pi(J(X)),\pi(J(Y))$ for any $X,Y\in\Gamma(\hhh TS)$.
\begin{proof}[Proof of \Cref{newnewcontractedsimonsstatementintro}]
Fix a local orthonormal frame $\E_1,\ldots,\E_{2n-1}$ satisfying \eqref{diagontildeh}. Then
\begin{equation*}
    \begin{split}
         \frac{1}{2}\Delta^{\hhh,S}|\tilde h|^2&\overset{\eqref{lapnormelapnonnorm}}{=}|\nabla^{\hhh,S}\tilde h|^2+\sum_{j=1}^{2n-1}\tilde h(\E_j,\E_j)\left(\Delta^{\hhh,S}\tilde h\right)(\E_j,\E_j)\\
        & \overset{\eqref{simonsformula}}{=}|\nabla^{\hhh,S}\tilde h|^2-q|\tilde h|^2+8\alpha^2|\tilde h|^2\\
        &\quad+4\sum_{j=1}^{2n-1}\tilde h(\E_j,\E_j)\hess^{\hhh,S}\alpha(\pi(J(\E_j)),\E_j)+4\sum_{j=1}^{2n-1}\tilde h(\E_j,\E_j)\hess^{\hhh,S}\alpha(\E_j,\pi(J(\E_j)))\\
        &\quad+\sum_{j=1}^{2n-1}\tilde h(\E_j,\E_j)\Big(16\alpha\pi(J(\E_j))\alpha-8\alpha^2h(\E_j,J(\vh))+4\left(\nabla_{\E_j}\vh\right)\alpha\Big)\langle \E_j,J(\vh)\rangle\\
        &\quad-4\sum_{j=1}^{2n-1}\tilde h(\E_j,\E_j)\tilde h(\E_j,J(\vh))\E_j\alpha +2\alpha\sum_{j=1}^{2n-1} \tilde h(\E_j,\E_j)h(\E_j,\nabla_{\pi(J(\E_j))}\vh)\\
        &\quad-2\alpha \sum_{j=1}^{2n-1} \tilde h(\E_j,\E_j)\langle\nabla_{\E_j}\nabla_{J(\vh)}\vh,\E_j\rangle-4\alpha^2\sum_{j=1}^{2n-1} \tilde h(\E_j,\E_j)h(\pi(J(\E_j)),\pi(J(\E_j))).
    \end{split}
\end{equation*}
We compute each therm separately. First, we deal with the Hessian term. Indeed, by \eqref{graddialfasololungojeinuconk},
\begin{equation*}
\begin{split}
      4&\sum_{j=1}^{2n-1}\tilde h(\E_j,\E_j)\hess^{\hhh,S}\alpha(\pi(J(\E_j)),\E_j)= 4\sum_{j=1}^{2n-1}\tilde h(\E_j,\E_j)\left(\pi(J(\E_j))\E_j\alpha-\left(\nabla^S_{\pi(J(\E_j))}\E_j\right)\alpha\right)\\
      &\overset{\eqref{graddialfasololungojeinuconk}}{=} 4\sum_{j=1}^{2n-1}\tilde h(\E_j,\E_j)\left(\pi(J(\E_j))\left(\langle \E_j,J(\vh)\rangle J(\vh)\alpha\right)-\left\langle\nabla_{\pi(J(\E_j))}\E_j,J(\vh)\right\rangle J(\vh)\alpha\right)\\
      &=4\sum_{j=1}^{2n-1}\tilde h(\E_j,\E_j)\left(\left\langle \E_j,\nabla_{\pi(J(\E_j))}J(\vh)\right\rangle J(\vh)\alpha+\left\langle \E_j,J(\vh)\right\rangle \pi(J(\E_j))J(\vh)\alpha\right)\\
      &=-4 J(\vh)\alpha\sum_{j=1}^{2n-1}\tilde h(\E_j,\E_j)\tilde h\left(\pi(J(\E_j)),\pi(J(\E_j))\right)+4\sum_{j=1}^{2n-1}\tilde h(\E_j,\E_j)\left\langle \E_j,J(\vh)\right\rangle \pi(J(\E_j))J(\vh)\alpha.
\end{split}
\end{equation*}
But
\begin{equation*}
    \begin{split}
        \pi(J(\E_j))J(\vh)\alpha&\overset{\eqref{inhsiiiiiiiii}}{=}J(\vh)\pi(J(\E_j))\alpha+\left(\nabla_{\pi(J(\E_j))}J(\vh)\right)\alpha-\left(\nabla_{J(\vh)}\pi(J(\E_j))\right)\alpha\\   &\overset{\eqref{pijeixinhprmo},\eqref{graddialfasololungojeinuconk}}{=}\left\langle\nabla_{\pi(J(\E_j))}J(\vh),J(\vh)\right\rangle J(\vh)\alpha-\left\langle\nabla_{J(\vh)}\pi(J(\E_j)),J(\vh)\right\rangle J(\vh)\alpha\\
        &\overset{\eqref{pijeixinhprmo}}{=}-\left\langle\nabla_{J(\vh)}\vh,J(\pi(J(\E_j)))\right\rangle J(\vh)\alpha\\
        &\overset{\eqref{pijeixinhprmo}}{=}\left\langle\nabla_{J(\vh)}\vh,\E_j\right\rangle J(\vh)\alpha-\langle J(\vh),\E_j\rangle\left\langle\nabla_{J(\vh)}\vh,J(\vh)\right\rangle J(\vh)\alpha\\
        &=\tilde h(J(\vh),\E_j) J(\vh)\alpha-\ell\langle J(\vh),\E_j\rangle J(\vh)\alpha,
    \end{split}
\end{equation*}
whence
\begin{equation*}
    \begin{split}
          4\sum_{j=1}^{2n-1}\tilde h(\E_j,\E_j)&\hess^{\hhh,S}\alpha(\pi(J(\E_j)),\E_j)=-4 J(\vh)\alpha\sum_{j=1}^{2n-1}\tilde h(\E_j,\E_j)\tilde h\left(\pi(J(\E_j)),\pi(J(\E_j))\right)\\
          &\quad+4J(\vh)\alpha \sum_{j=1}^{2n-1}\tilde h(\E_j,\E_j)\tilde h(J(\vh),\E_j)\langle J(\vh),\E_j\rangle-4\ell J(\vh)\alpha \sum_{j=1}^{2n-1}\tilde h(\E_j,\E_j)\langle J(\vh),\E_j\rangle^2\\
          &=-4 J(\vh)\alpha\sum_{j=1}^{2n-1}\tilde h(\E_j,\E_j)\tilde h\left(\pi(J(\E_j)),\pi(J(\E_j))\right)+4J(\vh)\alpha\left|\nabla_{J(\vh)}\vh\right|^2-4\ell^2J(\vh)\alpha.
    \end{split}
\end{equation*}
In the last equality we used the fact that
\begin{equation}\label{shapejveelle}
    \sum_{j=1}^{2n-1}\tilde h(\E_j,\E_j)\tilde h(J(\vh),\E_j)\langle J(\vh),\E_j\rangle=\left|\nabla_{J(\vh)}\vh\right|^2\qquad\text{and}\qquad \sum_{j=1}^{2n-1}\tilde h(\E_j,\E_j)\langle J(\vh),\E_j\rangle^2=\ell.
\end{equation}
On the other hand,
\begin{equation*}
\begin{split}
      4\sum_{j=1}^{2n-1}\tilde h(\E_j,\E_j)&\hess^{\hhh,S}\alpha(\E_j,\pi(J(\E_j)))
      \overset{\eqref{pijeixinhprmo},\eqref{graddialfasololungojeinuconk}}{=}   - 4J(\vh)\alpha\sum_{j=1}^{2n-1}\tilde h(\E_j,\E_j)\left\langle\nabla_{\E_j}\pi(J(\E_j)),J(\vh)\right\rangle \\
      &\overset{\eqref{pijeixinhprmo}}{=}- 4J(\vh)\alpha\sum_{j=1}^{2n-1}\tilde h(\E_j,\E_j)\left\langle J(\pi(J(\E_j))),\nabla_{\E_j}\vh\right\rangle\\
      & \overset{\eqref{pijeixinhprmo}}{=} 4J(\vh)\alpha\sum_{j=1}^{2n-1}\tilde h(\E_j,\E_j)\left\langle \E_j,\nabla_{\E_j}\vh\right\rangle- 4J(\vh)\alpha\sum_{j=1}^{2n-1}\tilde h(\E_j,\E_j)\langle J(\vh),\E_j)\left\langle J(\vh),\nabla_{\E_j}\vh\right\rangle\\
      &\overset{\eqref{shapejveelle}}{=}4J(\vh)\alpha|\tilde h|^2-4J(\vh)\alpha\left|\nabla_{J(\vh)}\vh\right|^2.
\end{split}
\end{equation*}
Regarding the fourth line,
\begin{equation*}
    \begin{split}
        \sum_{j=1}^{2n-1}&\tilde h(\E_j,\E_j)\Big(16\alpha\pi(J(\E_j))\alpha-8\alpha^2h(\E_j,J(\vh))+4\left(\nabla_{\E_j}\vh\right)\alpha\Big)\langle \E_j,J(\vh)\rangle\\
        &\overset{\eqref{pijeixinhprmo},\eqref{graddialfasololungojeinuconk}}{=}  -8\alpha^2\sum_{j=1}^{2n-1}\tilde h(\E_j,\E_j)\tilde h(\E_j,J(\vh))\langle \E_j,J(\vh)\rangle+ 4 J(\vh)\alpha\sum_{j=1}^{2n-1}\tilde h(\E_j,\E_j)\tilde h(J(\vh),\E_j)\langle \E_j,J(\vh)\rangle\\
        &\overset{\eqref{shapejveelle}}{=}-8\alpha^2\left|\nabla_{J(\vh)}\vh\right|^2+4J(\vh)\alpha \left|\nabla_{J(\vh)}\vh\right|^2.
    \end{split}
\end{equation*}
Moreover,
\begin{equation*}
    \begin{split}
        -4\sum_{j=1}^{2n-1}&\tilde h(\E_j,\E_j)\tilde h(\E_j,J(\vh))\E_j\alpha +2\alpha\sum_{j=1}^{2n-1} \tilde h(\E_j,\E_j)h(\E_j,\nabla_{\pi(J(\E_j))}\vh)\\
        &\overset{\eqref{graddialfasololungojeinuconk}}{=}-4J(\vh)\alpha\sum_{j=1}^{2n-1}\tilde h(\E_j,\E_j)\tilde h(\E_j,J(\vh))\langle J(\vh),\E_j\rangle +2\alpha\sum_{j=1}^{2n-1} \tilde h(\E_j,\E_j)h(\E_j,\nabla_{\pi(J(\E_j))}\vh)\\
        &\overset{\eqref{theh},\eqref{shapejveelle}}{=}
        -4J(\vh)\alpha\left|\nabla_{J(\vh)}\vh\right|^2+2\alpha\sum_{j=1}^{2n-1} \tilde h(\E_j,\E_j)^2 h\left(\pi(J(\E_j)),\E_j\right)\\
        &\quad-2\alpha^2\sum_{j=1}^{2n-1} \tilde h(\E_j,\E_j)\tilde h(\pi(J(\E_j)),\pi(J(\E_j)))\\
        &\overset{\eqref{theh}}{=}-4J(\vh)\alpha\left|\nabla_{J(\vh)}\vh\right|^2+2\alpha\sum_{j=1}^{2n-1} \tilde h(\E_j,\E_j)^2\tilde h\left(\pi(J(\E_j)),\E_j\right)\\
        &\quad-2\alpha^2\sum_{j=1}^{2n-1} \tilde h(\E_j,\E_j)^2\langle J(\pi(J(\E_j))),\E_j\rangle-2\alpha^2\sum_{j=1}^{2n-1} \tilde h(\E_j,\E_j)\tilde h(\pi(J(\E_j)),\pi(J(\E_j)))\\
        &\overset{\eqref{tildehinejejeiejei}}{=}-4J(\vh)\alpha\left|\nabla_{J(\vh)}\vh\right|^2+2\alpha^2\sum_{j=1}^{2n-1} \tilde h(\E_j,\E_j)^2-2\alpha^2\sum_{j=1}^{2n-1} \tilde h(\E_j,\E_j)^2\langle J(\vh),\E_j\rangle^2\\
        &\quad -2\alpha^2\sum_{j=1}^{2n-1} \tilde h(\E_j,\E_j)\tilde h(\pi(J(\E_j)),\pi(J(\E_j)))\\
        &\overset{\eqref{shapejveelle}}{=}-4J(\vh)\alpha\left|\nabla_{J(\vh)}\vh\right|^2+2\alpha^2|\tilde h|^2-2\alpha^2\left|\nabla_{J(\vh)}\vh\right|^2-2\alpha^2\sum_{j=1}^{2n-1} \tilde h(\E_j,\E_j)\tilde h(\pi(J(\E_j)),\pi(J(\E_j))).
    \end{split}
\end{equation*}
Finally, by \Cref{laplacianodihmodificato},
\begin{equation*}
    -2\alpha \sum_{j=1}^{2n-1} \tilde h(\E_j,\E_j)\langle\nabla_{\E_j}\nabla_{J(\vh)}\vh,\E_j\rangle\overset{\eqref{migliorie1}}{=}4\alpha^2\left|\nabla_{J(\vh)}\vh\right|^2-4\alpha^2|\tilde h|^2-\alpha\langle \nabla |\tilde h|^2,J(\vh)\rangle.
\end{equation*}
The thesis follows combining the previous computations and recalling \eqref{modohoritanlaplaperintro}.
\end{proof}
As regards the product $\left\langle\tilde h,\tilde h_J\right\rangle$, we observe that
\begin{equation}\label{newnewcauchyschwartz}
    \left|\left\langle\tilde h,\tilde h_J\right\rangle\right|=\sum_{\substack{j,k=1 \\ j,k\neq n}}^{2n-1}\tilde h(\Es_j,\Es_k)\tilde h(J(\Es_j),J(\Es_k))\leq \sum_{\substack{j,k=1 \\ j,k\neq n}}^{2n-1}\tilde h(\Es_j,\Es_k)^2=|\tilde h|^2-2\left|\nabla_{J(\vh)}\vh\right|^2+\ell^2,
\end{equation}
were we exploited a local orthonormal frame $\Es_1,\ldots,\Es_{2n-1}$ as in \eqref{baseadattatapersimons}. When $n=2$ and $S$ is minimal$, \left\langle\tilde h,\tilde h_J\right\rangle$ can be computed explicitly.

\begin{proposition}\label{newnewpropsuformadiscalprod}
      Let $S\subseteq\hh^2$ be a smooth, immersed, non-characteristic hypersurface without boundary. Assume that $S$ is minimal. Then
      \begin{equation}\label{newnewexpressionofhhj}
           \left\langle\tilde h,\tilde h_J\right\rangle=2\left|\nabla_{J(\vh)}\vh\right|^2-|\tilde h|^2.
      \end{equation}
      In particular,
      \begin{equation}\label{newnewboundperelle}
          2|\tilde h|^2-4\left|\nabla_{J(\vh)}\vh\right|^2+\ell^2\geq0.
      \end{equation}
\end{proposition}
\begin{proof}
    First, \eqref{newnewboundperelle} follows by \eqref{newnewcauchyschwartz} and \eqref{newnewexpressionofhhj}. Let $\Es_1,\ldots,\Es_{2n-1}$ be as in \eqref{baseadattatapersimons}. Then
    \begin{equation*}
        |\tilde h|^2=2\left|\nabla_{J(\vh)}\vh\right|^2-\ell^2+\tilde h(\Es_1,\Es_1)^2+\tilde h(\Es_3,\Es_3)^2+2\tilde h(\Es_1,\Es_3)^2.
    \end{equation*}
    But, as $S$ is minimal,
    \begin{equation*}
    \begin{split}
          \left\langle\tilde h,\tilde h_J\right\rangle&=2\tilde h(\Es_1,\Es_1)\tilde h(\Es_3,\Es_3)-2\tilde h(\Es_1,\Es_3)^2\\
          &=\left(\tilde h(\Es_1,\Es_1)+\tilde h(\Es_3,\Es_3)\right)^2-h(\Es_1,\Es_1)^2-h(\Es_3,\Es_3)^2-2\tilde h(\Es_1,\Es_3)^2\\
            &\overset{H\equiv 0}{=}\ell^2-h(\Es_1,\Es_1)^2-h(\Es_3,\Es_3)^2-2\tilde h(\Es_1,\Es_3)^2\\
            &=2\left|\nabla_{J(\vh)}\vh\right|^2-|\tilde h|^2.
    \end{split}
    \end{equation*}
\end{proof}
Owing to \Cref{newnewpropsuformadiscalprod}, the contracted Simons formula in $\hh^2$ reads as follows.
\begin{corollary}\label{newnewsimonscontractedcorollary}
       Let $S\subseteq\hh^2$ be a smooth, immersed, non-characteristic hypersurface without boundary. Assume that \eqref{graddialfasololungojeinuconk} holds. Assume that $S$ is minimal. Then 
       \begin{equation}\label{newnewcontractedsimonsinh2prima}
    \begin{split}
         \frac{1}{2}\hat\Delta^{\hhh,S}|\tilde h|^2
        & =|\nabla^{\hhh,S}\tilde h|^2-q|\tilde h|^2+4\alpha^2|\tilde h|^2-8\alpha^2\left|\nabla_{J(\vh)}\vh\right|^2+2\alpha^2\ell^2\\
        &+4\left(J(\vh)\alpha+\alpha^2\right)\left( 2|\tilde h|^2-4\left|\nabla_{J(\vh)}\vh\right|^2+\ell^2\right)\\
        &\quad+\left(8J(\vh)\alpha+6\alpha^2\right) \left(\left|\nabla_{J(\vh)}\vh\right|^2-\ell^2\right).
    \end{split}
\end{equation}
\end{corollary}
While \eqref{newnewcontractedsimonsinh2prima} is just a  direct consequence of \eqref{newnewcontractedsimons} and \eqref{newnewexpressionofhhj}, the form in which we have expressed it allows certain important aspects to be highlighted. For instance, \eqref{newnewboundperelle} provides a lower bound for the second line of \eqref{newnewcontractedsimonsinh2prima} as soon as 
\begin{equation}\label{hpintermediatraqejvraffinata}\tag{$\mathrm{P}3$}
        J(\vh)\alpha+\alpha^2\geq 0.
    \end{equation}
 Condition \eqref{hpintermediatraqejvraffinata} implies that the fundamental function $q$ is bounded below by $|\tilde h|^2$ for any $n\geq 1$, as 
 \begin{equation}\label{segnodiq}
            \begin{split}
                  q =|\tilde h|^2 +4J(\vh)\alpha +(2n+2)\alpha^2             \overset{\eqref{hpintermediatraqejvraffinata}}{\geq} |\tilde h|^2+(2n-2)\alpha^2\geq |\tilde h|^2.            \end{split}
        \end{equation}
    We point out that \eqref{hpintermediatraqejvraffinata} appears naturally in the context of complete minimal hypersurfaces, as suggested by the following examples.
    \begin{example}[Vertical hypersurfaces, $\mathrm{II}$]\label{newnewverticalhypdue}
Vertical hyperplanes satisfy equality in \eqref{hpintermediatraqejvraffinata}, since $\alpha\equiv 0$.    
More generally every vertical hypersurface (not necessarily minimal) satisfies equality in \eqref{hpintermediatraqejvraffinata}.
\end{example}
\begin{example}[Catenoidal hypersurfaces, $\mathrm{III}$]\label{newnewexamplecatenoidtre}
Every catenoidal hypersurface arising from \Cref{newnewexamplecatenoid} satisfy \eqref{hpintermediatraqejvraffinata}. Precisely, a long but simple computation shows that
\begin{equation*}
    J(\vh)\alpha+\alpha^2=3E^2|z|^{-8}.
\end{equation*}
More generally, every minimal horizontally umbilic hypersurface satisfies \eqref{hpintermediatraqejvraffinata}. Indeed, in this case,  $J(\vh)\alpha+\alpha^2=\frac{1}{2n-2}|\tilde h|^2$ (cf. \cite[Proposition 4.2]{MR3794892}).
\end{example}
\begin{example}[Helicoidal hypersurfaces]\label{newnewexampleelicoid}
Let $S\subseteq\hh^n$ be the helicoid parametrized by 
\begin{equation*}
    F:\rr^{2n}\longrightarrow\hh^n,\qquad F(s,\theta,\xi_2,\ldots,\xi_n,\eta_2,\ldots,\eta_n)=(s\cos\theta,\xi_2,\ldots,\xi_n,s\sin\theta,\eta_2,\ldots,\eta_n,\theta).
\end{equation*}
    An easy computation shows that $S$ is a smooth, embedded, complete, minimal, non-characteristic hypersurface. Moreover, it is not difficult to verify that
    \begin{equation*}
        J(\vh)\alpha+\alpha^2=\left(\frac{1+s^2}{(1+s^2)^2+s^2\left(\xi_1^2+\cdots+\xi_n^2+\eta_1^2+\cdots+\eta_n^2\right)}\right)^2,
    \end{equation*}
    whence $S$ verifies \eqref{hpintermediatraqejvraffinata}.
\end{example}
The properties of a minimal hypersurface of being complete and non-characteristic appear crucial for the validity of \eqref{hpintermediatraqejvraffinata}. This fact is already evident in $\hh^1$. In the first Heisenberg group, where we recall that \eqref{taiwansplitting} and \eqref{graddialfasololungojeinuconk} are always satisfied, every complete, minimal, non characteristic surface in $\hh^1$ satisfies \eqref{hpintermediatraqejvraffinata} (cf. \cite{MR3406514,GR24}). Instead, there are examples of complete, minimal, characteristic surfaces for which \eqref{hpintermediatraqejvraffinata} fails. In turn, the latter provide counterexamples in higher dimension too.
\begin{example}[Hyperbolic paraboloids,  $\mathrm{III}$]\label{newnewhyperparatre}
    Let $n\geq 1$. Let $S$ hyperbolic paraboloid of \Cref{newnewhyperparauno}. We recall that $S$ is a complete, minimal hypersurface, satisfying \eqref{taiwansplitting} and \eqref{graddialfasololungojeinuconk}, with uncountably many characteristic points. A direct computations reveals that
    \begin{equation*}
        J(\vh)\alpha+\alpha^2=-\frac{1}{4x_1^2+\cdots+4x_n^2}
    \end{equation*}
    on $S\setminus S_0$, whence $S$ does not satisfy \eqref{hpintermediatraqejvraffinata}. 
\end{example}
As further support to the intrinsic nature of \eqref{hpintermediatraqejvraffinata}, it is interesting to analyze its behavior in the flattest, characteristic example.
\begin{example}[The horizontal hyperplane]\label{newnewhorizontalhyperplane}
    Let $S\subseteq \hh^n$ be the horizontal hyperplane $\{t=0\}$. $S$ is a complete, minimal, characteristic hypersurface, with $S_0=\{0\}$ and $\alpha\neq 0$ on $S\setminus S_0$. Moreover, $S\setminus S_0$ verifies \eqref{taiwansplitting} and \eqref{graddialfasololungojeinuconk}. In addition, a direct computation shows that $J(\vh)\alpha+\alpha^2\equiv 0$ on $S\setminus S_0$.
\end{example}
   To conclude this section, we observe that the control on the horizontal shape operator provided by \eqref{taiwansplitting} would allow to discard the contribution provided by the last line of \eqref{newnewcontractedsimonsinh2}. Indeed, \eqref{taiwansplitting} is equivalent to require that $\left|\nabla_{J(\vh)} \vh\right|^2=\ell^2$. Therefore, when \eqref{taiwansplitting}, \eqref{graddialfasololungojeinuconk} and \eqref{hpintermediatraqejvraffinata} hold, \eqref{newnewcontractedsimonsinh2} implies the following lower bound.
    \begin{corollary}
       Let $S\subseteq\hh^2$ be a smooth, immersed, non-characteristic hypersurface without boundary. Assume that \eqref{taiwansplitting}, \eqref{graddialfasololungojeinuconk} and \eqref{hpintermediatraqejvraffinata} hold. Assume that $S$ is minimal. Then 
       \begin{equation}\label{newnewcontractedsimonsinh2}
    \begin{split}
         \frac{1}{2}\hat\Delta^{\hhh,S}|\tilde h|^2
        \geq|\nabla^{\hhh,S}\tilde h|^2-q|\tilde h|^2+4\alpha^2|\tilde h|^2-6\ell^2.     
    \end{split}
\end{equation}
\end{corollary}
 The inequality provided by \eqref{newnewcontractedsimonsinh2} is sharp in the class of minimal hypersurfaces satisfying \eqref{taiwansplitting}, \eqref{graddialfasololungojeinuconk} and \eqref{hpintermediatraqejvraffinata}, as the next example shows.
 \begin{example}[Catenoidal hypersurfaces, $\mathrm{IV}$]\label{newnewexamplecatenoidalquattro}
     Let $S_E\subseteq\hh^2$ be any catenoidal hypersurface as in \eqref{newnewexamplecatenoid}. Since $S_E$ satisfies \eqref{graddialfasololungojeinuconk}, then it satisfies \eqref{newnewcontractedsimonsinh2} by \Cref{newnewsimonscontractedcorollary}. A computation shows that $|\tilde h|^2=6E|z|^{-8}$. We know from \eqref{newnewellpercatenoid} that $\ell^2=|\nabla_{J(\vh)}\vh|^2=4E^2|z|^{-8}$. In particular, as $2|\tilde h|^2=3\ell^2$ and $\ell^2=|\nabla_{J(\vh)}\vh|^2$, the last two lines of \eqref{newnewcontractedsimonsinh2prima} vanish, and $S_E$ satisfies the equality in \eqref{newnewcontractedsimonsinh2}.
 \end{example}
   \subsection{Improved Kato inequalities}\label{seckatoineq}
    In this section we provide a lower bound for $|\nabla^{\hhh,S}\tilde h|^2$ in terms of $|\nabla^{\hhh,S}|\tilde h|^2|^2$, basically under assumptions \eqref{taiwansplitting} and \eqref{graddialfasololungojeinuconk}. Namely, we show that
\begin{equation}\label{katostatement2conk}
            \left(1+\frac{k}{2n-1}\right)|\nabla^{\hhh,S}|\tilde h|^2|^2\leq 4|\tilde h|^2|\nabla^{\hhh,S}\tilde h|^2 +4\alpha^2|\tilde h|^2\left((4k-2)|\tilde h|^2+(2+2kn-2k-4n)\hnn^2\right)
         \end{equation}
         for any $k\in [0,2]$.
     \begin{proof}[Proof of \Cref{newnewkatoconkintro}]
  Let $\En_1,\ldots,\En_{2n-1}$ be any local orthonormal frame of $\hhh TS$. 
        Notice that
         \begin{equation}\label{katofirststepconk}
         \begin{split}
             \En_i|\tilde h|^2&=\sum_{j,k=1}^{2n-1}\En_i(\tilde h(\En_j,\En_k)^2)\\
              &=2\sum_{j,k=1}^{2n-1}\tilde h(\En_j,\En_k)\nabla^S_{\En_i}\tilde h(\En_j,\En_k)+ 4\sum_{j,k=1}^{2n-1}\tilde h(\En_j,\En_k)\tilde h(\nabla^S_{\En_i}\En_j,\En_k)\\
                 &=2\sum_{j,k=1}^{2n-1}\tilde h(\En_j,\En_k)\nabla^S_{\En_i}\tilde h(\En_j,\En_k)+4\sum_{j,k,s=1}^{2n-1}\tilde h(\En_j,\En_k)\tilde h(\En_s,\En_k)\langle \nabla^S_{\En_i}\En_j,\En_s\rangle\\
              &\overset{\eqref{metricequationnablas}}{=}2\sum_{j,k=1}^{2n-1}\tilde h(\En_j,\En_k)\nabla^S_{\En_i}\tilde h(\En_j,\En_k)
         \end{split}
         \end{equation}
         for any $i=1,\ldots,2n-1$, 
         so that
         \begin{equation}\label{katosecondstepconk}
             \begin{split}
                 |\nabla^{\hhh,S}|\tilde h|^2|^2
                 \overset{\eqref{katofirststepconk}}{=}4\sum_{i=1}^{2n-1}\left(\sum_{j,k=1}^{2n-1}\tilde h(\En_j,\En_k)\nabla^S_{\En_i}\tilde h(\En_j,\En_k)\right)^2
                 \leq 4|\tilde h|^2|\nabla^{\hhh,S}\tilde h|^2,
             \end{split}
         \end{equation}
         where in the last passage we used the Cauchy-Schwarz inequality. In particular, \eqref{newnewkatostatement1conkintro} follows. Assume that $S$ is minimal, and that \eqref{taiwansplitting} and \eqref{graddialfasololungojeinuconk} hold.  
                 Let $\E_1,\ldots,\E_{2n-1}$ be as in \eqref{diagontildeh}. 
        Since we are assuming \eqref{taiwansplitting}, we set $\E_n=J(\vh)$. Arguing as in \eqref{katosecondstepconk}, 
         \begin{equation*}
             \begin{split}
                 |\nabla^{\hhh,S}|\tilde h|^2|^2&\leq 4|\tilde h|^2\sum_{i,j=1}^{2n-1}\nabla^S_{\E_i}\tilde h(\E_j,\E_j)^2\\
                  &=4|\tilde h|^2\sum_{\substack{i,j=1 \\ i\neq j}}^{2n-1}\nabla^S_{\E_i}\tilde h(\E_j,\E_j)^2+4|\tilde h|^2\sum_{i=1}^{2n-1}\nabla^S_{\E_i}\tilde h(\E_i,\E_i)^2\\
                  &\overset{\eqref{nablahtrace}}{=}4|\tilde h|^2\sum_{\substack{i,j=1 \\ i\neq j}}^{2n-1}\nabla^S_{\E_i}\tilde h(\E_j,\E_j)^2+4|\tilde h|^2\sum_{i=1}^{2n-1}\left(\sum_{\substack{j=1 \\ j\neq i}}^{2n-1}\nabla^S_{\E_i}\tilde h(\E_j,\E_j)\right)^2\\
                  &\leq 4|\tilde h|^2\sum_{\substack{i,j=1 \\ i\neq j}}^{2n-1}\nabla^S_{\E_i}\tilde h(\E_j,\E_j)^2+4(2n-2)|\tilde h|^2\sum_{\substack{i,j=1 \\ i\neq j}}^{2n-1}\nabla^S_{\E_i}\tilde h(\E_j,\E_j)^2\\
                  &=4(2n-1)|\tilde h|^2\sum_{\substack{i,j=1 \\ i\neq j}}^{2n-1}\nabla^S_{\E_i}\tilde h(\E_j,\E_j)^2.
             \end{split}
         \end{equation*}
         Set
         \begin{equation*}
             \mathcal R=\sum_{\substack{i,j=1 \\ i\neq j}}^{2n-1}\nabla^S_{\E_i}\tilde h(\E_j,\E_j)^2-\sum_{\substack{i,j=1 \\ i\neq j}}^{2n-1}\nabla^S_{\E_j}\tilde h(\E_i,\E_j)^2.
         \end{equation*}
         Then
         \begin{equation*}
         \begin{split}
              &\left(1+\frac{k}{2n-1}\right)|\nabla^{ \hhh ,S}|\tilde h|^2|^2            
             \leq 4|\tilde h|^2\left(\sum_{i,j=1}^{2n-1}\nabla^S_{\E_i}\tilde h(\E_j,\E_j)^2+k\sum_{\substack{i,j=1 \\ i\neq j}}^{2n-1}\nabla^S_{\E_i}\tilde h(\E_j,\E_j)^2\right)\\
             &\qquad =4|\tilde h|^2\left(\sum_{i,j=1}^{2n-1}\nabla^S_{\E_i}\tilde h(\E_j,\E_j)^2+k\sum_{\substack{i,j=1 \\ i\neq j}}^{2n-1}\nabla^S_{\E_j}\tilde h(\E_i,\E_j)^2 +k\mathcal R\right)\\
              &\qquad = 4|\tilde h|^2\left(\sum_{i,j=1}^{2n-1}\nabla^S_{\E_i}\tilde h(\E_j,\E_j)^2+2\sum_{\substack{i,j=1 \\ i\neq j}}^{2n-1}\nabla^S_{\E_j}\tilde h(\E_i,\E_j)^2 +(k-2)\sum_{\substack{i,j=1 \\ i\neq j}}^{2n-1}\nabla^S_{\E_j}\tilde h(\E_i,\E_j)^2+k\mathcal R\right)\\
                 &\qquad \leq  4|\tilde h|^2\left(|\nabla^{\hhh,S}\tilde h|^2 +(k-2)\sum_{\substack{i,j=1 \\ i\neq j}}^{2n-1}\nabla^S_{\E_j}\tilde h(\E_i,\E_j)^2+k\mathcal R\right)\\
         \end{split}
         \end{equation*}
       We apply the Codazzi equation \eqref{codazzihtildeeqaba} to handle the remainder $\mathcal R$. In this way, by \eqref{taiwansplitting} and \eqref{graddialfasololungojeinuconk},
\begin{equation*}
    \begin{split}
        \mathcal R&=\sum_{\substack{i,j=1 \\ i\neq j}}^{2n-1}\nabla^S_{\E_i}\tilde h(\E_j,\E_j)^2-\sum_{\substack{i,j=1 \\ i\neq j}}^{2n-1}\nabla^S_{\E_j}\tilde h(\E_i, \E_j)^2\\
        &\overset{\eqref{taiwansplitting}}{=}\sum_{\substack{j=1 \\ j\neq n}}^{2n-1}\nabla^S_{J(\vh)}\tilde h(\E_j,\E_j)^2+\sum_{\substack{i=1 \\ i\neq n}}^{2n-1}\nabla^S_{\E_i}\tilde h(J(\vh),J(\vh))^2+\sum_{\substack{i,j=1 \\ i\neq j,\,i,j\neq n}}^{2n-1}\nabla^S_{\E_i}\tilde h(\E_j,\E_j)^2-\sum_{\substack{i,j=1 \\ i\neq j}}^{2n-1}\nabla^S_{\E_j}\tilde h(\E_i,\E_j)^2\\
        &\overset{\eqref{codazzihtildeeqaba},\eqref{graddialfasololungojeinuconk}}{=}\sum_{\substack{j=1 \\ j\neq n}}^{2n-1}\left(\nabla^S_{\E_j}\tilde h(J(\vh),\E_j)-\alpha\tilde h(\E_j,\E_j)\right)^2\\
        &\quad+\sum_{\substack{i=1 \\ i\neq n}}^{2n-1}\nabla^S_{J(\vh)}\tilde h(\E_i ,J(\vh))^2+\sum_{\substack{i,j=1 \\ i\neq j,\,i,j\neq n}}^{2n-1}\nabla^S_{\E_j}\tilde h(\E_i,\E_j)^2-\sum_{\substack{i,j=1 \\ i\neq j}}^{2n-1}\nabla^S_{\E_j}\tilde h(\E_i,\E_j)^2\\
        &=-2\alpha\sum_{\substack{j=1 \\ j\neq n}}^{2n-1}\tilde h(\E_j,\E_j)\nabla_{\E_j}^S\tilde h(J(\vh),\E_j)+\alpha^2|\tilde h|^2-\alpha^2\ell^2\\
        &\overset{\eqref{taiwansplitting}}{=}2\alpha\sum_{\substack{j=1 \\ j\neq n}}^{2n-1}\tilde h(\E_j,\E_j)\tilde h(\nabla^S_{\E_j}J(\vh),\E_j)+2\alpha\sum_{\substack{j=1 \\ j\neq n}}^{2n-1}\tilde h(\E_j,\E_j)\tilde h(J(\vh),\nabla^S_{\E_j}\E_j)+\alpha^2|\tilde h|^2-\alpha^2\ell^2\\
        &\overset{\eqref{diagontildeh}}{=}2\alpha\sum_{\substack{j=1 \\ j\neq n}}^{2n-1}\tilde h(\E_j,\E_j)^2\langle\nabla_{\E_j}J(\vh),\E_j\rangle+2\alpha\ell\sum_{\substack{j=1 \\ j\neq n}}^{2n-1}\tilde h(\E_j,\E_j) \langle \nabla_{\E_j}\E_j,J(\vh)\rangle               
  +\alpha^2|\tilde h|^2-\alpha^2\ell^2  \\
        &=-2\alpha\sum_{\substack{j=1 \\ j\neq n}}^{2n-1}\tilde h(\E_j,\E_j)^2h(\E_j,J(\E_j))+2\alpha\ell\sum_{\substack{j=1 \\ j\neq n}}^{2n-1}\tilde h(\E_j,\E_j)h(\E_j,J(\E_j))+\alpha^2|\tilde h|^2-\alpha^2\ell^2 \\
        &\overset{\eqref{hnonsimejjej}}{=}2\alpha^2\sum_{\substack{j=1 \\ j\neq n}}^{2n-1}\tilde h(\E_j,\E_j)^2-2\alpha^2\ell\sum_{\substack{j=1 \\ j\neq n}}^{2n-1}\tilde h(\E_j,\E_j)+\alpha^2|\tilde h|^2-\alpha^2\ell^2\\
        &\overset{H\equiv 0}{=}3\alpha^2|\tilde h|^2-\alpha^2\ell^2.
    \end{split}
\end{equation*}
       
       On the other hand,
       \begin{equation*}
           \begin{split}
               \sum_{\substack{i,j=1 \\ i\neq j}}^{2n-1}\nabla^S_{\E_j}\tilde h(\E_i,\E_j)^2&\geq \sum_{\substack{j=1 \\ j\neq n}}^{2n-1}\nabla^S_{\E_j}\tilde h(J(\vh),\E_j)^2\\   &\overset{\eqref{taiwansplitting}}{=} \sum_{\substack{j=1 \\ j\neq n}}^{2n-1}\left(\tilde h(\nabla^S_{\E_j}J(\vh),\E_j)+\tilde h(J(\vh),\nabla^S_{\E_j}\E_j)\right)^2\\    
                &       \overset{\eqref{diagontildeh}}{=} \sum_{\substack{j=1 \\ j\neq n}}^{2n-1}\left(\langle \nabla_{\E_j}J(\vh),\E_j\rangle\tilde h(\E_j,\E_j)+\langle \nabla_{\E_j}\E_j,J(\vh)\rangle\hnn\right)^2\\          
          &\overset{\eqref{hnonsimejjej}}{=}\sum_{\substack{j=1 \\ j\neq n}}^{2n-1}\left(\alpha\tilde h(\E_j,\E_j)-\alpha\hnn\right)^2\\
                  &\overset{H\equiv 0}{=}\alpha^2|\tilde h|^2+(2n-1)\alpha^2\hnn^2.
           \end{split}
       \end{equation*}
The thesis follows combining the previous computations.
     \end{proof}
     \subsection{Simons-Kato inequalities}\label{secsimonskato}
     Combining the results of \Cref{secsimineq} and \Cref{seckatoineq}, we can derive several consequences. For future convenience, we focus on the case $n=2$, since higher dimensional estimates follows similarly. To this aim, we fix a minimal hypersurface $S\subseteq\hh^2$ such that \eqref{taiwansplitting}, \eqref{graddialfasololungojeinuconk} and \eqref{hpintermediatraqejvraffinata} hold. For any $\delta\geq 0$  we define the function $\adel$ by
     \begin{equation}\label{adeldef}
         \adel(p)=\sqrt{|\tilde h_p|^2+\delta}
     \end{equation}
for any $p\in S$. Notice that the function $\adel$ belongs to $C^\infty(S)$, for any $\delta>0$. This desingularization will be crucial in the forthcoming \Cref{stabilitisectionnnnnn}. Observe that $A_\delta\geq|\tilde h|$ and that $A_\delta\to|\tilde h|$ uniformly on $S$ as $\delta\to 0^+$. Since $\adel\geq 0$, we can multiply \eqref{newnewcontractedsimonsinh2} by $4\adel^2$, so that
\begin{equation}\label{simonscondelta}
         2\adel^2\hat\Delta^{\hhh,S}|\tilde h|^2\geq 4\adel^2|\nabla^{\hhh,S}\tilde h|^2-4q\adel^2|\tilde h|^2+4\alpha^2\adel^2\left(4|\tilde h|^2-6\hnn^2\right)
    \end{equation}
for any $\delta\geq 0$. In particular, by \eqref{newnewboundperelle}, we deduce from  \eqref{simonscondelta} that
\begin{equation}\label{simonscondeltadue}
         2\adel^2\hat\Delta^{\hhh,S}|\tilde h|^2\geq 4\adel^2|\nabla^{\hhh,S}\tilde h|^2-4q\adel^2|\tilde h|^2+4\alpha^2|\tilde h|^2\left(4|\tilde h|^2-6\hnn^2\right)
    \end{equation}
for any $\delta\geq 0$. Moreover, \eqref{katostatement2conk} implies that
\begin{equation}\label{katocondelta}
            \left(1+\frac{k}{2n-1}\right)|\nabla^S|\tilde h|^2|^2\leq 4A(\delta)^2|\nabla^{\hhh,S}\tilde h|^2 +4\alpha^2|\tilde h|^2\left((4k-2)|\tilde h|^2+(2+2kn-2k-4n)\hnn^2\right)
         \end{equation}
for any $\delta\geq 0$ and any $k\in [0,2]$.
Combining \eqref{segnodiq}, \eqref{simonscondeltadue} and \eqref{katocondelta}, we conclude that
\begin{equation}\label{simonskatoeqnugualedue}
    2\adel^2\hat\Delta^{\hhh,S}|\tilde h|^2\geq \left(1+\frac{k}{2n-1}\right)|\nabla^S|\tilde h|^2|^2-4q\adel^4+4\alpha^2|\tilde h|^2\fisd
\end{equation}
for any $\delta\geq 0$ and any $k\in [0,2]$, where
\begin{equation}\label{doveedefinitafis}
      \fisd(p):=(6-4k)|\tilde h_p|^2-2k\hnn(p)^2
\end{equation}
for any $p\in S$. 

\section{The improved stability inequality}\label{stabilitisectionnnnnn}
    In this section, we establish the sub-Riemannian analogue of \cite[Theorem 1]{MR0423263} in the Heisenberg group by proving \Cref{newnewmainstabilityintrostatement}. Our approach relies on the Simons–Kato inequalities obtained in \Cref{secsimonskato}. For future purposes, we specialize our exposition to $\hh^2$, although the reader should not encounter difficulties in adapting the same strategy to higher-dimensional cases.
    \begin{proof}[Proof of \Cref{newnewmainstabilityintrostatement}]
    We recall that we are assuming 
    \begin{equation}\label{stimafondamentale}
            \fisd(p)\geq 0
        \end{equation}      
        for a given $k\in [0,2]$ and for any $p\in S$. 
    Let us recall (cf. \cite[Lemma 11.6]{MR2644313}) that
    \begin{equation}\label{chainrulegarofalo}
        \hat\Delta^{\hhh,S}\left( F\circ u\right)=\left(\Ddot{F}\circ u\right)|\nabla^{\hhh,S}u|^2+\left(\Dot{F}\circ u\right)\hat\Delta^{\hhh,S}u
    \end{equation}
    for any $F\in C^2(\rr)$ and any $u\in C^2(S)$.
        Fix $\beta,\delta>0$, $k\in [0,2]$ and a test function $\varphi\in C^1_c(S)$. Choosing $\xi=\adel^\beta\varphi$ in the stability inequality \eqref{stabilityinequality}, where $\adel$ is as in \eqref{adeldef}, we infer that
\begin{equation*}
    \begin{split}
        \int_S&q\left(\adel^\beta\varphi\right)^2\,d \sigma_\hhh\overset{\eqref{stabilityinequality}}{\leq}\int_S\left|\nabla^{ \hhh,S}\left(\adel^\beta\varphi\right)\right|^2\,\,d \sigma_\hhh\\
        &=\int_S\left|\frac{\beta}{2}\varphi\adel^{\beta-2}\nabla^{\hhh,S}|\tilde h|^2+\adel^\beta\nabla^{\hhh,S}\varphi\right|^2\,d\sigma_\hhh\\
        &=\int_S\frac{\beta^2}{4}\varphi^2\adel^{2\beta-4}|\nabla^{\hhh,S}|\tilde h|^2|^2+\beta\varphi\adel^{2\beta-2}\langle\nabla^{\hhh,S}|\tilde h|^2,\nabla^{\hhh,S}\varphi\rangle+\adel^{2\beta}|\nabla^{\hhh,S}\varphi|^2\,d\sigma_\hhh\\ 
         &=\int_S\frac{\beta^2}{4}\varphi^2\adel^{2\beta-4}|\nabla^{\hhh,S}|\tilde h|^2|^2+\frac{1}{2}\langle\nabla^{\hhh,S}\adel^{2\beta},\nabla^{\hhh,S}\varphi^2\rangle+\adel^{2\beta}|\nabla^{\hhh,S}\varphi|^2\,d\sigma_\hhh\\
         &\overset{\eqref{ibpformula}}{=}\int_S\frac{\beta^2}{4}\varphi^2\adel^{2\beta-4}|\nabla^{\hhh,S}|\tilde h|^2|^2-\frac{1}{2}\varphi^2\hat\Delta^{\hhh,S}\adel^{2\beta}+\adel^{2\beta}|\nabla^{\hhh,S}\varphi|^2\,d\sigma_\hhh\\
         &\overset{\eqref{chainrulegarofalo}}{=}\int_S\frac{\beta}{2}\left(1-\frac{\beta}{2}\right)\varphi^2\adel^{2\beta-4}|\nabla^{\hhh,S}|\tilde h|^2|^2-\frac{\beta}{4}\varphi^2\adel^{2\beta-4}\left(2\adel^2\hat\Delta^{\hhh,S}|\tilde h|^{2}\right)+\adel^{2\beta}|\nabla^{\hhh,S}\varphi|^2\,d\sigma_\hhh\\
         &\overset{\eqref{simonskatoeqnugualedue}}{\leq}\int_S\frac{\beta}{2}\left(1-\frac{\beta}{2}\right)\varphi^2\adel^{2\beta-4}|\nabla^{\hhh,S}|\tilde h|^2|^2+\adel^{2\beta}|\nabla^{\hhh,S}\varphi|^2-\beta\varphi^2\alpha^2|\tilde h|^2\adel^{2\beta-4}\fisd\,d\sigma_\hhh\\
         &\quad +\int_S-\frac{\beta}{4}\left(1+\frac{k}{2n-1}\right)\varphi^2\adel^{2\beta-4}|\nabla^{\hhh,S}|\tilde h|^2|^2+\beta q\left(\adel^{\beta}\varphi\right)^2\,d\sigma_\hhh\\
         &\overset{\eqref{stimafondamentale}}{\leq}\int_S\frac{\beta}{4}\left(1-\beta-\frac{k}{2n-1}\right)\varphi^2\adel^{2\beta-4}|\nabla^{\hhh,S}|\tilde h|^2|^2+\beta q\left(\adel^{\beta}\varphi\right)^2+\adel^{2\beta}|\nabla^{\hhh,S}\varphi|^2\,d\sigma_\hhh.\\
    \end{split}
\end{equation*}
  Therefore, factoring out, we deduce that
  \begin{equation}\label{firstssydelta}
      \frac{\beta}{4}\left(\frac{k}{2n-1}+\beta-1\right)\int_S\varphi^2\adel^{2\beta-4}|\nabla^{\hhh,S}|\tilde h|^2|^2\,d\sigma_\hhh\leq\int_S(\beta-1)q\left(\adel^{\beta}\varphi\right)^2+\adel^{2\beta}|\nabla^{\hhh,S}\varphi|^2\,d\sigma_\hhh.
  \end{equation}
  Assume first that $\beta\in\left[  1,1+\sqrt{\frac{k}{2n-1}} \right)$. In particular, $\beta-1\geq 0$. Therefore, exploiting again \eqref{stabilityinequality} and combining Cauchy-Schwarz and Young inequalities,
\begin{equation*}
    \begin{split}
        \int_S&q\left(\adel^\beta\varphi\right)^2\,d\sigma_\hhh\overset{\eqref{stabilityinequality}}{\leq}\int_S\left|\nabla^{ \hhh,S}\left(\adel^\beta\varphi\right)\right|^2\,\,d \sigma_\hhh\\
        &=\int_S\frac{\beta^2}{4}\varphi^2\adel^{2\beta-4}|\nabla^{\hhh,S}|\tilde h|^2|^2+\beta\varphi\adel^{2\beta-2}\langle\nabla^{\hhh,S}|\tilde h|^2,\nabla^{\hhh,S}\varphi\rangle+\adel^{2\beta}|\nabla^{\hhh,S}\varphi|^2\,d\sigma_\hhh\\
        &\leq\int_S\frac{\beta^2}{4}\varphi^2\adel^{2\beta-4}|\nabla^{\hhh,S}|\tilde h|^2|^2+\beta\left(\varphi\adel^{\beta-2}\left|\nabla^{\hhh,S}|\tilde h|^2\right|\right)\left(\adel^\beta\left|\nabla^{\hhh,S}\varphi\right|\right)+\adel^{2\beta}|\nabla^{\hhh,S}\varphi|^2\,d\sigma_\hhh\\
        &\leq\int_S\frac{\beta^2}{4}\varphi^2\adel^{2\beta-4}|\nabla^{\hhh,S}|\tilde h|^2|^2+\adel^{2\beta}|\nabla^{\hhh,S}\varphi|^2\,d\sigma_\hhh\\
        &\quad+\int_S\frac{\eps\beta}{4}\varphi^2\adel^{2\beta-4}\left|\nabla^{\hhh,S}|\tilde h|^2\right|^2+\frac{\beta}{\eps}\adel^{2\beta}\left|\nabla^{\hhh,S}\varphi\right|^2\,d\sigma_\hhh\\
        &=\int_S\frac{\beta}{4}(\beta+\eps)\varphi^2\adel^{2\beta-4}\left|\nabla^{\hhh,S}|\tilde h|^2\right|^2+\left(1+\frac{\beta}{\varepsilon}\right)\adel^{2\beta}\left|\nabla^{\hhh,S}\varphi\right|^2\,d\sigma_\hhh
    \end{split}
\end{equation*}
for any given $\eps>0$, so that
\begin{equation}\label{secondssydelta}
\begin{split}
      (\beta-1)\int_Sq\left(\adel^\beta\varphi\right)^2\,d\sigma_\hhh\leq \frac{\beta}{4}(\beta-1)&(\beta+\eps)\int_S\varphi^2\adel^{2\beta-4}\left|\nabla^{\hhh,S}|\tilde h|^2\right|^2\,d\sigma_\hhh\\
    &+(\beta-1)\left(1+\frac{\beta}{\varepsilon}\right)\int_S\adel^{2\beta}\left|\nabla^{\hhh,S}\varphi\right|^2\,d\sigma_\hhh.
\end{split}
\end{equation}
Combining \eqref{firstssydelta} and \eqref{secondssydelta},
\begin{equation*}
    \begin{split}
         \frac{\beta}{4}&\left(\frac{k}{2n-1}+\beta-1\right)\int_S\varphi^2\adel^{2\beta-4}|\nabla^{\hhh,S}|\tilde h|^2|^2\,d\sigma_\hhh\leq \frac{\beta}{4}(\beta-1)(\beta+\eps)\int_S\varphi^2\adel^{2\beta-4}\left|\nabla^{\hhh,S}|\tilde h|^2\right|^2\,d\sigma_\hhh\\
         &\quad+(\beta-1)\left(1+\frac{\beta}{\varepsilon}\right)\int_S\adel^{2\beta}\left|\nabla^{\hhh,S}\varphi\right|^2\,d\sigma_\hhh+\int_S\adel^{2\beta}|\nabla^{\hhh,S}\varphi|^2\,d\sigma_\hhh,
    \end{split}
\end{equation*}
   so that, factoring out and dividing by $\frac{\beta}{4}$,
\begin{equation}\label{thirdssydelta}
         P(\eps,\beta,k)\int_S\varphi^2\adel^{2\beta-4}|\nabla^{\hhh,S}|\tilde h|^2|^2\,d\sigma_\hhh \leq Q(\eps,\beta)\int_S\adel^{2\beta}\left|\nabla^{\hhh,S}\varphi\right|^2\,d\sigma_\hhh,
\end{equation}
        where
        \begin{equation*}
            P(\eps,\beta,k)=-\beta^2+(2-\eps)\beta+\frac{k}{2n-1}-1+\eps\qquad\text{and}\qquad Q(\eps,\beta)=4+\frac{4\beta-4}{\eps}.
        \end{equation*}
        Notice that 
        \begin{equation*}
            Q(\eps,\beta)\geq 4
        \end{equation*}
        for any $\eps>0$ and any $\beta\geq 1$. Moreover,
          \begin{equation*}
       P(0,\beta,k)
		\begin{cases}>0&\text{ if  }\beta\in \left[1,1+\sqrt{\frac{k}{2n-1}}\right)\\
			=0 &\text{ if }\beta=1+\sqrt{\frac{k}{2n-1}}\\
            <0 &\text{ if }  \beta>1+\sqrt{\frac{k}{2n-1}}.
		\end{cases}
   \end{equation*}
   Since the map $\eps\mapsto P(\eps,\beta,k)$ is continuous and $\beta<1+\sqrt{\frac{k}{2n-1}}$, there exists $\hat \eps=\hat\eps (k,\beta)$ such that 
   \begin{equation}\label{fourthstepssy}
       P(\hat\eps,\beta,k)>0.
   \end{equation}
   Choosing $\eps=\hat\eps$ in \eqref{thirdssydelta}, we exploit \eqref{fourthstepssy} to infer that
     \begin{equation}\label{thirdstepssyboh}
                \int_S\varphi^2\adel^{2\beta-4}|\nabla^{\hhh,S}|\tilde h|^2|^2\,d\sigma_\hhh \leq \frac{Q(\hat\eps,\beta)}{P(\hat\eps,\beta,k)}\int_S\adel^{2\beta}\left|\nabla^{\hhh,S}\varphi\right|^2\,d\sigma_\hhh.
        \end{equation}
      Finally, recalling \eqref{hpintermediatraqejvraffinata}, we combine \eqref{secondssydelta} and \eqref{thirdssydelta} to conclude that
\begin{equation*}
    \begin{split}
      \int_S\varphi^2|\tilde h|^{2\beta+2}\,d \sigma_\hhh& \overset{\eqref{segnodiq}}{\leq} \int_Sq\left(\adel^\beta\varphi\right)^2\,d\sigma_\hhh\\
      &\overset{\eqref{secondssydelta}}{\leq}\frac{\beta}{4}(\beta+\hat\eps)\int_S\varphi^2\adel^{2\beta-4}\left|\nabla^{\hhh,S}|\tilde h|^2\right|^2\,d\sigma_\hhh+\left(1+\frac{\beta}{\hat\varepsilon}\right)\int_S\adel^{2\beta}\left|\nabla^{\hhh,S}\varphi\right|^2\,d\sigma_\hhh\\
      &\overset{\eqref{thirdstepssyboh}}{\leq}\left(\frac{\beta(\beta+\hat\eps)Q(\hat\eps,\beta)}{4P(\hat\eps,\beta,k)}+1+\frac{\beta}{\hat\eps}\right)\int_S\adel^{2\beta}\left|\nabla^{\hhh,S}\varphi\right|^2\,d\sigma_\hhh.
    \end{split}
\end{equation*}
Since the constant
\begin{equation*}
    C(\beta,k)=\left(\frac{\beta(\beta+\hat\eps)Q(\hat\eps,\beta)}{4P(\hat\eps,\beta,k)}+1+\frac{\beta}{\hat\eps}\right)
\end{equation*}
is independent on $\delta>0$, the dominated convergence theorem allows to let $\delta\to 0^+$ in the previous inequality to deduce that
\begin{equation}\label{seventhstepssy}
    \int_S\varphi^2|\tilde h|^{2\beta+2}\,d\sigma_\hhh\leq C(\beta,k) \int_S|\tilde h|^{2\beta}|\nabla^{\hhh,S}\varphi|^2\,d \sigma_\hhh.
\end{equation}
In order to conclude, we exploit \eqref{seventhstepssy} replacing $\varphi$ with $\varphi^{\beta+1}$. In this way H\"older's inequality implies 
\begin{equation*}
    \begin{split}
        \int_S|\tilde h|^{2\beta+2}\varphi^{2\beta+2}\,d \sigma_\hhh &= \int_S|\tilde h|^{2\beta+2}\left(\varphi^{\beta+1}\right)^2\,d \sigma_\hhh \\
        &\overset{\eqref{seventhstepssy}}{\leq}C(\beta,k) \int_S|\tilde h|^{2\beta}|\nabla^{\hhh,S}(\varphi^{\beta+1})|^2\,d \sigma_\hhh\\
        &=(\beta+1)^2C(\beta,k) \int_S\left(|\tilde h|\varphi\right)^{2\beta}|\nabla^{\hhh,S}\varphi|^2\,d\sigma_\hhh\\
        &\leq(\beta+1)^2C(\beta,k)\left(\int_S|\tilde h|^{2\beta+2}\varphi^{2\beta+2}\,d \sigma_\hhh \right)^{\frac{\beta}{\beta+1}}\left(\int_S|\nabla^{\hhh,S}\varphi|^{2\beta+2}\,d \sigma_\hhh \right)^{\frac{1}{\beta+1}},
    \end{split}
\end{equation*}
whence
\begin{equation*}
     \int_S|\tilde h|^{2\beta+2}\varphi^{2\beta+2}\,d \sigma_\hhh \leq (\beta+1)^{2\beta+2}C(\beta,k)^{\beta+1}\int_S|\nabla^{\hhh,S}\varphi|^{2\beta+2}\,d \sigma_\hhh .
\end{equation*}
The thesis follows when $\beta\geq 1$. Finally, assume that $\beta\in\left[   \frac{2n-1-k}{2n-1},1  \right)$. 
        In this case,      
          \eqref{firstssydelta} implies that
        \begin{equation*}
           \int_S|\tilde h|^{2\beta+2}\varphi^2\,d \sigma_\hhh \leq\int_Sq\left(\adel^{\beta}\varphi\right)^2\,d\sigma_\hhh\leq \frac{1}{1-\beta}\int_S\adel^{2\beta}|\nabla^S\varphi|^2\,d S,
        \end{equation*}
        and the thesis follows as in the previous case.
    \end{proof}
    \section{The Bernstein problem}\label{abttsect}
   As an important applications of our curvature estimates, we solve the Bernstein problem in the class of hypersurfaces $S\subseteq\hh^2$ satisfying \eqref{taiwansplitting}, \eqref{graddialfasololungojeinuconk} and \eqref{hpintermediatraqejvraffinata}, providing the sub-Riemannian analogue of \cite[Theorem 2]{MR0423263}. To this aim, we assume the validity of suitable sub-Riemannian volume growth conditions,  inspired
    by the behavior of  perimeter minimizers (cf. \cite[Theorem 2.2]{psv}).
    \begin{proposition}\label{stiedivolumesezionebbbbbbbb}
        Let $E\subseteq \hh^n$ be a global perimeter minimizer with smooth, non-characteristic boundary. Then there exists a constant $c>0$ such that
        \begin{equation*}\label{ude}
            P_\hh(E,B_r(p))\leq cr^{2n+1}
        \end{equation*}
        for any $r>0$ and any $p\in\partial E$.
    \end{proposition}
       \begin{remark}
       With regard to the above volume growth condition, notice that $2n+1=Q-1$, where $Q:=2n+2$ is both the \emph{homogeneous} and the \emph{metric} dimension of $\hn$ (cf. \cite{MR3587666}).
    \end{remark}
   In order to apply \Cref{newnewmainstabilityintrostatement}, we need to ensure the validity of \eqref{stimafondamentale} for suitable values of $k$.
    To this aim, we see that 
    \begin{equation}\label{nuovosemplice}
        \fisd(p)\overset{\eqref{doveedefinitafis}}{=}(6-4k)|\tilde h_p|^2-2k\hnn(p)^2\overset{\eqref{newnewboundperelle}}{\geq}\left(6-\frac{16}{3}k\right)|\tilde h_p|^2\geq 0
    \end{equation}
 provided that $k\leq\frac{9}{8}$.
Before stating our main result, we point out that, in view of \cite{MR4103357}, when $n\geq 2$ the notion of completeness for an embedded hypersurface $S\subseteq\hh^n$ can be equivalently given by the restriction of the ambient metric on $S$ or by the intrinsic metric of $S$, provided that the latter are induced by the Euclidean, Riemannian or sub-Riemannian structure of $\hh^n$. Therefore, in the following we will talk without ambiguity of complete hypersurfaces.
    
    \begin{theorem}\label{mainmainmainmain}
        Let $S\subseteq\hh^2$ be a smooth, complete, connected, embedded, two sided non-characteristic hypersurface. Assume that $S$ is stable. Assume that $S$ verifies \eqref{taiwansplitting}, \eqref{graddialfasololungojeinuconk} and \eqref{hpintermediatraqejvraffinata}. Assume in addition that there exists $p\in S$ and a constant $c>0$ such that
        \begin{equation}\label{volumegrowth}  \lim_{r\to+\infty}\frac{\sigma_\hhh(S\cap B_r(p))}{r^{2n+1}}\leq c.
        \end{equation}
        Then $S$ is a vertical hyperplane.
    \end{theorem}
    \begin{proof}
       Fix $p\in S$ as in the statement. Let $(R_j)_j$ be a sequence of positive numbers such that $         \lim_{j\to\infty}R_j=+\infty$. In this way, up to a subsequence, we deduce from 
       \eqref{volumegrowth} that
       \begin{equation}\label{volumegrowthnonlim}
           \sigma_\hhh(S\cap B_{2R_j}(p))\leq \hat c R_j^{2n+1}
       \end{equation}
     for any $j\in\mathbb N$ and a suitable constant $\hat c>0$.  In view of \cite[Lemma 3.6]{MR1221840}, it is possible to find a positive constant $\tilde C>0$ and a sequence of non-negative  functions $(\varphi_j)_j\subseteq C_c^1(\hh^2)$ such that 
       \begin{equation}\label{cutoff}
           \varphi_j\equiv 1\text{ in }B_{R_j}(p),\qquad  \varphi_j\equiv 0\text{ in }\hh^2\setminus B_{2R_j}(p)\qquad\text{and}\qquad|\nabla^{\mathcal{H}}\varphi_j|\leq\frac{\tilde C}{R_j}.
       \end{equation}
       We wish to apply \Cref{newnewmainstabilityintrostatement} to the sequence $(\varphi_j)_j$. To this aim, since $S$ is complete, then $S\cap \supp \varphi_j$ is compact in $S$ for any $j\in\mathbb N$, whence $(\varphi_j)_j\subseteq C^1_c(S)$.
        Fix $k\in \left[0,\frac{9}{8}\right]$. Then, in view of \eqref{nuovosemplice}, \eqref{stimafondamentale} holds.
       Therefore, for any fixed $j\in\mathbb N$, we can apply \eqref{newnewmaineqstabintro}, so that
       \begin{equation*}
           \begin{split}
               \int_{S\cap B_{R_j}(p)}|\tilde h|^{2\beta+2}\,d \sigma_\hhh &\overset{\eqref{cutoff}}{\leq} \int _{S\cap B_{2R_j}(p)}|\tilde h|^{2\beta+2}\varphi^{2\beta+2}\,d \sigma_\hhh \\
               &\overset{\eqref{newnewmaineqstabintro}}{\leq}\int _{S\cap B_{2R_j}(p)}|\nabla^{\mathcal{H},S}\varphi|^{2\beta+2}\,d \sigma_\hhh \\
               &\leq\int _{S\cap B_{2R_j}(p)}|\nabla^{\mathcal{H}} \varphi|^{2\beta+2}\,d \sigma_\hhh \\
               &\overset{\eqref{cutoff}}{\leq} \left(\frac{\tilde C}{R_j}\right)^{2\beta+2} \sigma_\hhh(S\cap B_{2R_j}(p))\\
               &\overset{\eqref{volumegrowthnonlim}}{\leq}\tilde cR_j^{3-2\beta}
           \end{split}
      \end{equation*}   
       for a suitable positive constant $\tilde c$ independent of $j\in\mathbb N$. Observe that, if we could choose $\beta>\frac{3}{2}$ in the previous inequality, we could pass to the limit as $j\to\infty$ to infer that
       \begin{equation*}
           \int_S|\tilde h|^{2\beta+2}\,d \sigma_\hhh =0,
       \end{equation*}
       whence $\tilde h\equiv 0$. To this aim, it suffices to notice that $           \frac{3}{2}<1+\sqrt{\frac{k}{3}}\iff k>\frac{3}{4}. $
Therefore, any choice of $k\in\left(\frac{3}{4},\frac{9}{8}\right]$ allows to conclude that $\tilde h\equiv 0$. Finally, being $S$ non-characteristic, we can apply \cite[Theorem 1.1]{ruled} to conclude that $S$ is a vertical hyperplane.
    \end{proof}
    As a consequence of \Cref{mainmainmainmain}  and \Cref{stiedivolumesezionebbbbbbbb}, we get the following corollary.
    \begin{corollary}
        Let $E\subseteq\hh^2$ be a global perimeter minimizer. Let $\partial E$ be smooth, connected and non-characteristic. Assume that $\partial E$ verifies \eqref{taiwansplitting}, \eqref{graddialfasololungojeinuconk} and \eqref{hpintermediatraqejvraffinata}. Then $\partial E$ is a vertical hyperplane.
    \end{corollary}

\appendix
\section{Weakening \eqref{hpintermediatraqejvraffinata}}\label{appendice}
In this final Appendix we show how \eqref{hpintermediatraqejvraffinata} can be suitably refined. To this aim, we introduce a parametric form of the latter, requiring that
         \begin{equation}\label{hpintermediatraqejvraffinataapp}
        4J(\vh)\alpha+(4+\omega)\alpha^2\geq 0.
    \end{equation}
    for a fixed constant $\omega\in[0,2]$. Observe that \eqref{hpintermediatraqejvraffinataapp} is clearly implied by \eqref{hpintermediatraqejvraffinata} for any $\omega\in [0,2]$.
When $n=2$ and \eqref{taiwansplitting}, \eqref{graddialfasololungojeinuconk} and \eqref{hpintermediatraqejvraffinataapp} hold, we obtain a parametric Simons inequality of the form
    \begin{equation*}\label{bestsimonsinh222222222app}
         \frac{1}{2}\hat\Delta^{\hhh,S}|\tilde h|^2\geq|\nabla^{\hhh,S}\tilde h|^2-q|\tilde h|^2+(4-2\omega)\alpha^2|\tilde h|^2-(6-3\omega)\alpha^2\hnn^2.
    \end{equation*}
   For any $\delta\geq 0$, let $\adel$ be as in \eqref{adeldef}. As $\omega\leq 2$, then $6-3\omega\geq 0$, so that we apply \eqref{newnewboundperelle} to infer 
\begin{equation*}
    (4-2\omega)|\tilde h|^2-(6-3\omega)\hnn^2\overset{\eqref{newnewboundperelle}}{\geq}(4-2\omega)|\tilde h|^2-\frac{2}{3}(6-3\omega)|\tilde h|^2=0.
\end{equation*}
Therefore, arguing \emph{verbatim} as in \Cref{secsimonskato}, we deduce that
\begin{equation*}\label{simonscondeltadueapp}
         2\adel^2\hat\Delta^{\hhh,S}|\tilde h|^2\geq 4\adel^2|\nabla^{\hhh,S}\tilde h|^2-4q\adel^2|\tilde h|^2+4\alpha^2|\tilde h|^2\left((4-2\omega)|\tilde h|^2-(6-3\omega)\hnn^2\right)
    \end{equation*}
for any $\delta\geq 0$ and any $\omega\in [0,2]$. 
Therefore, following again the lines of \Cref{secsimonskato}, we conclude that
\begin{equation}\label{simonskatoeqnugualedueapp}
    2\adel^2\hat\Delta^{\hhh,S}|\tilde h|^2\geq \left(1+\frac{k}{2n-1}\right)|\nabla^S|\tilde h|^2|^2-4q\adel^4+4\alpha^2|\tilde h|^2\fisdapp
\end{equation}
for any $\delta\geq 0$ and any $k,\omega\in [0,2]$, where
\begin{equation*}\label{doveedefinitafisapp}
      \fisdapp(p):=(6-2\omega-4k)|\tilde h_p|^2+\left(3\omega-2k\right)\hnn(p)^2
\end{equation*}
for any $p\in S$. Exploiting \eqref{simonskatoeqnugualedueapp}, the conclusions of \Cref{newnewmainstabilityintrostatement} continue to hold provided that 
\begin{equation}\label{stimafondamentaleapp}
    \fisdapp(p)\geq 0
\end{equation}
for any $p\in S$. Moreover, in order to apply \Cref{newnewmainstabilityintrostatement} to the proof of our main theorem, we recall that we need to ensure the validity of \eqref{stimafondamentaleapp} for a suitable $k\in\left(\frac{3}{4},2\right]$.
 Let us describe our approach as follows. Let $m_\hnn\in[0,\frac{2}{3}]$ be such that
    \begin{equation}\label{mhnnpersec10app}
        \hnn(p)^2\geq m_\hnn |\tilde h_p|^2
    \end{equation}
    for any $p\in S$. Notice that the upper bound for $m_\hnn$ follows from \eqref{newnewboundperelle}, while $m_\hnn=0$ can be chosen whether no further information is available. 
    \begin{proposition}\label{lemmaprimadimaininsec10app}
        Assume that
    \begin{equation}\label{stimamhhheomegaapp}
       m_\hnn\leq\frac{2}{9}\quad\implies\quad\omega<u(m_\hnn):=\frac{3m_\hnn-6}{6m_\hnn-4}.
    \end{equation}
     Then there exists $k\in(\frac{3}{4},2]$ such that \eqref{stimafondamentaleapp} is satisfied.
    \end{proposition}
    \begin{remark}
    Notice that the function $s\mapsto u(s)$ is continuous and increasing on $\left[0,\frac{2}{9}\right]$, and moreover $u(0)=\frac{3}{2}$ and $u\left(\frac{2}{9}\right)=2$. In particular, \eqref{stimamhhheomegaapp} holds for any $\omega<\frac{3}{2}$ without any further information on $m_\hnn$, while the best choice $\omega=2$ can be made as soon as $m_\hnn>\frac{2}{9}$.
    \end{remark}
    \begin{proof}[Proof of \Cref{lemmaprimadimaininsec10app}]
        Let $k\in(\frac{3}{4},2]$. Assume first that $0\leq\omega\leq\frac{1}{2}$. Then $3\omega-2k<0$, so that
\begin{equation*}
   \fisdapp(p)=(6-2\omega-4k)|\tilde h_p|^2+\left(3\omega-2k\right)\hnn(p)^2\overset{\eqref{newnewboundperelle}}{\geq}\left(6-\frac{16}{3}k\right)|\tilde h_p|^2
\end{equation*}
for any $p\in S$, whence \eqref{stimafondamentaleapp} holds for any $k\in\left(\frac{3}{4},\frac{9}{8}\right].$ On the other hand, assume that $\omega>\frac{1}{2}$. Since $\omega>\frac{1}{2}$, we can chose $k>\frac{3}{4}$ small enough to ensure that $3\omega-2k>0$. Assume first that $m_\hnn\in\left[0,\frac{2}{9}\right]$. In this way, by \eqref{stimamhhheomegaapp}, $\omega<u(m_\hnn)$, so that
\begin{equation*}
\begin{split}
      \fisdapp(p)&=(6-2\omega-4k)|\tilde h_p|^2+\left(3\omega-2k\right)\hnn(p)^2\\
      &\overset{\eqref{mhnnpersec10app}}{\geq}(6-\omega(2-3m_\hnn)-4k-2km_\hnn)|\tilde h_p|^2\\
      &\overset{\eqref{stimamhhheomegaapp}}{>}\left(4\left(\frac{3}{4}-k\right)+2m_\hnn\left(\frac{3}{4}-k\right)\right)|\tilde h_p|^2
\end{split}
\end{equation*}
for any $p\in S$. 
As the last term in the above inequality vanishes when $k=\frac{3}{4}$, and due to the last strict inequality, the thesis follows when $m_\hnn\in\left[0,\frac{2}{9}\right]$. Finally, assume that $m_\hnn>\frac{2}{9}$.Recalling that $w\leq 2$,
\begin{equation*}
    \begin{split}
         \fisdapp(p)&=(6-2\omega-4k)|\tilde h_p|^2+\left(3\omega-2k\right)\hnn(p)^2\\
         &\overset{\eqref{mhnnpersec10app}}{\geq}\left(6-2\omega-4k+m_\hnn(3\omega-2k)\right)|\tilde h_p|^2\\
         &> \left(6-\frac{4}{3}\omega-\frac{40}{9}k\right)|\tilde h_p|^2\\
         &\geq\frac{40}{9}\left(\frac{3}{4}-k\right)|\tilde h_p|^2
    \end{split}
\end{equation*}
for any $p\in S$, whence the thesis follows as in the previous case.
    \end{proof}  
Replacing \eqref{hpintermediatraqejvraffinata} with \eqref{hpintermediatraqejvraffinataapp} allows then to provide the following refined version of \Cref{mainmainmainmain}    \begin{theorem}\label{mainmainmainmainapp}
        Let $S\subseteq\hh^2$ be a smooth, complete, connected, embedded, two sided non-characteristic hypersurface. Assume that $S$ is stable. Assume that $S$ verifies \eqref{taiwansplitting}, \eqref{graddialfasololungojeinuconk} and \eqref{hpintermediatraqejvraffinataapp}, where $\omega$ is as in \eqref{stimamhhheomegaapp}. Assume in addition that there exists $p\in S$ and a constant $c>0$ such that
        \begin{equation*}  \lim_{r\to+\infty}\frac{\sigma_\hhh(S\cap B_r(p))}{r^{2n+1}}\leq c.
        \end{equation*}
        Then $S$ is a vertical hyperplane.
    \end{theorem}
    
\bibliographystyle{abbrv}
\bibliography{biblio}
\end{document}